\documentclass[leqno]{amsart}

\pdfoutput=1

\usepackage{amsmath}
\usepackage{amsthm}
\usepackage{amsfonts}
\usepackage{amssymb}
\usepackage{enumitem}
\usepackage[pdftex,bookmarks=true]{hyperref}        
\usepackage{commath}    

\usepackage{graphicx}
\usepackage[caption=false]{subfig}   
\usepackage{tikz}             
\usetikzlibrary{matrix,arrows,decorations.pathmorphing}        
\usepackage{mathrsfs}               
\usepackage{relsize}                
\usepackage{caption}

\ifpdf
  \DeclareGraphicsExtensions{.eps,.pdf,.png,.jpg}
\else
  \DeclareGraphicsExtensions{.eps}
\fi

\newcommand{\body}{Z}                   
\newcommand{\bodyint}{\mathring{Z}}             
\newcommand{\wvec}{\mathsf{w}}
\newcommand{\centroid}{\operatorname{centroid}}
\newcommand{\vol}{\operatorname{vol}}
\newcommand{\Vol}{\operatorname{Vol}}
\newcommand{\mathapprox}[1]{\mathscr{#1}}       

\newtheorem{prop}{Proposition}[section]
\newtheorem{defn}{Definition}[section]
\newtheorem{theorem}{Theorem}[section]
\newtheorem{corollary}{Corollary}[section]
\newtheorem{lemma}{Lemma}[section]
\newtheorem*{maintheorem}{Main Theorem}

\numberwithin{theorem}{section}
\numberwithin{equation}{section}
\numberwithin{figure}{section}

\newcommand{\TheTitle}{A Centroid for Sections of a Cube in a Function Space, with application to Colorimetry}
\newcommand{\TheAuthors}{Glenn Davis}

\title[A Centroid for Sections of a Cube in a Function Space]{\TheTitle}

\author{Glenn Davis}
\email{gdavis@gluonics.com}

\ifpdf
\hypersetup{
  pdftitle={\TheTitle},
  pdfauthor={\TheAuthors}
}
\fi

\begin{document}


\date{March 13, 2020. version 4}

\begin{abstract}
The definition of the centroid in finite dimensions does not apply in a function space
because of the lack of a translation invariant measure.
Another approach, suggested by Nik Weaver, is to use a suitable collection of finite-dimensional subspaces.
For a specific collection of subspaces of $L^1[0,1]$, 
this approach is shown to be successful when the subset
is the intersection of a cube with a closed affine subspace of finite codimension.
The techniques used are the classical Laplace Transform and saddlepoint method for asymptotics.
Applications to spectral reflectance estimation in colorimetry are presented.
\end{abstract}

\maketitle


\section{Introduction}\label{sec:introduction}

Operator equations often have infinitely many solutions,
even when the solution is constrained to a feasible region.
We consider the finite-rank surjective operator equation:
\begin{equation}\label{eq:operator}
\Lambda f = y  \hspace{30pt} \text{where} \hspace{5pt}  \Lambda : L^1[0,1] ~ \twoheadrightarrow ~ \mathbb{R}^m \hspace{5pt} \text{is linear and continuous}
\end{equation}
with feasible region the unit cube $Q^\infty \subset L^\infty[0,1] \subset L^1[0,1]$.
$Q^\infty$ is the non-negative part of the unit ball in $L^\infty[0,1]$.
When the domain and range of $\Lambda$ are arbitrary Banach spaces,
a \emph{regularizer} (or \emph{penalty function} or \emph{stabilizing functional})
is often used to single out a solution, using prior knowledge about the desired solution, see \cite{schuster2012} and \cite{tarantola2005}.

The finite-rank case is a big simplification, and  allows for novel methods for singling out a solution.
An essentially bounded function $\mathbf{w} : [0,1] \to \mathbb{R}^m$ 
induces a suitable operator $\Lambda_\mathbf{w}$ defined by 
$\Lambda_\mathbf{w} f = \int_0^1 f(x) \mathbf{w}(x) \, dx$.
Conversely, a given operator $\Lambda$ is 
defined by $m$ coordinate functionals in the dual space of $L^1[0,1]$,
which is isomorphic to $L^\infty[0,1]$.
Denote the corresponding functions in $L^\infty[0,1]$ by $w_1, \ldots, w_m$,
and define $\mathbf{w}(x) := (w_1(x), \ldots, w_m(x))$.
This $\mathbf{w}$ is essentially bounded and
$\Lambda f$ is given by the previous integral.
We call $\mathbf{w}$ the \emph{responsivity} and  $w_1, \ldots, w_m$ the \emph{responsivities} of $\Lambda_\mathbf{w}$.
The right side $y$ of \ref{eq:operator} is the \emph{response}.
 
In this article $\Lambda$ is always taken as surjective (otherwise one can simply choose an isomorphism with the actual range).
It is easily shown (see \ref{thm:Gram}) that surjectivity is equivalent to the linear independence of $w_1, \ldots, w_m$,
and so we always assume this linear independence.
For a given $\tau \in \mathbb{R}^m$ note that the linear combination
$\tau_1 w_1 + \ldots + \tau_m w_m$ can be expressed more compactly as $\langle \tau , \mathbf{w} \rangle$.
We move freely back and forth between these notations, and sometimes denote it by $w_\tau$.
In this article we compute, with appropriate definitions and under certain conditions on $\Lambda$, 
a \emph{centroid} of $\Lambda^{-1}(y) \cap Q^\infty$, for suitable $y \in \mathbb{R}^m$,
and show that it is a solution of \ref{eq:operator}.
The main result is

\begin{maintheorem} 
Let $\mathbf{w} : [0,1] \to \mathbb{R}^m$ be a step function
with $w_1, \ldots, w_m$ linearly independent and 
$\Lambda_\mathbf{w} : L^1[0,1] \twoheadrightarrow \mathbb{R}^m$ the induced surjective linear operator.
If $y_0 \in \operatorname{int}(\Lambda_\mathbf{w}(Q^\infty))$ 
then for the  profile-gauge directed filtration $\mathapprox{P}$ of $L^1[0,1]$
\begin{equation*} 
\operatorname{centroid}_{\mathapprox{P}} \left( \Lambda_\mathbf{w}^{-1}(y_0) \cap Q^\infty \right) ~~=~~ \sigma( \langle \tau_0, \mathbf{w} \rangle)
\end{equation*}
where $\tau_0 \in \mathbb{R}^m$ is the unique solution to the saddlepoint equation
\begin{equation*} 
\int_0^1 \sigma( \langle \tau, \mathbf{w}(x) \rangle ) \mathbf{w}(x) \,dx  =  y_0
\end{equation*}
and where  $\sigma(t) := (\coth(t/2) - 2/t + 1)/2$.
\end{maintheorem}

The ``squashing function" $\sigma(t)$ is an analytic diffeomorphism from $\mathbb{R}$ to (0,1),
and so this centroid is as regular as the $m$ responsivities $w_1, \ldots, w_m$.
$\sigma(t)$ can be written as $(L(t/2) + 1)/2$, where $L(x) := \coth(x) - 1/x$ is the \emph{Langevin function}, see \cite{wiki:Langevin}.
Both functions have a removable singularity at 0.
The above centroid is a ``squashed" linear combination of these responsivities.
For a plot of $\sigma(t)$ see Figure \ref{fig:squashing}.

Recall that $\mathbf{w}$ is a \emph{step function} iff there is a finite partition of $[0,1]$ into subintervals
so that $\mathbf{w}$ is constant on each subinterval.
The saddlepoint equation has a unique solution by Theorem \ref{thm:diffeo}.
The profile-gauge directed filtration $\mathapprox{P}$
and $\operatorname{centroid}_{\mathapprox{P}}$ 
are defined later.

The rest of this paper is organized as follows.
In the next two Sections we discuss the centroid in finite dimensions,
and give the definition of a centroid in infinite dimensions with respect to a specific \emph{directed filtration} of $L^1[0,1]$
by finite-dimensional subspaces.
Next we describe the image $\Lambda(Q^\infty) \subset \mathbb{R}^m$,
and show that the saddlepoint equation appearing in the main result has a unique solution.
Next we bring in the directed filtration $\mathapprox{P}$ and reduce our problem to the calculation of
the centroid of a section of the $n$-dimensional cube $Q^n$.
The next Section applies the classical (bilateral) Laplace Transform to the volume of these sections of $Q^n$.
Thanks to the very special shape of $Q^n$ it is possible to find integral expressions
for both numerator and denominator of \ref{eq:centroid}.
The next Section states, but does not prove, a slightly generalized form of the classical Laplace Approximation,
to be used as $n \to \infty$.
We next prove important properties of the complex functions $P(s)$ and $K(s)$ that arise from the Laplace Transform.
The next Section brings it all together and gives the lengthy proof of the main result;
this proof uses the classical Laplace Transform complex inversion formula and the saddlepoint method.
To prepare for an application, the next Section explores reparameterization from $[0,1]$ to arbitrary $[a,b]$.

Next is a numerical application, with some plots and figures, in \ref{sec:numerical}.
The application is \emph{colorimetry}, where a function $f \in Q^\infty$ corresponds to
the spectral reflectance (or transmittance) function of a material.
For a reflectance to be physically feasible, it must be between 0 and 1.
There is then a brief treatment of the unbounded case
(including application to colorimetry),
where $Q^\infty$ is replaced by the \emph{non-negative orthant} in $L^1[0,1]$.

This is followed by discussion of open problems, and a proof of the Laplace Approximation.

In the rest of this introduction are some definitions and notations.

For $f$ a measurable function on $[0,1]$ denote the \emph{essential image} of $f$ by $\operatorname{ess.im}(f)$,
and the \emph{essential support} of $f$ by $\operatorname{supp}(f)$.
For a $C^2$ function $f : \mathbb{R}^m \to \mathbb{R}$ denote the Hessian matrix of $f$ at $x_0$ by $f''(x_0)$.
For a complex number $s$, we usually write $s = \tau + i\upsilon$ as the expansion into complex and imaginary parts,
and similarly for a complex vector $s \in \mathbb{C}^m$.

For $f \in L^1[0,1]$ and $g \in L^\infty[0,1]$, 
${\langle f, g \rangle}_{1,\infty}$ := $\int_0^1 f(x)g(x) \,dx$,
and ${\langle g, f \rangle}_{\infty,1}$ := ${\langle f, g \rangle}_{1,\infty}$.
The notation $\langle \cdot,\cdot \rangle$ means the standard inner product in a Hilbert space, 
either finite or infinite as indicated by context.
For a vector $\upsilon \in \mathbb{R}^m$,
$\abs{\upsilon} := \sqrt{\langle \upsilon,\upsilon \rangle}$ denotes the $L^2$ norm.

For a subset $A \subset X$ denote the boundary of $A$ by $\partial A$,
the closure of $A$ by $\operatorname{cl}(A)$,
and the interior of $A$ by $\operatorname{int}(A)$ or sometimes by $\mathring{A}$.
Denote the indicator function of $A$ by $\mathbf{1}_A$.
For vectors $x$ and $y$, $[x,y]$ denotes the closed segment joining them,
and $[x,y)$ denotes the half-open segment.
$I_m$ is the $m{\times}m$ identity matrix.
$B_{\delta}$ is the ball in $\mathbb{R}^m$ centered at $0$ with radius $\delta$.

Define the infinite-dimensional unit cube $Q^\infty := \{ f ~|~ \operatorname{ess.im}(f) \subseteq [0,1] \}$.
Obviously $Q^\infty \subset L^\infty$ and we give it the $L^\infty$ topology, where $Q^\infty$ is closed.
$\operatorname{int}(Q^\infty)$ means the interior considering $Q^\infty \subset L^\infty[0,1]$, 
and \emph{not} $Q^\infty$ as a subset of the two larger function spaces, where its interior is empty.
The interior $\operatorname{int}(Q^\infty) = \{ f ~|~ \operatorname{ess.im}(f) \subseteq (0,1) \}$.
The boundary $\partial Q^\infty = \{ f \in Q^\infty ~|~ 0 \in \operatorname{ess.im}(f) \text{ or }  1 \in \operatorname{ess.im}(f) \}$.
The set of ``vertices'' of $Q^\infty$ is the set $\{ f ~|~ \operatorname{ess.im}(f) \subseteq \{0,1\} \}$;
so the ``vertices'' are the indicator functions $\mathbf{1}_A$ where $A \subseteq [0,1]$ and $A$ is measurable.
A linear functional on $Q^n$ takes its maximum at a vertex, and the same is true for $Q^\infty$.
The mapping $f \mapsto 1-f$ is the \emph{standard involution} of $Q^\infty$ which has a unique fixed point
$f_c(x) \equiv 1/2$.
$Q^\infty$ is clearly symmetric about $f_c$;
later we show in Section \ref{sec:nets}, with an appropriate definition, 
that $f_c$ is the centroid of $Q^\infty$.
If $Z := \Lambda(Q^\infty)$, the standard involution of $Q^\infty$ pushes forward to an involution of $Z$,
with unique fixed point $\Lambda(f_c)$.

Let $y \in \Lambda( \operatorname{int}(Q^\infty) )$.
Then $\Lambda^{-1}(y)$ is an affine subspace with codimension $m$.
The cube section $\Lambda^{-1}(y) \cap  Q^\infty$ is non-empty and we want to investigate
whether its centroid can be defined in a reasonable way. 
The inclusion and mapping situation is
\[
Q^\infty \cap \Lambda^{-1}(y) ~~\subset~~ Q^\infty ~~ \subset ~~  L^\infty[0,1] ~~ \subset L^2[0,1] ~~ \subset L^1[0,1]  ~~  {\overset{\Lambda}\twoheadrightarrow}  ~~ \mathbb{R}^m
\]
The inclusions and the linear map $\Lambda$ are continuous.
The cube $Q^\infty$ and the cube section $Q^\infty \cap \Lambda^{-1}(y)$ are closed and bounded in all three topologies.

\section{Centroid in Finite Dimensions}\label{sec:centroid}

Let $A \subset \mathbb{R}^k$ be bounded with \emph{non-empty interior}. 
Define
\begin{equation}\label{eq:centroid}
\centroid(A) :=  \frac{\int_A x \, dx}{\int_A 1 \, dx} = \frac{\int_A I \, d \mu}{\mu(A)}
\end{equation}
where $I$ is the identity function on $\mathbb{R}^k$, and $\mu$ is standard Lebesgue measure.
Since $A$ is bounded and has an interior, $\mu(A)$ is finite and positive.
The numerator of \ref{eq:centroid} is called the \emph{moment} of $A$.
Because isomorphisms of $\mathbb{R}^k$ simply scale Lebesgue measure (by the determinant of the associated matrix),
and since $\mu$ is translation invariant, it follows that
\begin{prop}\label{prop:finite}
If $A \subset \mathbb{R}^k$ has non-empty interior, then
\begin{enumerate}[label=(\alph*)]
\item $L( \centroid (A) ) = \centroid( L(A) ) ~ $where $L : \mathbb{R}^k \to \mathbb{R}^k$ is a linear isomorphism
\item $\centroid(A+x) = \centroid(A)+x ~~~$ where  $x \in \mathbb{R}^k$
\end{enumerate}
\end{prop}
$A$ is called \emph{symmetric} iff $-A = A$.
In this case, because of cancelation, the  moment vanishes and so does the centroid, i.e. $\centroid(A)=0$.
$A$ is called \emph{symmetric about x} iff $-A = A - 2x$, which implies $\centroid(A)=x$.

Assume now that $A$ is closed,
and let $H_1^+ := \{ (x_1,\ldots,x_k) ~|~ x_1 \ge 0 \}$ be the closed halfspace.
If $A \subset H_1^+$ then the $x_1$ coordinate of the moment is clearly positive, 
and therefore $\centroid(A) \subset \operatorname{int}(H_1^+)$.
By rotation and translation, the same is true with $H_1^+$ replaced by \emph{any} halfspace.
Since a closed convex set $A$ is the intersection of all the halfspaces that contain it (\cite{leonard2015} p. 194),
it follows that $\centroid(A) \in A$.
But if $\centroid(A) \in \partial A$ then there is a supporting halfspace $H$ at $\centroid(A)$
and by above argument $\centroid(A) \in \operatorname{int}(H)$ which is impossible.
Therefore $\centroid(A) \in \operatorname{int}(A)$.

There is a probabilistic interpretation.
If $X$ is a random variable uniformly distributed in $A$, then $\centroid(A)$ is the expected value of $X$.
So if the only information one has about a point $x$ is that $x \in A$,
then $\centroid(A)$ is a ``good" estimate for $x$.

Now let $F$ be an $k$-dimensional affine subspace (or \emph{flat}) of some larger space, and let $A \subset F$.
Pick an affine isomorphism $\lambda : \mathbb{R}^k \leftrightarrow F$ and define
$\centroid(A) := \lambda( \centroid( \lambda^{-1}(A) ) )$.
Alternatively, one can define $\centroid(A)$ using \ref{eq:centroid}
with $\mu$ replaced by the pushforward measure $\lambda_*(\mu)$.
In either case, by \ref{prop:finite}, $\centroid(A)$ does not depend on the $\lambda$ selected.

The our case, $A = F \cap Q^n$ where $Q^n := [0,1]^n$ is the unit $n$-cube,
and $F \subset \mathbb{R}^n$ intersects the interior of $Q^n$.
$A$ is an $k$-dimensional polyhedron.
Because of the very special shape of $Q^n$ it is possible to find integral expressions
for both numerator and denominator of \ref{eq:centroid}, see Section \ref{sec:Volumes}.

Generalizing the centroid \ref{eq:centroid} to an infinite-dimensional real Banach space $E$
presents an immediate problem: 
there is no non-trival locally finite translation-invariant measure on $E$, see \cite{wiki:IDLM}.
Instead, we use a method suggested by Nik Weaver in \cite{Weaver2012} that uses a directed set of finite-dimensional subspaces of $E$, 
see Section \ref{sec:nets}.

\section{Centroid from a Net}\label{sec:nets}

In this section we define the nets used to define centroids of infinite-dimensional subsets of Banach spaces.
The following definition was inspired by Nik Weaver in \cite{Weaver2012}.

\begin{defn}\label{defn:directedfiltration}
For a Banach space $E$, a \emph{directed filtration} $\mathapprox{D}$ of $E$ is a collection of finite-dimensional subspaces of $E$
which form a directed set when ordered by inclusion,
and whose union is dense in $E$.
\end{defn}
The directed set condition means that if subspaces $V$ and $V'$ are in $\mathapprox{D}$
then there is a bigger subspace in $\mathapprox{D}$ that contains them both.
The bigger one is not necessarily the sum $V + V'$.
It is not necessary for a directed filtration to be closed under sums of subspaces.

For our Banach space $L^1[0,1]$ we only consider three directed filtrations:
\begin{itemize}[label=\tiny\textbullet]
\item{the collection of \emph{all} finite-dimensional subspaces, denoted by $\mathapprox{F}$ for `full'.}
\item{
the collection of all $V_n$ for $n \in \mathbb{N}$ 
where $V_n$ has the basis $\{\mathbf{1}_{[(j-1)/n,j/n]}, ~ j=1 \ldots n \}$.
Equivalently, $V_n$ is the space of all step functions whose jumps are at multiples of $1/n$.
We have $V_n \subseteq V_m$ iff $n$ divides $m$.
This directed filtration is denoted by $\mathapprox{P}$, for the ``profile-gauge" (a common woodworking tool).
}
\item{a filtration derived from $\mathapprox{P}$ from a composition operator, see equation \ref{eqn:repapp}}
\end{itemize}

\emph{Remark}.
The closest concept we could find in the literature is in \cite{Browder1966},
where a Banach space $E$ has \emph{property $(\pi)_1$} iff it has some directed filtration 
\emph{and} for each $V_\alpha$ in the filtration, there is a projection $P_\alpha : E \to V_\alpha$ with norm 1.
Property $(\pi)_1$ is repeated in \cite{cioranescu1990} where our filtration $\mathapprox{P}$ is given on page 83.
Regarding the projections $P_\alpha$, we define and use these projections in section \ref{sec:reduction},
but they do not seem necessary for defining a centroid, so they are not part of our definition.

Let $\mu_V$ denote Lebesgue measure on $V \in \mathapprox{D}$.
Suppose a bounded infinite-dimensional subset $A \subset E$ is given and satisfies this compatibility requirement with $\mathapprox{D}$:
\begin{equation}\label{req:AA}
\text{if } W \in \mathapprox{D} \text{ then there is a } V \in \mathapprox{D} \text{ with } W \subseteq V \text{ and }  \mu_V(V \cap A) > 0
\end{equation}
Define a mapping from $\mathapprox{D}$ to $E$ by $V \mapsto \centroid(V \cap A)$, and obtain a \emph{net}.
If necessary pass to the cofinal subset of $\mathapprox{D}$
consisting of all $V{\in}\mathapprox{D}$ with $\mu_V(V \cap A)>0$, so $\centroid(V \cap A)$ is defined.
For background on nets and their convergence see \cite{hu1964}, Section III.3.

\begin{defn}\label{defn:centroid}
If the net $V \mapsto \centroid(V{\cap}A)$ for $V{\in}\mathapprox{D}$ converges in $L^1[0,1]$, then define
\begin{equation*}
\centroid_\mathapprox{D}(A) := \lim_{V \in \mathapprox{D}} \centroid(V \cap A)
\end{equation*}
\end{defn}
If the net converges we call this the \emph{centroid of $A$ with respect to $\mathapprox{D}$ }.
Otherwise this centroid is undefined, see \cite{Weaver2012} for an example.

The next proposition follows directly from \ref{prop:finite} and the definitions
\begin{prop}\label{prop:infinite}
Let $E$ be an infinite-dimensional Banach space with directed filtration $\mathapprox{D}$, and let $F$ be isomorphic to $E$.
Let $A \subset E$ with well-defined $\centroid_\mathapprox{D}(A)$.  
Then
\begin{enumerate}
\item if $L : E \to F$ is a linear isomorphism then $L(\mathapprox{D})$ is a directed filtration of $F$, and $L( \centroid_\mathapprox{D} (A) ) = \centroid_{L(\mathapprox{D})}( L(A) )$
\item if $L : E \to E$ is a linear isomorphism that leaves $\mathapprox{D}$ invariant, then \\$L( \centroid_\mathapprox{D} (A) ) = \centroid_\mathapprox{D}( L(A) ) ~~~$
\item if $x \in \bigcup_{V \in \mathapprox{D}}V$, then $~\centroid_\mathapprox{D}(A+x) = \centroid_\mathapprox{D}(A)+x ~~~$
\end{enumerate}
\end{prop}

Consider any $E$ and its full filtration $\mathapprox{F}$.
Suppose $A \subset E$ and $A$ is bounded and symmetric (i.e. $-A = A$);
then $A \cap V$ is symmetric for any $V \in \mathapprox{F}$ and so
$\centroid(A \cap V) = 0$ and the net of centroids converges trivially to $0$.
Therefore by part $(c)$ of the previous proposition, $\centroid_\mathapprox{F}(A+x) = x$ for any $x$.
It follows that the centroid of any ball is at its center (with respect to $\mathapprox{F}$).
Consider the centroid of the convex hull of two balls.
A short calculation shows that it is the center of the larger ball when the radii are unequal,
and is the midpoint of the segment joining their centers when the radii are equal.
Thus the centroid is not a continuous function of the radii - not a very satisfactory situation.
As another example of bad behavior, similar to the half-ball example in \cite{Weaver2012},
let $B$ be the unit ball in $L^2[0,1]$ and $u \in \partial B$.
Then $\centroid_\mathapprox{F}(B \cap u^\perp) = 0$ since the set is symmetric.
Let $C := \operatorname{cxhull} (u \cup (B \cap u^\perp) )$, which is a cone over $B \cap u^\perp$.
Then it can be shown that $\centroid_\mathapprox{F}(C){=}0{\in}\partial C$.
In finite dimensions, the centroid must be in the \emph{interior} of a convex body,
but that is false here.
Adding the cone over $B \cap u^\perp$ does not shift its centroid.

Consider a homeomorphism $\varphi$ of [0,1].
The composition operator $C_\varphi(f) := f \circ \varphi$ is an isomorphism of $L^1[0,1]$.
We want a description of the subspaces in the directed filtration $C_\varphi(\mathapprox{P})$.
If $V_n$ is a subspace of $\mathapprox{P}$, then it is straightforward to show that
$C_\varphi(V_n)$ is the subspace of all step functions with jumps at $\varphi^{-1}(j/n)$, for $j=1,...,(n-1)$.

Consider the main subject of this paper: $A := Q^\infty \cap \Lambda_\mathbf{w}^{-1}(y) \subset L^1[0,1]$.
For the directed filtration $\mathapprox{P}$ 
and when $\mathbf{w}$ is a step function,
we will see in Section \ref{sec:mainresult} that $\centroid_\mathapprox{P}(A)$ is defined and can be calculated.
Requirement \ref{req:AA} is shown in Theorem \ref{thm:union}.
We will also give in Section \ref{sec:numerical}  an example of two directed filtrations of $L^1[0,1]$ 
that yield two different centroids of $A$.

\section{Convexity}\label{sec:convexity}

In this section we collect a few technical facts
about convex sets in Banach spaces that will be used later for $Q^\infty$ and $Q^n$.

\begin{theorem}\label{thm:segment}
\emph{[line segment principle]}
Let $A$ be convex set in a Banach space $X$.
Let $a_0 \in \operatorname{int}(A)$ and $x_1 \in \operatorname{cl}(A)$.
Then all points in the segment $[a_0,x_1)$ are interior points of $A$.
\end{theorem}
\begin{proof}
See \cite{rockafellar2009variational}, p 58.
\end{proof}

\begin{theorem}\label{thm:interior}
Let $T : X \to Y$ be a linear, continuous, and surjective map of Banach spaces $X$ and $Y$.
Let $A \subset X$ be a convex set with non-empty interior.
Then $T(A)$ is convex with non-empty interior, and
$\operatorname{int}(T(A)) = T(\operatorname{int}(A))$.
\end{theorem}
\begin{proof}
That $T(A)$ is convex is very easy.
By the open mapping theorem (\cite{ConwayFA} page 93), 
$T( \operatorname{int}(A) )$ is open, and it is certainly a subset of $T(A)$, therefore 
$T( \operatorname{int}(A) ) \subseteq \operatorname{int}(T(A))$.
Conversely, suppose $y \in \operatorname{int}(T(A))$; we must show $y=T(a)$ for some $a \in \operatorname{int}(A)$.
Let $a_0 \in \operatorname{int}(A)$;
then again by the open mapping theorem, $T(a_0)$ is an interior point of $T(A)$.
If $T(a_0)=y$ we are done.
Otherwise consider the non-trivial segment $[T(a_0),y]$.
Since $y$ is an interior point the segment can be extended slightly to $[T(a_0),y_1]$
where $y_1 \in T(A)$, but $y_1 \ne y$.
Pick $a_1 \in A$ so $T(a_1)=y_1$; trivially $a_1 \in \operatorname{cl}(A)$.
Apply \ref{thm:segment} to $a_0$ and $a_1$ to conclude
that every point in $[a_0,a_1)$ is an interior point of $A$.
Since $T$ maps the segment $[a_0,a_1)$ to $[T(a_0),y_1)$,
and since the given $y \in [T(a_0),y_1)$,
one of those interior points in $[a_0,a_1)$ maps to $y$.
\end{proof}

Note that in the above we have made no assumptions about whether points in $\partial A$ belong to $A$ or not.
$A$ might be open or closed or neither.
The next fact adds the assumptions that the convex set is closed and bounded, 
but drops the assumption that the set has a non-empty interior.

\begin{theorem}\label{thm:compact}
Let $C$ be a closed, bounded, and convex set in a reflexive Banach space $X$ (e.g. a Hilbert space).
Let $T : X \to \mathbb{R}^m$ be a continuous linear transformation.
Then $T(C)$ is compact and convex in $\mathbb{R}^m$.
\end{theorem}

\begin{proof}\label{proof:compact}
Since $X$ is a Banach space, and $C$ is closed and convex, $C$ is weakly closed;
this is Mazur's Theorem, in \cite{LangReal} page 85.  
Since $X$ is reflexive, closed balls are weakly compact;
this is Kakutani's Theorem, in \cite{ConwayFA} page 135.  
Since $C$ is bounded it is contained in a weakly compact ball, and since $C$ is weakly closed, $C$ is weakly compact. 
Since $T$ is continuous in the strong topologies, it is continuous in the weak topologies; in \cite{ConwayFA} page 171.
So $T(C)$ is weakly compact in  $\mathbb{R}^m$,
and since the weak and strong topologies on  $\mathbb{R}^m$ are the same,  $T(C)$ is compact in  $\mathbb{R}^m$.
That $T(C)$ is convex is very easy.
\end{proof}

\begin{defn}
A \emph{convex body} in $\mathbb{R}^m$ is a compact convex subset with non-empty interior.
\end{defn}

In the next section we show that $\Lambda(Q^\infty)$ is a convex body.

\section{Image of the Cube : $\Lambda(Q^\infty)$}\label{sec:image}

Let $w_j \in L^\infty[0,1], j=1,\ldots,m$.
These $m$ functions are the coordinate functions of an essentially bounded function $\mathbf{w} : [0,1] \to \mathbb{R}^m$,
and they induce a continuous linear function $\Lambda_\mathbf{w}: L^1[0,1] \to \mathbb{R}^m$ given by
\[
\Lambda_\mathbf{w}(f) := \int_0^1 f(x) \mathbf{w}(x) \, dx  ~~=~~ (\langle f,w_1\rangle_{1,\infty}, \ldots , \langle f,w_m\rangle_{1,\infty}) ~~ \text{for} ~ f \in L^1[0,1]
\]
Conversely, 
every continuous linear operator $L^1[0,1] \to \mathbb{R}^m$ is $\Lambda_\mathbf{w}$ for a unique $\mathbf{w}$.
This is because the dual of $L^1[0,1]$ is isomorphic to $L^\infty[0,1]$ by \cite{LangReal} p. 292.

\begin{theorem}\label{thm:Gram}
For functions $w_j$ as above, the following are equivalent.
\begin{enumerate}[label=(\alph*)]
\item the set \{$w_j$\}, $j=1,\ldots,m$ is linearly independent 
\item the Gram matrix $G_{ij} = \langle w_i, w_j \rangle$ is positive-definite
\item the restriction of $\Lambda_\mathbf{w}$ to $L^2[0,1]$ is surjective
\item the restriction of $\Lambda_\mathbf{w}$ to $L^\infty[0,1]$ is surjective
\item $\Lambda_\mathbf{w}$  is surjective
\end{enumerate}
\end{theorem}
\begin{proof}
The equivalence of $(a),(b),(c)$ is straightforward, e.g. see \cite{Natanson} vol I  page 193. 
Note that all the action is in the Hilbert space $L^2[0,1]$.
The implications $(d) \Rightarrow (c) \Rightarrow (e)$ are trivial.
And this leaves $(e) \Rightarrow (d)$.
To show this we use the denseness of $L^\infty \subset L^1$.
Suppose $\Lambda_\mathbf{w}$ is surjective, but the restriction is \emph{not},
and let $y_0 \in \mathbb{R}^m$ but not in $\Lambda_\mathbf{w}(L^\infty)$.
$\Lambda_\mathbf{w}(L^\infty)$ is closed in $\mathbb{R}^m$,
so let $U$ be a neighborhood of $y_0$ that does not intersect $\Lambda_\mathbf{w}(L^\infty)$.
Pick $x_0 \in L^1$ so $\Lambda_\mathbf{w}(x_0) = y_0$ and a sequence $u_n \in L^\infty$ with $u_n \to x_0$.
Then $\Lambda_\mathbf{w}(u_n) \to \Lambda_\mathbf{w}(x_0) = y_0$.
For $n$ large enough $\Lambda_\mathbf{w}(u_n) \in U$, and this is a contradiction.
\end{proof}

\medskip
\begin{center}
\underline{From now on we assume that $(a),\ldots,(e)$ in the previous theorem are all true.}
\end{center}
\bigskip

In this section we examine the image $\Lambda_\mathbf{w}(Q^\infty)$.
For simplicity we abbreviate:
\begin{equation*}
\body := \Lambda_\mathbf{w}(Q^\infty)  \hspace{10pt}  \text{and}  \hspace{10pt}  \bodyint := \operatorname{int}(\body)
\end{equation*}

\emph{Remarks}.
$\body$ is a \emph{zonoid}, i.e. the image of an atom-free vector measure.
From $\Lambda_\mathbf{w}$ a suitable measure can be constructed;
see \cite{Akemann2001} page 105 and \cite{Bolker1969} page 338,
and for the reverse direction see \cite{rudinFA} p. 120.
If $\mathbf{w}$ is a step function as in the Main Theorem (or any function with finite range),
then $\body$ is a \emph{zonotope}; i.e. a zonoid that is also a polytope,
or equivalently the linear image of the unit cube $Q^N$ for some $N$,
see \cite{Bolker1969} page 330.
$\body$ could have been written $\body_\mathbf{w}$, 
but the dependence on $\mathbf{w}$ is suppressed for simplicity;
the function $\mathbf{w}$ is fixed in this section, up to the very end.

\begin{theorem}\label{thm:body}
The set $\body$ is a convex body in $\mathbb{R}^m$, and $\bodyint = \Lambda_\mathbf{w}( \mathrm{int}(Q^\infty) )$.
\end{theorem}
\begin{proof}
The statement about $\bodyint$ follows from \ref{thm:interior},
where $T$ is the restriction of $\Lambda_\mathbf{w}$ to $L^\infty[0,1]$, where $Q^\infty$ has non-empty interior.
Now in \ref{thm:compact} let $T$ be the restriction of $\Lambda_\mathbf{w}$ to
the Hilbert space $H = L^2[0,1]$.  
Since the cube $Q^\infty$ is closed in $H$, we conclude that $\body$ is compact.
And since $\bodyint$ is non-empty, $\body$ is a convex body (i.e. convex and compact with interior).
\end{proof}

The statement about the interiors in the previous theorem is important.
It states that if $z \in \bodyint$, then $\Lambda_\mathbf{w}f{=}z$ 
has a solution $f \in \operatorname{int}(Q^\infty)$.
This means that the set of \emph{all} solutions  $Q^\infty \cap  \Lambda_\mathbf{w}^{-1}(z)$
is a `big polyhedron', with codimension $m$ in $L^\infty[0,1]$.
Our goal is to compute a centroid of this solution set.
In this article, we do not consider the case where $z \in \partial \body$.
Continuing this emphasis on $\bodyint$, 
the next goal is to construct a diffeomorphism from $\mathbb{R}^m$ to $\bodyint$
that is used later in Section \ref{sec:mainresult}.

For $\mathbf{u} \in S^{m-1}$, the \emph{support function} $\psi_\body$ of $\body$ is defined by
\begin{equation*}
\psi_\body(\mathbf{u}) := \sup_{z \in \body} \langle \mathbf{u} , z \rangle
\end{equation*}
The \emph{support hyperplane} $H_\mathbf{u}(\body)$ and  \emph{face} $F_\mathbf{u}(\body)$ are defined by
\begin{equation*}
H_\mathbf{u}(\body) := \{ x ~|~ \langle \mathbf{u} , x \rangle = \psi_\body(\mathbf{u}) \}
\hspace{20pt} \text{and} \hspace{20pt}
F_\mathbf{u}(\body) := H_\mathbf{u}(\body) \cap \body
\end{equation*}
Since $\body$ is compact, the sup is taken at some point in $\body$, so $F_\mathbf{u}(\body)$ is non-empty,
and it is also convex, see \cite{Bolker1969}.
$\psi_\body$ is  Lipschitz continuous with constant $\sup_{z \in \body} \norm{z}$, see \cite{rockafellar2015convex}.
For $\mathbf{t} \in \mathbb{R}^m$ define 
$w_\mathbf{t}(x) := t_1 w_1(x) + \cdots + t_m w_m(x) = \langle \mathbf{t} , \mathbf{w}(x) \rangle$.

\begin{theorem}\label{thm:maximizing} $\text{}$
Let $f \in Q^\infty$ and $\mathbf{u} \in S^{m-1}$ be fixed.
\begin{enumerate}[label=(\alph*)]
\item $\langle \mathbf{u}, \Lambda_\mathbf{w}(f) \rangle = \psi_\body(\mathbf{u})$  
iff  $f(x){=}1$ when $w_\mathbf{u}(x)>0$, and $f(x){=}0$ when $w_\mathbf{u}(x)<0$
\item if $z \in \body$ and $\langle \mathbf{u} , z \rangle = \psi_\body(\mathbf{u})$, then  $z \in \partial \body$
\item if $f(x){=}1$ when $w_\mathbf{u}(x)>0$ and $f(x){=}0$ when $w_\mathbf{u}(x)<0$,
then $\Lambda_\mathbf{w}(f) \in \partial \body$
\end{enumerate}

\end{theorem}

\begin{proof}
For $(a)$, rewrite $\psi_\body(\mathbf{u})$ as
\begin{align*}
\psi_\body(\mathbf{u}) ~&:=~ \sup_{z \in \body} \langle \mathbf{u}, z \rangle  ~=~ \sup_{z \in \Lambda_\mathbf{w}(Q^\infty)} \langle \mathbf{u}, z \rangle  ~=~ \sup_{f \in Q^\infty} \langle \mathbf{u}, \Lambda_\mathbf{w}(f) \rangle \\
~&=~ \sup_{f \in Q^\infty} \langle \mathbf{u} , (\langle f,w_1 \rangle, \ldots, \langle f,w_m \rangle ) \rangle  ~=~ \sup_{f \in Q^\infty} \langle f , w_\mathbf{u}  \rangle
\end{align*}
When $\psi_\body(\mathbf{u})$ is expressed this way, in terms of $f$, part $(a)$ is clear.

For $(b)$, if $z \in \bodyint$ then there is a point $z' \in \body$ in an open ball around $z$ 
where $\langle \mathbf{u} , z' \rangle > \langle \mathbf{u} , z \rangle = \psi_\body(\mathbf{u})$.
This contradicts the definition of $\psi_\body$ and so $z \in \partial \body$.

Part $(c)$ follows at once from parts $(a)$ and $(b)$.
\end{proof}

\emph{Remark}.
Note that if the set $w_\mathbf{u}^{-1}(0)$ has measure 0, 
then the maximizing $f$ is unique in $L^1[0,1]$ and $F_\mathbf{u}(\body)$ is a single point.
If its measure is positive, one can change $f$ arbitrarily on $w_\mathbf{u}^{-1}(0)$ 
without changing the value of 
$\langle f , w_\mathbf{u}  \rangle = \langle \mathbf{u} , \Lambda_\mathbf{w}(f) \rangle$,
and the face $F_\mathbf{u}(\body)$ can be non-trivial in general.

\begin{defn}\label{defn:squashing}
A $C^1$ function $\sigma : \mathbb{R} \to [0,1]$ is a \emph{squashing function} iff
\begin{enumerate}[label=(\alph*)]
\item $\sigma'(t) > 0$ for all $t$
\item  $\lim_{t \to -\infty} = 0$ and $\lim_{t \to \infty} = 1$
\end{enumerate}
\end{defn}
For a $f \in L^\infty[0,1]$, the composition $\sigma \circ f$ is in the interior of $Q^\infty$
and is a `squashed' version of $f$.
A standard example of a squashing function is $\sigma(t) := (\tanh(t)+1)/2$,
but we will soon see that a different one is more important for us.

Define a $C^1$ function $G_\sigma : \mathbb{R}^m \to \bodyint$ by the formula
\begin{equation}\label{eqn:Gsigma}
G_\sigma( \mathbf{t} ) ~=~ G_\sigma( t_1, \ldots , t_m ) ~:=~ \Lambda_\mathbf{w}( \sigma ( t_1 w_1 + \cdots + t_m w_m ) ) 
= \int_0^1 \sigma ( \langle \mathbf{t},\mathbf{w}(x) \rangle ) \mathbf{w}(x) \, dx
\end{equation}

We will show that $G_\sigma : \mathbb{R}^m \to \bodyint$ is a diffeomorphism in a number of steps.

\begin{lemma}\label{lemma:PD}
The derivative $\mathbf{D}G_\sigma$ is positive-definite everywhere.
\end{lemma}

\begin{proof}\label{proof:PD}
A short calculation shows that the $m \times m$ Jacobian of $G_\sigma$ is given by
\begin{equation}\label{eqn:DG}
(\mathbf{D}G_\sigma)_{ij} ~~=~~ \int_0^1 w_i(x) \sigma'(  t_1 w_1(x) + \cdots + t_m w_m(x) ) w_j(x) \, dx
\end{equation}
for $i=1,\ldots,m$ and $j=1,\ldots,m$.
Since $\sigma'(t)>0$ this can be written
\[
(\mathbf{D}G_\sigma)_{ij} ~=~ \langle \alpha w_i, \alpha w_j \rangle  ~~~~  \text{where} ~~ \alpha(x) = \sqrt{ \sigma'(  t_1 w_1(x) + \cdots + t_m w_m(x) ) } ~~> 0
\]
This is the Gram matrix of the set of functions $\{\alpha w_j\}$ which is clearly linearly independent,
since \{$w_j$\} is.
Thus by \ref{thm:Gram} $\mathbf{D}G_\sigma(t_1, \ldots ,t_m)$ is positive-definite.
\end{proof}

\emph{Remark}.
From the conclusion of this lemma, one might suspect that $G_\sigma$ is the gradient
of a convex real-valued function, and that $\mathbf{D}G_\sigma$ is its Hessian.
This suspicion is correct.
If $K(t) := \int_0^t \sigma(s) \, ds$ and
$h(\mathbf{t}) := \int_0^1 K( \langle \mathbf{t},\mathbf{w}(x) \rangle ) \, dx$
then $G_\sigma$ is the gradient of $h$.
The function $K(t)$ naturally appears later in Section \ref{sec:PK},
and $h(\mathbf{t})$ (in a slightly modified form) appears later in Section \ref{sec:mainresult}.

\begin{lemma}\label{lemma:increasing}
For fixed $x,h \in \mathbb{R}^m$ with $h \ne 0$,
the function  $s \mapsto \langle G_\sigma(x + s h), h \rangle$ is strictly increasing in $s$.
\end{lemma}

\begin{proof} 
The derivative of this function is
$s \mapsto \langle \mathbf{D}G_\sigma (x + sh) h, h \rangle$,
which is positive by \ref{lemma:PD}.
\end{proof}

\begin{theorem}\label{thm:injective}
The function $G_\sigma : \mathbb{R}^m \to \bodyint$ is injective.
\end{theorem}

\begin{proof} 
Let $s{=}0$ and $s{=}1$ in the previous lemma to get:
\begin{align*}
\langle  G_\sigma (x) , h \rangle ~~&<~~  \langle  G_\sigma (x + h) , h \rangle \\
0 ~~=~~ \langle  G_\sigma (x) - G_\sigma (x), h \rangle ~~&<~~  \langle  G_\sigma (x + h) - G_\sigma (x), h \rangle  
\end{align*}
So $G_\sigma (x + h) \ne G_\sigma (x)$ whenever $h \ne 0$.
\end{proof}

\begin{lemma}\label{lemma:toboundary}
Let $\mathbf{u} \in S^{m-1}$ be fixed. Then as $s \to \infty$
\begin{enumerate}[label=(\alph*)]
\item $G_\sigma(s \mathbf{u}) \to$ a point in $\partial \body$  (in the $L^1$ topology)
\item $\langle G_\sigma(s \mathbf{u}), \mathbf{u} \rangle \to \psi_\body( \mathbf{u} )$, 
and $\langle G_\sigma(s \mathbf{u}), \mathbf{u} \rangle$ is strictly increasing with $s$
\end{enumerate}
\end{lemma}

\begin{proof}
From property $(b)$ of $\sigma$ we have
$\lim_{s \to \infty} (\sigma \circ w_{s\mathbf{u}}) = f_\mathbf{u}$ pointwise, where
\[ 
f_\mathbf{u}(x) :=
\begin{cases}    
0 & \langle \mathbf{u} , \mathbf{w}(x) \rangle < 0  \\
1 & \langle \mathbf{u} , \mathbf{w}(x) \rangle > 0  \\
\sigma(0) & \langle \mathbf{u} , \mathbf{w}(x) \rangle = 0
\end{cases}
\]
This is by the defining properties of $\sigma$.
The family is bounded between 0 and 1, and so by the Lebesgue Dominated Convergence theorem, \cite{LangReal} p. 249, 
we have in $\mathbb{R}^m$:
\begin{equation*}
\lim_{s \to \infty} G_\sigma( s\mathbf{u} ) ~=~ \lim_{s \to \infty} \Lambda_\mathbf{w} (\sigma \circ w_{s\mathbf{u}}) ~=~   \Lambda_\mathbf{w}(f_\mathbf{u}) ~~\in~~ \partial{\body}
\end{equation*}
The final statement about $\partial{\body}$ is part $(c)$ of Theorem \ref{thm:maximizing}.
This shows part $(a)$.

For part $(b)$ the convergence follows from part $(a)$, the continuity of the inner product,
and part $(a)$ of Theorem \ref{thm:maximizing}.
The strictly increasing claim follows from Lemma \ref{lemma:increasing}.
\end{proof}

Now change the context of $\mathbf{u}$ from fixed to variable.

\begin{lemma}\label{lemma:uniform}
The sequence of functions $g_n : S^{m-1} \to \mathbb{R}$ given by
$g_n( \mathbf{u} ) :=  \langle G_\sigma(n \mathbf{u}), \mathbf{u} \rangle$, 
increases with $n$ and converges uniformly to $\psi_\body$.
\end{lemma}
\begin{proof}
By the previous lemma, the sequence converges pointwise,
and it increases to a continuous function.
Since $S^{m-1}$ is compact, the convergence is uniform by Dini's Theorem,
\cite{Stromberg2015} page 143.
\end{proof}

The next three lemmas are variants of each other,
with slightly different hypotheses and the same conclusion.
Any one of them can be used in Theorem \ref{thm:diffeo}.
Denote the closed unit ball in $\mathbb{R}^m$ by $B^m$.

\begin{lemma}\label{lemma:PH}
Let $g: B^m \to \mathbb{R}^m$ be a $C^0$ map.
Suppose that $\langle g(\mathbf{u}) , \mathbf{u} \rangle > 0$ for all $\mathbf{u} \in S^{m-1}$.
Then  $g(\mathbf{x}){=}0$ has a root $\mathbf{x}' \in \operatorname{int}(B^m)$.
\end{lemma}
\begin{proof}
Think of $g$ as a $C^0$ vector field on $B^m$.
By hypothesis $g$ is outward-pointing on $\partial B^m$.
If $g$ has no zeros in $\operatorname{int}(B^m)$ then by the Poincaré-Hopf Theorem,
see \cite{jubin2009} or \cite{Hirsch1976}, 
the Euler characteristic $\chi(B^m)=0$.
But $\chi(B^m)=1$ and so $g$ must have a zero $\mathbf{x}' \in \operatorname{int}(B^m)$.
\end{proof}

If $g$ is a gradient vector field, the proof is elementary.

\begin{lemma}\label{lemma:gradient}
Let $f: B^m \to \mathbb{R}$ be a $C^1$ function, and let $g := \operatorname{grad}(f)$.
Suppose that $\langle g(\mathbf{u}) , \mathbf{u} \rangle > 0$ for all $\mathbf{u} \in S^{m-1}$.
Then  $g(\mathbf{x}){=}0$ has a root $\mathbf{x}' \in \operatorname{int}(B^m)$.
\end{lemma}
\begin{proof}
Let $\mathbf{x}'$ be a point in $B^m$ that minimizes $f$.
Since $g$ is outward pointing on $\partial B^m$,
$\mathbf{x}' \notin \partial B^m$ and we must have  $\mathbf{x}' \in \operatorname{int}(B^m)$.
It follows at once that $g(\mathbf{x}'){=}0$.
\end{proof}

If the $g$ in Lemma \ref{lemma:PH} is an imbedding, we have a proof with very different flavor.

\begin{lemma}\label{lemma:degree}
Let $ g: B^m \to \mathbb{R}^m$ be a $C^0$ imbedding.
Suppose that $\langle g(\mathbf{u}) , \mathbf{u} \rangle > 0$ for all $\mathbf{u} \in S^{m-1}$.
Then  $g(\mathbf{x}){=}0$ has a root $\mathbf{x}' \in \operatorname{int}(B^m)$.
\end{lemma}
\begin{proof}
By the Jordan-Brouwer separation theorem, \cite{Spanier1966} p. 198, $g(S^{m-1})$
separates $\mathbb{R}^m$ into an interior $g(\operatorname{int}(B^m))$,
and an exterior component.
We are done if we can show that 0 is in the interior component,
and by standard theory of the topological degree,
this is equivalent to showing that
the degree of the map $\mathbf{u} \mapsto g(\mathbf{u})/\abs{g(\mathbf{u})}$ is non-zero.
This map from $S^{m-1}$ to itself is well-defined because $g(\mathbf{u}) \ne 0$.
If $g(\cdot)$ is replaced by the identity the degree is 1,
and since the degree is a homotopy invariant we need only exhibit a homotopy from 
$\mathbf{u} \mapsto g(\mathbf{u})/\abs{g(\mathbf{u})}$ to the identity.
The obvious homotopy works.
Define $g_t(\mathbf{u}) := (1-t)g(\mathbf{u}) + t\mathbf{u}$, for $t \in [0,1]$.
Then
\[
\langle g_t(\mathbf{u}), \mathbf{u} \rangle ~=~ (1-t) \langle g(\mathbf{u}), \mathbf{u} \rangle + t  ~>~ 0
\hspace{10pt} \implies \hspace{10pt}
g_t(\mathbf{u})  \ne 0
\]
And therefore $\mathbf{u} \mapsto g_t(\mathbf{u})/\abs{g_t(\mathbf{u})}$ is well-defined
and gives the desired homotopy.
\end{proof}

\begin{theorem}\label{thm:diffeo}
The function $G_\sigma : \mathbb{R}^m \to \bodyint$ is a diffeomorphism.
\end{theorem}
\begin{proof}
$G_\sigma$ is injective by Theorem \ref{thm:injective}, so we must show it is surjective.
Let $z \in \bodyint$.
Then $ \langle z, \mathbf{u} \rangle < \psi_\body( \mathbf{u} )$, for all $\mathbf{u} \in S^{m-1}$.
By Lemma \ref{lemma:uniform} pick $n$ so large that
$\langle z, \mathbf{u} \rangle < \langle G_\sigma(n \mathbf{u}), \mathbf{u} \rangle$.
Define $g(\mathbf{t}) := G_\sigma(n \mathbf{t}) - z$, for $\mathbf{t} \in B^m$.
Then $g(\mathbf{t})$ satisfies the hypotheses of the three previous Lemmas
\ref{lemma:PH}, \ref{lemma:gradient}, and \ref{lemma:degree}.
To apply \ref{lemma:gradient} $g$ must be a gradient vector field; define
$f(\mathbf{t}) := n^{-1} \int_0^1 K( \langle n \mathbf{t} , \mathbf{w}(x) \rangle ) \, dx  - \langle \mathbf{t}, z \rangle$
and then $\operatorname{grad}(f) = g$.
So 
$g(\mathbf{t})=0$ has a root $\mathbf{t}'$.
Thus $G_\sigma(n \mathbf{t}') = z$ and we are done.
\end{proof}

In numerical work, when one must actually compute $G_\sigma^{-1}(z)$,
Lemma \ref{lemma:gradient} may be preferred, since it is usually more efficient
to find the minimum of a function of $m$ variables than an $m$-dimensional root.

To summarize, we have shown that for any $z \in \bodyint$ and for any squashing function $\sigma(t)$,
$\Lambda_\mathbf{w}f{=}z$ has a unique solution of the form
$f = \sigma(t_1 w_1 + \ldots + t_m w_m)$.   
Later we will see that for a certain specific $\sigma(t)$,
this solution $f$ is the centroid of \emph{all} solutions $Q^\infty \cap  \Lambda_\mathbf{w}^{-1}(z)$.
Denote the continuous map from $z \in \bodyint$ to $f \in \operatorname{int}(Q^\infty)$ by $\gamma$.
Then
\begin{equation}\label{eqn:inverse}
\text{the composition } \hspace{30pt} \bodyint ~~
{\overset{\gamma}\rightarrow} ~~
\operatorname{int}(Q^\infty) ~~
{\overset{\Lambda_\mathbf{w}}\twoheadrightarrow}~~
\bodyint \hspace{30pt}  \text{ is the identity on } \bodyint.
\end{equation}
So $\gamma$ is a right-inverse for $\Lambda_\mathbf{w}$.  
$\gamma$ is certainly not linear, but later we will see that with a reparametrization,
it can be made linear on the line segment from 0 to $\Lambda_\mathbf{w}(\mathbf{1}) = \int_0^1 \mathbf{w}(x) \, dx$.
After reparameterization, $\gamma$ maps this diagonal of the zonoid $\body$
to the line segment of constant functions - the diagonal of the cube $Q^\infty$.
A very different type of right-inverse is constructed in \cite{Samet1987};
that right-inverse takes values in the characteristic functions (the extreme points or ``vertices") of $Q^\infty$,
and it is defined on all of $\body$, including $\partial \body$.

\medskip
Now change the context of $\mathbf{w}$ from fixed to variable, 
and the context of $z$ from variable to fixed.
For later use in Section \ref{sec:discussion} observe that:
\begin{corollary}\label{cor:diffeo2}
Suppose that for a given $\mathbf{w}$ and $\mathbf{t}$
\[
\int_0^1 \sigma( \langle \mathbf{t},\mathbf{w}(x) \rangle ) \mathbf{w}(x) \, dx  ~=~ z_0 ~\in~ \bodyint
\]
If a perturbation of $\mathbf{w}$ is sufficiently small (in the $L^\infty$-norm),
then this equation still has a unique solution $\mathbf{t}$, and $\mathbf{t}$ has $C^1$-dependence on $\mathbf{w}$.
Thus $\sigma( \langle \mathbf{t},\mathbf{w} \rangle )$ also has $C^1$-dependence on $\mathbf{w}$.
\end{corollary}
\begin{proof} Define
$f( \mathbf{w}, \mathbf{t} ) := \int_0^1 \sigma( \langle \mathbf{t},\mathbf{w}(x) \rangle ) \mathbf{w}(x) \, dx$.
Then $f( \mathbf{w}, \mathbf{t} )$ is certainly $C^1$ and by Theorem \ref{thm:diffeo} the partial derivative $D_2 f$
at the given $(\mathbf{w},\mathbf{t})$ is an isomorphism of $\mathbb{R}^m$.
The result now follows from the implicit function theorem, Lang \cite{LangReal} p. 125.
\end{proof}


\emph{Example} for $m=1$.
Here $\body$ is trivial - a compact interval.
The endpoints of the interval are easily computed to be
$\int_0^1 \min(w_1(x),0)\,dx$ and $\int_0^1\max(w_1(x),0)\,dx$.

\emph{Example} for $m=2$, see Figure \ref{fig:diffeo}.
\begin{figure}[!ht]
    \centering
    \subfloat{\includegraphics[width=1.0\linewidth]{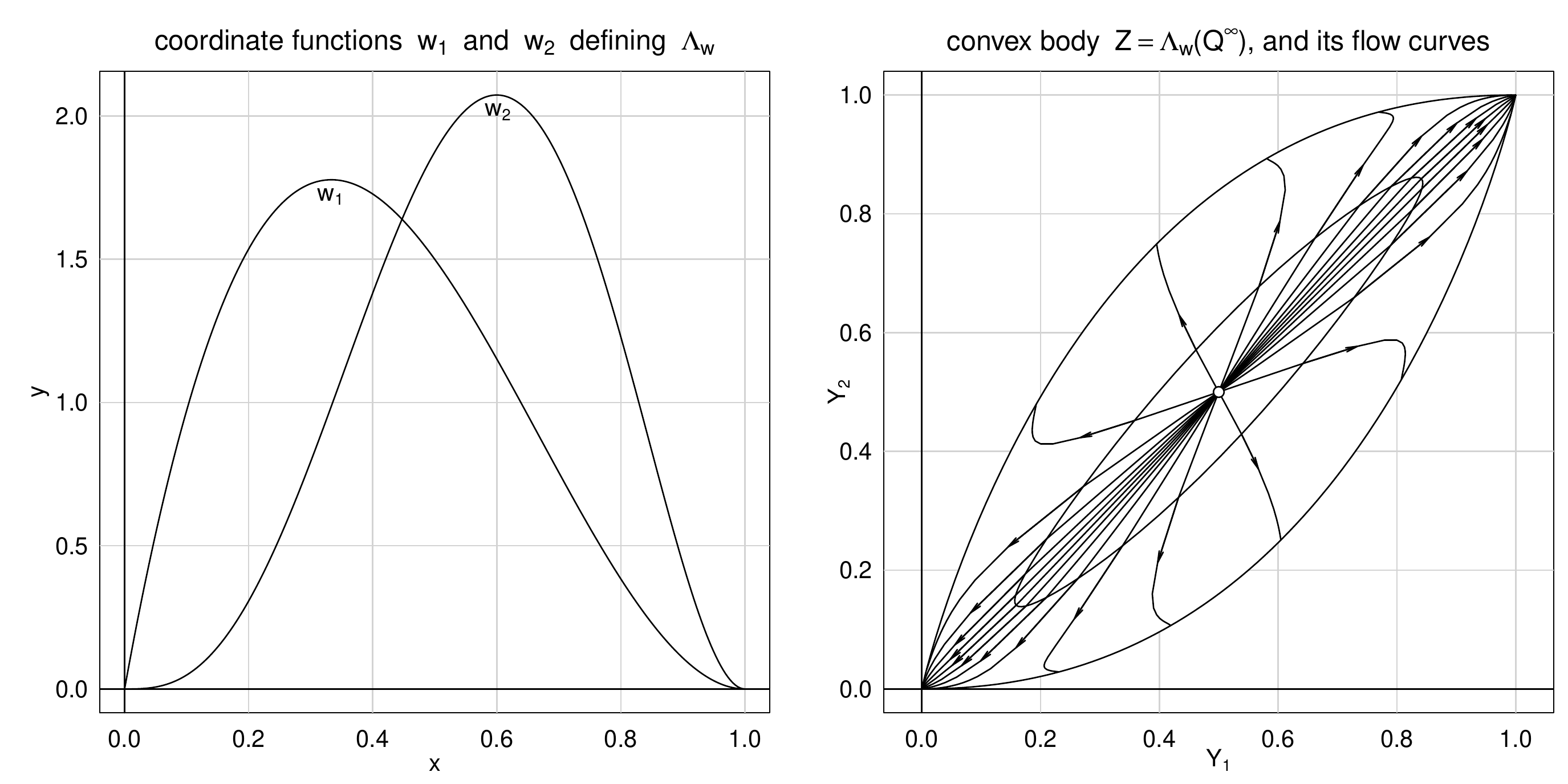}} 
    \caption{Example of $\body$ for $m{=}2$, where $Y_i = \langle w_i, f \rangle$.  
            The area under both $w_1$ and $w_2$ is $1$, and that is why the point $(1,1)$ is on $\partial \body$. 
            The flow curves depend on the choice of $\sigma$; in this example $\sigma(t):=(\tanh(t)+1)/2$.
            The flow has a single source at $\sigma(0)(1,1) = (1/2,1/2)$.}
    \label{fig:diffeo}
\end{figure}
In this example $\body$ is strictly convex, but in general $\body$ can have non-trivial faces or ``flat spots".
This $\body$ has 2 singular points, but in general $\body$ can have any number of them, including 0.
These features, i.e. non-trival faces and singular points, of $\partial \body$ are not pursued here.

The vector field on $\mathbb{R}^m$ given by $g(x) = x$ generates a flow whose integral curves are rays based at $0$.
$G_\sigma$ pushes the field forward to a field on $\bodyint$, 
and this extends to a field on all of $\body$ by setting it to $0$ on $\partial \body$.
The resulting flow on $\body$ has a source at $G_\sigma(0) = \sigma(0)\int_0^1 \mathbf{w}(x) \, dx$.
All other points in $\bodyint$ flow to $\partial \body$, see Figure \ref{fig:diffeo}.

\section{Reduction from $Q^\infty$ to the Finite-Dimensional Cube $Q^n$}\label{sec:reduction}

In this section we reduce the centroid problem from $Q^\infty$ to $Q^n$
using the subspaces $V_n \subset L^\infty[0,1]$
Recall the definition of subspace $V_n$ from Section \ref{sec:nets}.
Define $I_{j,n} := [(j-1)/n,j/n], j=1\ldots n$.
The $n$ indicator functions of $I_{j,n}$ are a standard basis for $V_n$,
and the vectors $\mathsf{e}_i$ are a standard basis of $\mathbb{R}^n$.
There is a canonical isomorphism $\mathbb{R}^n \leftrightarrow V_n$ that preserves the $\infty$-norm,
and this defines a natural inclusion $Q^n \hookrightarrow Q^\infty$.
There is natural identification  $Q^n \simeq Q^\infty \cap V_n$.
There is natural projection $P_n : L^1[0,1]  \to V_n$.
For $f \in L^1[0,1]$ and for $x \in I_{j,n}$
\begin{equation*}
 P_n(f)(x) := n \int_{I_{j,n}} f(x)\,dx  = \text{the average of } f \text{ on } I_{j,n}
\end{equation*}

\smallskip
\begin{center}
\begin{tikzpicture}[scale=1.5]
\node (A) at (0,0.75) {$Q^n$};
\node (B) at (1,0.75) {$\mathbb{R}^n$};
\node (C) at (2,0.75) {$V_n$};
\node (D) at (0,0) {$Q^\infty$};
\node (E) at (1,0) {$\subset$};
\node (F) at (2,0) {$L^\infty[0,1]$};
\node (G) at (0.5,0.75) {$\subset$};
\node (H) at (2,0.375) {$\cap$};
\node (I) at (2.75,0) {$\subset$};
\node (J) at (3.4,0) {$L^1[0,1]$};
\node (K) at (4.75,0) {$\mathbb{R}^m$};
\draw [right hook->] (A) to (D);    
\draw [<->] (B) to (C);    
\draw [->>] (J) to  node[above]{$\Lambda_\mathbf{w}$} (K);
\draw [->>] (J) to  node[above]{$P_n$} (C);
\end{tikzpicture}
\end{center}
\medskip

\begin{theorem}\label{thm:projection}
For the projection $P_n$ and $f,g \in L^1[0,1]$
\begin{enumerate}[label=(\alph*)]
\item $\langle f, P_n(g) \rangle_{1,\infty} ~=~ \langle P_n(f), g \rangle_{\infty,1}$
\item $P_n(f) \to f$ in $L^1[0,1]$, as $n \to \infty$
\end{enumerate}
\end{theorem}

\begin{proof}
For $(a)$ both sides work out to be $\sum_j \left[ n \int_{I_{j,n}} f(x)\,dx \right] \left[ n \int_{I_{j,n}} g(x)\,dx \right]$.
For $(b)$ abbreviate $f_n := P_n(f)$.
For a fixed $x$ the intervals $I_{j,n}$ that contain $x$ ``shrink nicely" to $x$ as $n \to \infty$,
and so by Lebesgue's Differentiation Theorem $f_n(x)$ converges pointwise to $f(x)$ a.e., see \cite{Rudin} p. 168.
Since $f_n$ is dominated by $\abs{f}$, $f_n \to f$ in $L^1[0,1]$,
by Lebesgue's Dominated Convergence Theorem, \cite{Rudin} p. 27.
\end{proof}

Property $(b)$ implies that $\bigcup V_n$ is dense in $L^1[0,1]$,
and so the profile-gauge collection $\mathapprox{P}$ := $\{ V_n \}$ is a directed filtration of $L^1[0,1]$.
In what follows, $\operatorname{int}(Q^n)$ means the interior considering
$Q^n \subset V_n$ (or equivalently $Q^n \subset \mathbb{R}^n$),
and \emph{not} $Q^n$ as a subset of the larger function spaces.
As an example of this usage, note that $P_n( \operatorname{int}(Q^\infty) ) = \operatorname{int}(Q^n)$.

\begin{theorem}\label{thm:reduction1}
If $n$ is sufficiently large, $\Lambda_\mathbf{w}(Q^n)$ is a convex body in $\mathbb{R}^m$, and
$\operatorname{int}(\Lambda_\mathbf{w}(Q^n)) = \Lambda_\mathbf{w}(\operatorname{int}(Q^n))$.
\end{theorem}

\begin{proof}
The fact that $\Lambda_\mathbf{w}(Q^n)$ is compact and convex is trivial.
The rest follows immediately from \ref{thm:interior}
if we can show that the restriction of $\Lambda_\mathbf{w}$ to $V_n$ is surjective.
Consider the map $\Lambda_{P_n(\mathbf{w})}$ obtained from $\Lambda_{\mathbf{w}}$ by replacing each $w_j$ by $P_n(w_j)$.
We claim that these two mappings have the same range.
In fact this follows trivially from part $(a)$ of \ref{thm:projection}.
Now to show that $\Lambda_{P_n(\mathbf{w})}$ is surjective
then by \ref{thm:Gram} we must show that the Gram matrix
$G_{ij}^n := \langle P_n(w_i), P_n(w_j) \rangle$ is positive-definite for sufficiently large $n$.
But this follows at once because we know that $G_{ij} := \langle w_i, w_j \rangle$ is positive-definite
and $P_n(w_i) \to w_i$ as $n \to \infty$.
\end{proof}

\emph{Remark.}
Since $\Lambda_\mathbf{w}(Q^n)$ is the linear image of the cube $Q^n$,
$\Lambda_\mathbf{w}(Q^n)$ is a zonotope, though we do not need this fact.
Recall from the previous section that $Z := \Lambda_\mathbf{w}(Q^\infty)$ is a zonoid.

\begin{theorem}\label{thm:union}
For any $z \in \operatorname{int}(\Lambda_\mathbf{w}(Q^\infty))$,
$z \in \Lambda_\mathbf{w}(\operatorname{int}(Q^n))$ for $n$ sufficiently large.
\end{theorem}

\begin{proof}
First select $n$ so large that $\Lambda_\mathbf{w}(Q^n)$ has an interior, by \ref{thm:reduction1}.
Let $z_0$ be such an interior point.
If the given $z = z_0$ then we are done.
Otherwise, since $z$ is an interior point of $\body$, 
we can take the segment $[z_0,z]$ and extend it slightly past $z$ to $\widehat{z}$, 
where $\widehat{z}$ is also an interior point of $\body$.
By \ref{thm:interior} pick $f \in \mathrm{int}(Q^\infty)$ so that $\Lambda_\mathbf{w}(f) = \widehat{z}$.
Define $f_n := P_n(f)$ and note that $f_n \in \mathrm{int}(Q^n)$.
Define $z_n := \Lambda_\mathbf{w}(f_n)$ and note that $z_n \in \Lambda_\mathbf{w}(\mathrm{int}(Q^n))$.
By \ref{thm:projection}, $z_n \to \widehat{z}$ and thus $\widehat{z} \in \operatorname{cl}(\Lambda_\mathbf{w}(\mathrm{int}(Q^n)))$.
By construction $z_0 \in \mathrm{int}(\Lambda_\mathbf{w}(Q^n)) =  \mathrm{int}(\Lambda_\mathbf{w}(\mathrm{int}(Q^n)))$.
Since $z \in [z_0,\widehat{z})$ we can apply \ref{thm:segment} to conclude that
$z \in \mathrm{int}(\Lambda_\mathbf{w}(\mathrm{int}(Q^n))) = \Lambda_\mathbf{w}(\operatorname{int}(Q^n))$.
\end{proof}

The previous theorem can be interpreted as follows.
If $z \in \operatorname{int}(\body)$ then for sufficiently large $n$ the affine subspace
$\Lambda_\mathbf{w}^{-1}(z)$ intersects the interior of $Q^n$.
It is straightforward to show that in $\mathbb{R}^n$ that subspace is given by
\begin{equation}\label{eqn:Wmat}
W x = z  \hspace{10pt} \text{where }  \hspace{10pt}  W := n^{-1} [ (\iota \circ P_n)(w_1), \ldots, (\iota \circ P_n)(w_m) ]^T
\end{equation}
and where $\iota : V_n \to \mathbb{R}^n$ is the canonical isomorphism.
Row $i$ of $W$ is the projection of $w_i$ onto $V_n$ divided by $n$.
The initial data $\mathbf{w} : [0,1] \to \mathbb{R}^m$ is projected down to an $m\times n$ matrix $W$.
Since the restriction of $\Lambda_\mathbf{w}$ to $V_n$ is surjective,
the solution set $\Lambda_\mathbf{w}^{-1}(z) \cap Q^n$
is a polyhedron of dimension $n-m$.
In the sections below we compute an asymptotic expression
for the centroids of these polyhedra and show that, under certain conditions on $\mathbf{w}$,
these centroids converge to 
$\centroid_\mathapprox{P}( \Lambda_\mathbf{w}^{-1}(z) \cap Q^\infty)$, as $n \to \infty$.
The previous theorem shows that our set and $\mathapprox{P}$ satisfy the above requirement \ref{req:AA}.

\section{The Bilateral Laplace Transform of Volumes of Sections}\label{sec:Volumes}

In this section we consider the problem of computing the $k$-dimensional volume
of the intersection of the unit cube $Q^n := [0,1]^n$ with a $k$-dimensional affine subspace of $\mathbb{R}^n$.
We start with a linear subspace and parameterize the family of all affine subspaces that are parallel to it.
The volume of the intersection is a function of the parameters,
and the Bilateral Laplace Transform of this function has a quite simple form.

Let $\{\wvec_1,\ldots,\wvec_m\}$ be a linearly independent set in $\mathbb{R}^n$.
Let $W := [\wvec_1,\ldots,\wvec_m]^T$ be the full-rank $m\times n$ matrix whose $i$'th row is $\wvec_i$.
The case we have in mind is $\wvec_i = n^{-1}(\iota \circ P_n)(w_i)$, from equation \ref{eqn:Wmat} in the previous section.
For $y \in \mathbb{R}^m$ let $H_y$ := $\{x ~|~ Wx=y\}$.
As $y$ varies over $\mathbb{R}^m$ the affine subspaces $H_y$ sweep out all of $\mathbb{R}^n$.
The dimension of $H_y$ is $k:=n-m$ and there is a $k$-dimensional 
volume measure $\mu_k$ on $H_y$ defined by an isometry with $\mathbb{R}^k$.
An integral expression for the volume we want to compute is
\begin{equation}\label{volumeintegral}
\mu_k(H_y \cap Q^n) = \int_{H_y} \rho(x_1) \cdots \rho(x_n) \, d \mu_k  ~~~~~~~\text{where}~  \rho := \mathbf{1}_{[0,1]}
\end{equation}
and this is the route we will follow.
Since the $\wvec_i$ are linearly independent, the exterior product $\wvec_1 \wedge \cdots \wedge \wvec_m$ is non-zero.
Its length $\norm{\wvec_1 \wedge \cdots \wedge \wvec_m}$ is the $m$-volume of the parallelpiped spanned by the vectors,
and is the product of the singular values of $W$.
Since $W$ is full rank this product is non-zero.
We also have
\[
\norm{\wvec_1 \wedge \cdots \wedge \wvec_m} = \sqrt{\det(W W^T)} \hspace{40pt} \text{(} W W^T   \text{ is the \emph{Gram matrix} of } W \text{)}
\]

It is useful to investigate a more general integral.
Let $f_j : \mathbb{R} \to \mathbb{R}, j=1,\ldots,n$ be measurable functions and let
\begin{equation}\label{weightedconvolution}
g(y) ~~:=~~ \int_{H_y} f_1(x_1) \cdots f_n(x_n) \, d \mu_k   \hspace{40pt}  y \in \mathbb{R}^m
\end{equation}

The tool we use to explore $g(y)$ is the \emph{Bilateral Laplace Transform}.
Here are a few basic facts; for details in the univariate case see \cite{Widder} and for the multivariate case see \cite{brychkov1992}.
For any real-valued function of $m$ real variables $f(t) = f(t_1,\ldots,t_m)$ the Bilateral Laplace Transform $\mathcal{B}\{f\}$
is a complex-valued function of $m$ complex variables $s=(s_1,\ldots,s_m)$, and is defined by
\begin{equation}\label{eqn:BLT}
F(s) := \mathcal{B}(f)(s) := \int_{\mathbb{R}^m} \exp( -\langle s, t \rangle ) ~ f(t) \, dt
\end{equation}
where $\langle s, t \rangle := s_1 t_1 + \cdots + s_m t_m$ is the scalar product of a complex $m$-vector and a real $m$-vector.
The \emph{original} function $f(t)$ is commonly denoted in a lower-case letter,
and the \emph{transformed} function $F(s)$ in the corresponding upper-case letter.
In general, the region of convergence is an $m$-fold product of vertical strips in the complex plane.
If $f(t)$ has compact support, then $F(s)$ is defined everywhere, see \cite{brychkov1992} page 156.
If $f = \mathbf{1}_{[0,1]}$, as in expression \ref{volumeintegral} for the volume of a section of the cube $Q^n$, then 
$\mathcal{B} \{ \mathbf{1}_{[0,1]} \} = (1 - \exp(-s))/s$,
by an easy calculation.

\begin{theorem}\label{thm:BLT}
If $g(y)$ is defined by equation \ref{weightedconvolution} and $F_i := \mathcal{B}\{f_i\}$, then
\[
G(s) :=  \mathcal{B}\{g\}(s) = \norm{\wvec_1 \wedge \cdots \wedge \wvec_m} \prod_{i=1}^n F_i( (W^Ts)_i )
\]
where $(W^Ts)_i$ is the $i$'th component of the vector $W^Ts$.
\end{theorem}

\begin{proof}\label{proof:BLT}
By definition
\begin{align*}
G(s)&:= \int_{\mathbb{R}^m} \exp( -\langle s, y \rangle ) ~ g(y) \, dy  \\
    &= \int_{\mathbb{R}^m} \exp( -\langle s, y \rangle ) \int_{H_y} f_1(x_1)\cdots f_n(x_n) \, d\mu_k \, dy  \\
    &= \int_{\mathbb{R}^m} \int_{H_y} \exp( -\langle s, y \rangle ) f_1(x_1)\cdots f_n(x_n) \, d\mu_k \, dy
\end{align*}
As mentioned earlier, this is really an $n$-dimensional integral, but only $m$ variables are explicit in $y$,
and the remaining $n-m$ are implicit in the integration over $H_y$.
To continue, factor $W$ into its singular value decomposition
\[
W  ~~=~~  U ~ \Sigma ~ V^T  ~~=~~  U \left[ ~ D ~ 0 ~ \right] ~ \left[ ~ V_1 ~ V_2 ~ \right]^T  ~~=~~  U ~ D ~ V_1^T
\]
$U$ is $m\times m$ and $V$ is $n\times n$ and both are orthogonal, see \cite{golub1996} p. 71.
$D = \text{diag}( \sigma_1, \ldots , \sigma_m )$ where $\sigma_i$ is the $i$'th singular value of $W$.
Since $W$ has full rank, $D$ is invertible and the pseudo-inverse $W^+ = V_1 D^{-1} U^T$ satisfies $W W^+ = I_m$.
The matrix $V_2$ has $(n-m)$ columns that are orthonormal and form a basis of $H_0$ = the nullspace of $W$.
If we add parameters $t = (t_1,\ldots ,t_{n-m})$ then integration over $V_2 t$ is an \emph{explicit} integration over $H_y$,
and $d\mu_k$ can be replaced by $dt$.
Combine the vectors $y$ and $t$ to get the augmented $n$-vector $\bar{y} := \left[~ y ~ t ~ \right]$.
Now define $A := \left[ ~ W^+ ~ V_2 ~ \right]$ and make the substitution $x := A \bar{y}$ in the previous integral.
It follows that $x = W^+ y + V_2 t$ and so $W x = W W^+ y + W V_2 t = I_m y + 0 = y$ as required.
One easily checks that $\det(A) = \det(D^{-1}) = \norm{\wvec_1 \wedge \cdots \wedge \wvec_m}^{-1}$ and so $A$ is invertible.

Now apply the change of variables formula in integration, \cite{LangReal} p. 403.
Since the substitution function is linear,
the determinant of the Jacobian is constant and can moved outside the integral and inverted to get
\begin{align*}
G(s)&=  \norm{\wvec_1 \wedge \cdots \wedge \wvec_m}  \int_{\mathbb{R}^n} \exp( -\langle s, Wx \rangle )  f_1(x_1)\cdots f_n(x_n)  \, dx  \\
    &=  \norm{\wvec_1 \wedge \cdots \wedge \wvec_m}  \int_{\mathbb{R}^n} \exp( -\langle W^Ts, x \rangle )  f_1(x_1)\cdots f_n(x_n)  \, dx  \\
    &=  \norm{\wvec_1 \wedge \cdots \wedge \wvec_m} \int_{\mathbb{R}^n} \prod_{i=1}^n \exp( -(W^Ts)_i x_i )  f_i(x_i)   \, dx  \\
    &=  \norm{\wvec_1 \wedge \cdots \wedge \wvec_m} \prod_{i=1}^n \int_{\mathbb{R}}  \exp( -(W^Ts)_i x_i )  f_i(x_i)   \, dx_i \\
    &=  \norm{\wvec_1 \wedge \cdots \wedge \wvec_m} \prod_{i=1}^n  F_i( (W^Ts)_i )  ~~~~~~~\text{by definition of }F_i
\end{align*}
\end{proof}

\emph{Remark}.
In the special case $n{=}2$ and $m{=}1$ and $W{=}[1 ~ 1]$ then $G(s) {=} \sqrt{2} F_1(s) F_2(s)$.
Except for the $\sqrt{2}$, this is the same as $\mathcal{B}(f_1 * f_2)$ ($*$ is convolution, see \cite{brychkov1992} p. 114).
In equation \ref{weightedconvolution} the 1-simplex in $\mathbb{R}^2$ has volume $\sqrt{2}$,
but the convolution integral treats it as a graph with volume 1.
So  \ref{weightedconvolution} can be viewed as type of $n$-fold weighted convolution.

\begin{corollary}\label{cor:volume}
If $\vol(y) := \mu_k(H_y \cap Q^n)$ from equation \ref{volumeintegral}, then
\[
\Vol(s) :=  \mathcal{B}\{vol\}(s) = \norm{\wvec_1 \wedge \cdots \wedge \wvec_m} \prod_{i=1}^n P( (W^Ts)_i )
\]
where $P(s) := \mathcal{B} \{ \mathbf{1}_{[0,1]} \} = (1 - \exp(-s))/s$,
and  $(W^Ts)_i$ is the $i$'th component of the vector $W^Ts$.
\end{corollary}

In the next section we explore properties of $P(s)$ and the related function $K(s) := \log(P(-s))$.
After that we investigate the asymptotics of $\prod_{i=1}^n P( (W^Ts)_i )$ as $n \to \infty$.

\section{The Laplace Approximation}\label{sec:Laplace}

\medskip
The Laplace Aproximation is a classical tool in asymptotics.
The version presented here is modeled on the one
in Wong \cite{wong2001}, Theorem 3, p. 495,
but slightly more general.

\begin{theorem}\label{thm:LaplaceMethod.m}
Given a complex-valued function $f(x)$ on an open subset $U \subset \mathbb{R}^m$ 
with $x_0 \in B_0 \subset U$ where $B_0$ is a compact neighborhood, and
\begin{enumerate}[label=(\alph*)]
\item $f(x)$ is real and positive and $C^4$ on $B_0$, 
$x_0$ is a critical point of $f(x)$, and the Hessian of $f$ at $x_0$ is negative-definite.
We denote the $m \times m$ Hessian by $f''(x_0)$.
\item for any open set $V$ with $x_0 \in V \subset B_0$, 
the supremum of $\abs{f(x)}$ on $U \backslash V$ is strictly less than $f(x_0)$.  

\end{enumerate}
Also given is a sequence of real-valued functions $\phi_n(x)$ on $U$ such that
\begin{enumerate}[label=(\alph*),resume]
\item $\phi_n(x_0) \ne 0$ for all $n$
\item $\phi_n(x)$ is $C^1$ on $B_0$ and $\norm{\phi_n}_{C^1} \le K_1 \abs{\phi_n(x_0)}$ for some $K_1$ and $n$ sufficiently large,
        where $\norm{\cdot}_{C^1}$ is the $C^1{-}norm$ on $B_0$,
\item there are $n_0$, $K_2$, and $k$ so that $\int_{U \backslash B_0} \abs{\phi_n(x) f(x)^{n_0}} \, dx < K_2 n^k \abs{\phi_n(x_0)}$ for $n$ sufficiently large,
\end{enumerate}
Then
\[
\int_U  \phi_n(x)f(x)^n \, dx    ~ \sim ~  \abs{\det(f''(x_0))}^{-\frac{1}{2}}   { \left( \frac{2\pi}{n} \right) } ^{ \frac{m}{2} }    \phi_n(x_0) {f(x_0)}^{n+\frac{m}{2}}  \hspace{20pt} (n \to \infty)
\]
\end{theorem}
For the proof see Appendix \ref{proof:LaplaceMethod.m}.
\medskip

Since $f(x)$ is positive on $B_0$, we can define $h(x) := \log(f(x))$ on $B_0$,
use $f''(x_0)$ = $\allowbreak f(x_0) h''(x_0)$ at the critical point $x_0$,
and the conclusion takes the equivalent form
\[
\int_U \phi_n(x) f(x)^n \, dx   ~ \sim ~  \abs{\det(h''(x_0))}^{-\frac{1}{2}}   { \left( \frac{2\pi}{n} \right) } ^{ \frac{m}{2} }    \phi_n(x_0) e^{n h(x_0)}  \hspace{20pt}  (n \to \infty)
\]

The conclusion implies that the asymptotic limit only depends on the germs
of $f(x)$ and $\phi_n(x)$ at $x_0$.
The contribution away from $x_0$, even at $\infty$, is negligible.

\emph{Remarks}.
The special condition $(b)$ implies that $\abs{f(x)}$ has a unique maximum at $x_0$,
but the converse is false.
In the standard statement of Laplace's Approximation, $\phi_n(x):=\phi(x)$ independent of $n$
and this is what is proved in Copson \cite{Copson} and Widder \cite{Widder}.
Although \cite{Copson} works out a special example with a non-trivial sequence $\phi_n(x)$, on page 43.
Note that the integral on the left of the $\sim$ has an imaginary part in general,
while the right side does not.
The statement then says that the quotient of the imaginary part on the left side,
and the real asymptotic approximation on the right side,
tends to $0$ as $n \to \infty$.

\begin{corollary}\label{cor:LaplaceMethod}
Let $\widehat{\phi}_n(x)$ be another sequence of functions that satisfy parts $(c)$, $(d)$, and $(e)$
in Theorem \ref{thm:LaplaceMethod.m}.
Then
\[
\frac{\int_U  \widehat{\phi}_n(x)f(x)^n \, dx  }{ \int_U  \phi_n(x)f(x)^n \, dx   } ~ \sim ~  \frac{ \widehat{\phi}_n(x_0) }{ \phi_n(x_0) } \hspace{20pt}  (n \to \infty)
\]
\end{corollary}
Later, this will be appplied to the ratio of integrals \ref{eq:centroid} that define the centroid.

\section{The Functions $P(s)$, $K(s)$, and $K_r(s)$}\label{sec:PK}

In this section $s$ denotes a complex number, and all other numbers are real.

Let $\rho(t)$ be a bounded non-negative non-trivial function on $\mathbb{R}$ with compact support.
And define $P(s) := \mathcal{B}(\rho)(s) := \int_{-\infty}^\infty e^{-ts} \rho(t) \, dt$.
The prototypical case is $\rho = \mathbf{1}_{[0,1]}$, and $P(s)=(1-\exp(-s))/s$.
From the conditions on $\rho$, the integral defining $P(s)$ converges for every complex $s$, 
and so $P(s)$ is analytic on $\mathbb{C}$.
For complex $s$ we write $s = \tau + i\upsilon$.

\begin{theorem}\label{thm:Ps}
With $\rho(t)$ as above, $P(\tau + i\upsilon)$ has these properties 
\begin{enumerate}[label=(\alph*)]
\item $P(\tau)$ is real, and hence all derivatives $P'(\tau), ~ P''(\tau), \ldots$ are real
\item $P(\tau)>0$ and $P''(\tau)>0$
\item for any $\tau$, $P(\tau + i\upsilon) \to 0$ as $\abs{\upsilon} \to \infty$
\item for any $\tau$ and $\delta > 0$, $~~~~ \sup_{\abs{\upsilon} \ge \delta} \abs{P(\tau + i\upsilon)} < P(\tau)$
\end{enumerate}
\end{theorem}

\begin{proof}
Following Daniels \cite{Daniels}.
When $s$ is real the integrand in \ref{eqn:BLT} is real and this shows $(a)$.
When $s$ is real, the integrand is positive on supp($\rho$) (except on the boundary) and this shows $P(\tau)>0$.
Since $\rho(t)$ is bounded, differentiation under the integral is justified, so do it twice to get
\begin{equation}
P''(\tau) = \int_{-\infty}^\infty e^{-t\tau} \rho(t) t^2 \, dt
\end{equation}
When $s$ is real, the integrand is positive on supp($\rho$) (except on the boundary and at $t=0$) and this shows $(b)$.
Item $(c)$ is immediate from the Riemann-Lebesgue Lemma, \cite{Goldberg1965} page 7.
For $(d)$ we have
\[
\abs{P(\tau+i\upsilon)} = \left| \int_{-\infty}^\infty e^{-t(\tau+i\upsilon)} \rho(t) \, dt \right|  \\
                 \le \int_{-\infty}^\infty e^{-t\tau} \abs{  e^{-it\upsilon} } \rho(t) \, dt
                 = P(\tau)
\]
Now suppose that $\abs{P(\tau+i\upsilon)} = P(\tau)$ for some $\upsilon \ne 0$.
Then
\begin{align*}
e^{-i\alpha} {P(\tau+i\upsilon)} &= P(\tau)  ~~~~~~ \text{for some real } \alpha                           \\
e^{-i\alpha} \int_{-\infty}^\infty e^{-t(\tau+i\upsilon)} \rho(t) \, dt &=  \int_{-\infty}^\infty e^{-t\tau} \rho(t) \, dt  \\
\int_{-\infty}^\infty e^{-\tau t} \left( e^{-i(\alpha+\upsilon t)} -1 \right) \rho(t) \, dt &= 0                  \\
\intertext{Multiply by -1 and take the real part}
\int_{-\infty}^\infty e^{-\tau t} \left( 1 - \cos(\alpha+\upsilon t) \right) \rho(t) \, dt  &= 0
\end{align*}
Since $\upsilon \ne 0$, $1 - \cos(\alpha+\upsilon t)$ is positive a.e. 
and since $\rho(t) \ge 0$ and $\int_{-\infty}^\infty \rho(t) \, dt > 0$ this is a contradiction,
and so  $\abs{P(\tau+i\upsilon)} < P(\tau)$ for all $\upsilon \ne 0$.
The only way $(d)$ can be false if there is a sequence 
$\abs{\upsilon_n} \to \infty$ with  $\abs{P(\tau+i\upsilon_n)} \to P(\tau)$,
but this is ruled out by $(c)$.
So $(d)$ is true.
\end{proof}

Now specialize to the case where $\rho = \mathbf{1}_{[0,1]}$ and $P(s)=(1-\exp(-s))/s$.
Note that $s{=}0$ is a removable singularity, and that $P(0){=}1$.
If $t$ is real, then by \ref{thm:Ps} $P(t)>0$, so we can define
\begin{align*}
K(t) &:= \log( P(-t) ) = \log( (\exp(t)-1)/t ) \\
\text{and then } ~~~~ K'(t) &=  -P'(-t)/P(-t) = (\coth(t/2) - 2/t + 1)/2 = (L(t/2) + 1)/2 
\end{align*}
$L(x) := \coth(x) - 1/x$ is the \emph{Langevin function}, see \cite{wiki:Langevin}.
If one defines $\sigma(t) := K'(t)$ then one can show from
the above expression  that $\sigma(t)$ is a \emph{squashing function}, according to Definition \ref{defn:squashing}.
From now on we will call it \emph{the} squashing function.
The singularity at $t{=}0$ is removable.
For a plot of $\sigma(t)$, see Figure \ref{fig:squashing}.
\begin{figure}[!ht]
    \centering
    \subfloat{\includegraphics[width=1.0\linewidth]{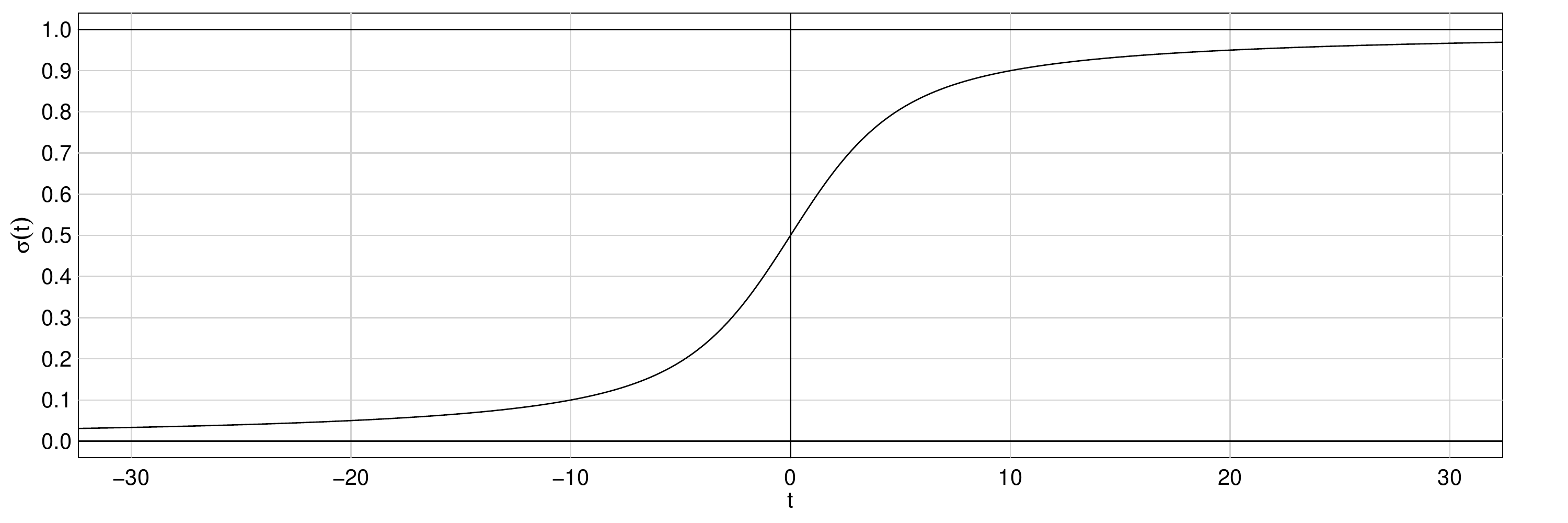}}    
    \caption{The squashing function $\sigma(t) := K'(t) = (L(t/2) + 1)/2$}
    \label{fig:squashing}
\end{figure}

\emph{Remarks}. 
Daniels shows that $K'(t)$ is a squashing function whenever
the smallest interval containing $\operatorname{supp}(\rho)$ is [0,1], see Daniels \cite{Daniels} p. 638.
If $\rho(t)$ is a probability density, then $K(t)$ is its \emph{cumulant-generating function}.
The letter `$K$' is the standard letter for this function.

The zeros of $P(s)$ are all $2\pi ki$ where $k$ is a non-zero integer.
If the half-lines $[2\pi i, +i\infty)$ and $[-2\pi i, -i\infty)$ are cut out of $\mathbb{C}$,
the remainder $\Omega$ is simply connected and $K(t)$ can be extended to  $\Omega$
by integration of $P'(s)/P(s)$ on a path, as in Rudin \cite{Rudin}, Theorem 13.11.
The following facts about $K(s)$, for complex $s$, are useful
\begin{align}\label{eqn:Kfacts}
\begin{split}
\exp( K(s) ) &= (\exp(s)-1)/s               \hspace{26pt} s \in \Omega  \\
\exp( \Re(K(s)) ) &= \abs{ (\exp(s)-1)/s }  \hspace{19pt}  s \in \Omega ~ \text{(take modulus of both sides)} \\
\Re(K(s)) &= \log \abs{ (\exp(s)-1)/s }  \hspace{5pt}     s \in \Omega ~ \text{(take log of both sides)} \\
\end{split}
\end{align}
It is convenient to use the abbreviation $K_r(s) := \Re(K(s))$.
Note that $K_r(s)$ can be defined for \emph{all} $s \in \mathbb{C}$,
if we define $K_r(s) = \log \abs{P(-s)}$, and consider $K_r$ as a function
from $\mathbb{C}$ to the extended line $[-\infty,\infty)$.
Note that $K_r(2\pi ki) = -\infty$ when $k$ is a non-zero integer.

\begin{lemma}\label{lemma:Kr}
On the vertical line $s = \tau + i\upsilon$, with fixed $\tau$ and variable $\upsilon$,
\begin{enumerate}[label=(\alph*)]
\item for $\delta>0$ and $\beta \ne 0$,  $~\sup_{\abs{\upsilon} \ge \delta} K_r(\tau + i\beta\upsilon) < K(\tau)$
\item for all $\upsilon$, $K_r(\tau + i\upsilon) \le K(\tau)$
\item for $\delta>0$ and $\abs{\beta_1} \le \abs{\beta_2}$, 
         $\sup_{\abs{\upsilon} \ge \delta} K_r(\tau + i\beta_2\upsilon) \le \sup_{\abs{\upsilon} \ge \delta} K_r(\tau + i\beta_1\upsilon)$
\item for any $\tau$ there is a constant $C$ (depending on $\tau$) so that \\
           $~ K_r( \tau + i\upsilon ) \le -\log\abs{\upsilon} + C\hspace{20pt}$  for \emph{all} $\upsilon$        
\item for $\delta>0$ and $\tau \ne 0$, there is a constant $C$ (depending on $\delta$ and $\tau$) so that \\
           $~ -K_r( \tau + i\upsilon ) \le \log \abs{\upsilon} + C \hspace{20pt}$    for $\abs{\upsilon} \ge \delta$
\end{enumerate}
\end{lemma}

\begin{proof}
Start with
\begin{align*}
\sup_{\abs{\upsilon} \ge \delta} \abs{P(-\tau - i\upsilon)}         ~&<~ P(-\tau)                    \hspace{30pt}   \text{from Theorem  } \ref{thm:Ps}          \\
\log \left( \sup_{\abs{\upsilon} \ge \delta} \abs{P(-\tau - i\upsilon)} \right)     ~&<~ K(\tau)     \hspace{35pt}   \text{take log of both sides}     \\
\sup_{\abs{\upsilon} \ge \delta} \log \abs{P(-\tau - i\upsilon)}     ~&<~ K(\tau)                    \hspace{35pt}   \text{take sup outside log}     \\
\sup_{\abs{\upsilon} \ge \delta} K_r(\tau + i\upsilon)       ~&<~   K(\tau)                          \hspace{35pt}   \text{from the definition of } K_r(s) \\
\end{align*}
And this is $(a)$ with $\beta{=}1$.
For $\beta{\ne}1$ just reparameterize and get a different $\delta > 0$.
Part $(b)$ is a corollary of part $(a)$.

For $(c)$ note that if both $\beta_0{=}0$ and $\beta_1{=}0$, then both sides are $K(\tau)$ so the inequality is trivial.
If $\beta_1{=}0$ and $\beta_2{\ne}0$, then the right side of the inequality is $K(\tau)$ and the result follows from $(b)$.
So now assume $0 < \abs{\beta_1} \le \abs{\beta_2}$.  Then
\begin{align*}
\{ \upsilon ~:~ \abs{\upsilon} \ge \abs{\beta_2}\delta \}        ~&\subseteq ~     \{ \upsilon ~:~ \abs{\upsilon} \ge \abs{\beta_1}\delta \} \\
\sup_{\abs{\upsilon} \ge \abs{\beta_2} \delta} K_r(\tau + i\upsilon)     &\le   \sup_{\abs{\upsilon} \ge \abs{\beta_1}\delta} K_r(\tau + i\upsilon)  \\
\sup_{\abs{\beta_2 \upsilon} \ge \abs{\beta_2} \delta} K_r(\tau + i \beta_2 \upsilon)     &\le   \sup_{\abs{\beta_1 \upsilon} \ge \abs{\beta_1}\delta} K_r(\tau + i \beta_1 \upsilon)  \\
\sup_{\abs{ \upsilon} \ge   \delta} K_r(\tau + i \beta_2 \upsilon)     &\le   \sup_{\abs{\upsilon} \ge \delta} K_r(\tau + i \beta_1 \upsilon)  \\
\end{align*}
And this is $(c)$.

For $(d)$ we have
\[
K_r( \tau + i\upsilon ) ~:=~ \log \abs{ \frac{\exp(\tau + i\upsilon) -1}{\tau + i\upsilon} } ~=~ \log \abs{\exp(\tau + i\upsilon) -1} - \log \abs{\tau + i\upsilon}
\]
\[
\log \abs{\exp(\tau + i\upsilon) -1}  \le  \log( \abs{\exp(\tau + i\upsilon)} + 1 ) ~=~ \log( \exp(\tau) + 1)
\]
\[
\abs{\upsilon} \le \abs{\tau + i\upsilon}  \hspace{10pt} \implies \hspace{10pt}   -\log\abs{\tau + i\upsilon} ~\le~  -\log\abs{\upsilon}
\]
Add the two previous inequalities to get $K_r( \tau + i\upsilon ) \le -\log\abs{\upsilon} + \log( \exp(\tau) + 1)$, and this is $(d)$.

For $(e)$ we need the fact that $\abs{ \exp(\tau + i\upsilon) -1}  \ge \abs{\exp(\tau)-1}$.
This can be seen by expanding into $\cos(\tau)$ and $\sin(\upsilon)$, or sketching a picture in $\mathbb{C}$.
We also need $\abs{\tau + i\upsilon} \le \sqrt{1 + \tau^2 / \delta^2} \abs{\upsilon}$, for $\abs{\upsilon} \ge \delta$, which is elementary.
Then
\begin{align*}
-K_r( \tau + i\upsilon ) &= \log \abs{ \frac{\tau + i\upsilon}{\exp(\tau + i\upsilon) -1} } \\
                        ~&\le~  \log \abs{\upsilon}  ~+~   \log(1 + \tau^2 / \delta^2)/2  ~-~ \log \abs{ \exp(\tau) - 1 }  \hspace{20pt}\text{for }  \abs{\upsilon} \ge \delta  \\
\end{align*}
and so we can take $C$ to be the two terms after $\log \abs{\upsilon}$.
We note that $(e)$ is false when $\tau{=}0$.
\end{proof}

\emph{Remark 1}.
Although $K(s)$ is analytic on $\Omega$,
we only really need it to be analytic on the horizontal strip $\{ \tau +i\upsilon : \abs{\upsilon}{<}2\pi \} \subset \Omega$.
This is used later in the proof of Lemma \ref{lemma:volume}.

\emph{Remark 2}.
In Section \ref{sec:mainresult} we must consider $\rho(t) := t \mathbf{1}_{[0,1]}$; i.e. the previous $\rho$ multiplied by $t$.
This $\rho(t)$ is still bounded, non-negative, non-trivial, and with compact support.
We denote the new functions corresponding to this new $\rho(t)$ by 
$\widehat{P}(s)$, $\widehat{K}(s)$, and $\widehat{K}_r(s)$.
An easy calculation gives $\widehat{P}(s) = ((s-1)e^s+1)/s^2$.
The roots of $\widehat{P}(s)$ are \emph{not} at $2\pi ki$ with $k$ a non-zero integer.
These roots are shifted to the left and away from 0,
and all are outside the horizontal strip $\{ \tau +i\upsilon : \abs{\upsilon}{<}2\pi \} \subset \Omega$.
The two roots closest to 0 are approximately -2.0888 $\pm$ 7.46149$i$.
Thus $\widehat{K}(s)$ is analytic on the strip, and $\widehat{K}_r(s)$ is a well defined function
from $\mathbb{C}$ to $[-\infty,\infty)$.
From inspection of the proof of Lemma \ref{lemma:Kr},
parts $(a)$, $(b)$, and $(c)$ are valid with $K_r$ replaced by $\widehat{K}_r$.
We use these facts about $\widehat{K}$ and $\widehat{K}_r$ in the proof of the Main Theorem \ref{proof:main}.
Fortunately we do not need the bounds in parts $(d)$ and $(e)$.
In fact part $(e)$ is false for $\widehat{K}_r$;
there are infinitely many vertical lines on which $\widehat{K}_r$ is $-\infty$.

\begin{corollary}\label{cor:Kr}
Let $\tau, \omega \in \mathbb{R}^m$ be fixed, and $\upsilon \in \mathbb{R}^m$ be variable.
\begin{enumerate}[label=(\alph*)]
\item if $\upsilon \ne 0$ and $\hat{\upsilon}:=\upsilon/\abs{\upsilon}$,
then there is a constant $C$, depending on $\langle \omega,\tau \rangle$, so
\[
K_r( \langle \omega, \tau \rangle + i\langle \omega, \upsilon \rangle ) ~\le~  -\log \abs{\upsilon}  -\log\abs{\langle \omega,\hat{\upsilon} \rangle}   + C
\]
\item if $\langle \omega,\tau \rangle \ne 0$ and $\delta{>}0$, 
then there is a constant $C$, depending on $\langle \omega,\tau \rangle$ and $\delta$, so 
\[
-K_r( \langle \omega, \tau \rangle + i\langle \omega, \upsilon \rangle ) ~\le~ \log \abs{\upsilon} + C    \hspace{20pt} \text{for } \abs{\upsilon} \ge \delta
\]
\end{enumerate}
\end{corollary}

\begin{proof}
For $(a)$ pick $C$  by Lemma \ref{lemma:Kr} so that
\[
K_r( \langle \omega, \tau \rangle + i\langle \omega, \upsilon \rangle ) ~\le~ -\log\abs{\langle \omega, \upsilon \rangle} + C   ~=~ -\log\abs{\upsilon} -\log\abs{\langle \omega, \hat{\upsilon} \rangle} + C
\]
and this is $(a)$.  
For $(b)$ pick $C$  by Lemma \ref{lemma:Kr} so that
\begin{equation*}
-K_r( \langle \omega, \tau \rangle + i y )   ~\le~   \log \abs{y} + C   \hspace{20pt} \text{for} ~  \abs{y} \ge \delta
\end{equation*}
Increase $C$ if necessary so that
\begin{equation*}
-K_r( \langle \omega, \tau \rangle + i y )   ~\le~   \log \abs{y} - \log\abs{\omega} + C   \hspace{20pt} \text{for} ~  \abs{y} \ge \delta
\end{equation*}
Let $M = \sup_{\abs{y} \le \delta} \left( -K_r( \langle \omega, \tau \rangle + i y ) \right)$.  
Then $M$ is finite because $\langle \omega, \tau \rangle \ne 0$ and the argument avoids the singular points of $K_r$.
Increase $C$ if necessary, so that $M \le \log(\delta) + C$.
Now pick any $\upsilon \in \mathbb{R}^m$ with $\abs{\upsilon} \ge \delta$. \\
case 1: $\abs{ \langle \omega, \upsilon \rangle} \le \delta$
\begin{equation*}
-K_r( \langle \omega, \tau \rangle + i \langle \omega, \upsilon \rangle )   ~\le~  M  ~\le~ \log(\delta) + C   ~\le~   \log\abs{\upsilon} + C
\end{equation*}
case 2: $\delta \le  \abs{ \langle \omega, \upsilon \rangle}$  By Cauchy-Schwarz we know that  $\abs{ \langle \omega, \upsilon \rangle} \le \abs{\omega}\abs{\upsilon}$. 
\begin{align*}
-K_r( \langle \omega, \tau \rangle + i \langle \omega, \upsilon \rangle )  &~\le~ \log\abs{\langle \omega, \upsilon \rangle}  - \log\abs{\omega} + C  \\
                                                       &~\le~ \log\abs{\upsilon} + \log\abs{\omega} - \log\abs{\omega}  + C ~=~ \log\abs{\upsilon} + C
\end{align*}
And we are done in either case.
\end{proof}

\section{The Main Result}\label{sec:mainresult}

The goal in this section is to prove the main result.

\begin{maintheorem}\label{thm:main} 
Let $\mathbf{w} : [0,1] \to \mathbb{R}^m$ be a step function
with $w_1, \ldots, w_m$ linearly independent and 
$\Lambda_\mathbf{w} : L^1[0,1] \twoheadrightarrow \mathbb{R}^m$ the induced surjective linear operator.
If $y_0 \in \operatorname{int}(\Lambda_\mathbf{w}(Q^\infty))$ 
then for the  profile-gauge directed filtration $\mathapprox{P}$ of $L^1[0,1]$
\begin{equation*} 
\operatorname{centroid}_{\mathapprox{P}} \left( \Lambda_\mathbf{w}^{-1}(y_0) \cap Q^\infty \right) ~~=~~ \sigma( \langle \tau_0, \mathbf{w} \rangle)
\end{equation*}
where $\tau_0 \in \mathbb{R}^m$ is the unique solution to the saddlepoint equation
\begin{equation}\label{eqn:saddlepoint}
\int_0^1 \sigma( \langle \tau, \mathbf{w}(x) \rangle ) \mathbf{w}(x) \,dx  =  y_0
\end{equation}
and where  $\sigma(t) := (\coth(t/2) - 2/t + 1)/2  = (L(t/2) + 1)/2$.
\end{maintheorem}

The main tools are the Bilateral Laplace Transform complex inversion formula,
the saddlepoint method, and the Laplace Approximation.
The variables $s$ and $z$ are complex $m$-vectors, and all other scalars and vectors are real.
If $z = (z_1,\ldots,z_m)$, 
and $\mathbf{w}(x) := (w_1(x), \ldots, w_m(x))$ are the coordinate functions of $\mathbf{w}$,
then $\langle z, \mathbf{w}(x) \rangle = z_1 w_1(x) + \cdots + z_m w_m(x)$.
The variables $y_0, \tau, \tau_0,\text{ and } w_{j,n}$ are real $m$-vectors.

Recall the Laplace Transform of the volume function $\vol(y) := \mu_k(H_y \cap Q^n)$ from Corrolary \ref{cor:volume}
\[
\Vol(s) :=  \mathcal{B}\{v\}(s) = \norm{\wvec_1 \wedge \cdots \wedge \wvec_m} \prod_{j=1}^n P( W_j^T s ) 
\]
or using the definition of $K(s)$
\begin{equation*}
\Vol(s) = \norm{\wvec_1 \wedge \cdots \wedge \wvec_m} \prod_{j=1}^n \exp( K(-W_j^T s ) ) 
~=~ \norm{\wvec_1 \wedge \cdots \wedge \wvec_m}  \exp \left( \sum_{j=1}^n  K(-W_j^T s ) \right)
\end{equation*}
where $W_j^T$ is the $j$'th row of $W^T := [\wvec_1,\ldots,\wvec_m]$.
Recall that in our case
\begin{equation}\label{eqn:wjn}
W_j = \int_{I_{j,n}} \mathbf{w}(x) \, dx  \hspace{20pt} \text{and define} ~  w_{j,n} := n W_j = n \int_{I_{j,n}} \mathbf{w}(x) \, dx  \in \mathbb{R}^m
\end{equation}

Using the multidimensional complex inversion formula, \cite{brychkov1992} p. 100, we can recover $\vol(y)$ from $\Vol(s)$ with
\[
\vol(y) = (2\pi i)^{-m} \lim_{\upsilon_k \to \infty} \int_{c_1-i\upsilon_1}^{c_1+i\upsilon_1} \int_{c_2-i\upsilon_2}^{c_2+i\upsilon_2} \ldots  \int_{c_m-i\upsilon_m}^{c_m+i\upsilon_m} e^{\langle s,y \rangle} \Vol(s) \, ds_m \ldots ds_2 \, ds_1
\]
where 
$s := (s_1,\ldots,s_m)$ and $s_k = c_k + i\upsilon_k$ and
$(c_1,\ldots,c_m)$ is any real $m$-vector in the interior of the region of convergence of $\Vol(s)$.
The integral is understood in the sense of the principal value,
and we think of it as built from $m$ vertical lines (later perturbed to vertical paths) in $\mathbb{C}$, each one crossing the real axis.
A more thorough discussion in the univariate case is in Chapter VI of Widder \cite{Widder}.

Following \cite{brychkov1992} it is convenient to use the abbreviated notation
\[
\vol(y) = (2\pi i)^{-m} \int_{(c)} e^{\langle s,y \rangle} \Vol(s) \, ds    \hspace{20pt} \text{where }c = (c_1,\ldots,c_m) \in \mathbb{R}^m
\]
It is also convenient to make the substitution $z = -s/n$ to get
\[
\vol(y) = \left( \frac{n}{2\pi i} \right)^m \int_{(\tau)} \exp(-n{\langle z,y \rangle}) \Vol(-nz) \, dz    \hspace{20pt} \text{where }\tau \in \mathbb{R}^m
\]
Now substitute the above expression for $\Vol(s)$
\begin{equation}\label{eqn:volume3}
\vol(y) ~=~ \norm{\wvec_1 \wedge \cdots \wedge \wvec_m} \left( \frac{n}{2\pi i} \right)^m \mathop{\mathlarger{\mathlarger{\int}}}_{(\tau)} \exp \left( \sum_{j=1}^n  K(n W_j^T z) - n{\langle z,y \rangle} \right) \, dz
\end{equation}
Use the fact that $n W_j = w_{j,n}$ from \ref{eqn:wjn},and rewrite the exponent in the integral as
\[
\left( \sum_{j=1}^n K( \langle w_{j,n}, z \rangle ) - n \int_0^1 K(  \langle z, \mathbf{w}(x) \rangle )\,dx \right) ~+~  n \left( \int_0^1 K(  \langle z, \mathbf{w}(x) \rangle )\,dx -  \langle y, z \rangle \right)
\]
After reorganizing, we have
\begin{equation}\label{eqn:volume2}
\vol(y) = \norm{\wvec_1 \wedge \cdots \wedge \wvec_m}  \left( \frac{n}{2\pi i} \right)^m  \int_{(\tau)} \phi_n(z) f(z)^n \, dz
\end{equation}
where
\begin{equation}\label{eqn:phi_n}
\phi_n(z) := \exp \left[ \sum_{j=1}^n K( \langle w_{j,n}, z \rangle ) - n \int_0^1 K(  \langle z, \mathbf{w}(x) \rangle )\,dx  \right]  ~\hspace{10pt} w_{j,n}{:=n}\int_{I_{j,n}} \mathbf{w}(x) \, dx
\end{equation}
or
\begin{align*}
\phi_n(z) &:= \exp \left[ \sum_{j=1}^n  \left( K( \langle w_{j,n}, z \rangle ) - n \int_{I_{j,n}} K( \langle z, \mathbf{w}(x) \rangle ) \, dx  \right) \right]    \\
f(z) &:= \exp  \left[ h(z) \right]  \hspace{20pt}  h(z) := \int_0^1 K(  \langle z, \mathbf{w}(x) \rangle )\,dx  ~-~ \langle y, z \rangle 
\end{align*}
The purpose for all this reorganization is that \ref{eqn:volume2} is now in the form that the Laplace Approximation can be used.
Note that $f(z)$ does not depend on $n$, which one of the chief requirements.
Before proceeding we have an easy

\begin{lemma}\label{lemma:alpha}
If $0 < \alpha_1 \le \alpha_2$ and $\delta > 0$, then there is a $C$, depending on $\alpha_1$, $\alpha_2$, and $\delta$, so that
\[
-\alpha_2 \log\abs{\upsilon} \le -\alpha_1 \log\abs{\upsilon} + C \hspace{20pt} \text{for } \abs{\upsilon} \ge \delta
\]
\end{lemma}
\begin{proof}
\begin{align*}
\abs{\upsilon} \ge \delta ~\implies~ -(\alpha_2 - \alpha_1)\log\abs{\upsilon} &\le -(\alpha_2 - \alpha_1)\log(\delta) \\
-\alpha_2 \log\abs{\upsilon} &\le -\alpha_1 \log\abs{\upsilon} + (\alpha_1 - \alpha_2)\log(\delta)
\end{align*}
\end{proof}

\begin{lemma}\label{lemma:volume}
With $\Lambda_\mathbf{w}$ as above and $y_0 \in \operatorname{int}(\Lambda_\mathbf{w}(Q^\infty))$,
then the saddlepoint equation \ref{eqn:saddlepoint} has a unique solution $\tau_0$, 
and the volume of the cube section 
$\vol(y_0) := \mu_k(\Lambda_\mathbf{w}^{-1}(y_0) \cap Q^n)$ 
has the asymptotic expression
\begin{equation}\label{eqn:denom1}
\vol(y_0)  ~\sim~   \norm{\wvec_1 \wedge \cdots \wedge \wvec_m} { \left( \frac{n}{2\pi} \right) }^\frac{m}{2}   \abs{\det(f''(\tau_0))}^{-\frac{1}{2}}  \phi_n(\tau_0) {f(\tau_0)}^{n+\frac{m}{2}}       \hspace{10pt}(n \to \infty)
\end{equation}
\end{lemma}

\begin{proof}
We seek a real critical point $\tau_0$ of $f(z)$ that we can substitute in \ref{eqn:volume2}
and apply the saddlepoint method.
A trivial calculation shows that $f(z)$ and $h(z)$ have the same critical points.
By calculating partial derivatives of $h(z)$ and setting to $0$, 
$\tau_0$ must be a root of the $m$-dimensional saddlepoint equation
\begin{align*}
\int_0^1 K'( \langle \tau, \mathbf{w}(x) \rangle ) \mathbf{w}(x)\,dx  ~~=~~ y_0
\end{align*}
Since the left side is $G_{K'}(\tau)$, there is a unique root $\tau_0$ by Theorem \ref{thm:diffeo}.
Apply the Complex Morse Lemma to deform the $m$ paths of integration,
so that for $z$ on the paths and near $\tau_0$,
$f(z)$ is real and the multivariate Laplace Approximation can be used.
Note that since $f(z)$ is analytic, it is certainly of class $C^4$ as required.
Also note that for $K(z)$ to be defined, all paths must be in the simply-connected $\Omega \subset \mathbb{C}$.
So if any component of $\tau_0$ is $0$, deform the corresponding path slightly,
but only away from the real axis, so that its real part is non-zero.  
Denote the real values on the paths away from $\mathbb{R}^m$ by $\tau'_0$.
The path homotopies do not change the value of the integral \ref{eqn:volume2}.

We must check conditions $(a),\ldots,(e)$ of the Laplace Approximation \ref{thm:LaplaceMethod.m}.
The restriction of $f(z)$ to $\mathbb{R}^m$ is denoted by $f(\tau) = f(\tau_1,\ldots,\tau_m)$.
The restriction of $f(z)$ to the plane of integration $\tau_0 + i\upsilon$ is denoted by $f(\upsilon) = f(\upsilon_1,\ldots,\upsilon_m)$.
And similarly for $h(z)$.

For part $(a)$, a trivial calculation shows that the Hessian of $f(\tau)$ at $\tau_0$
is a positive multiple of the Hessian of $h(\tau)$ at $\tau_0$.
But the latter Hessian is the same as the Jacobian $(\mathbf{D}G_{K'})_{ij}$ in equation \ref{eqn:DG}.
We saw there that the matrix is positive-definite, and so the Hessian of $f(\tau)$ at $\tau_0$ is positive-definite too.
Therefore the Hessian of $f(\upsilon)$ at $0$ is negative definite and $(a)$ is shown.

For part $(b)$ our goal is to show that for any $\delta > 0$
\begin{equation*}
\sup_{\abs{\upsilon}\ge\delta} \Re \left[ h(\tau + i\upsilon) \right]  ~~<~~ \Re \left[ h(\tau) \right]
\end{equation*}
Since we are only looking at the real part, the $\langle y_0, z \rangle$ term in $h(z)$ drops out
and we are left with just the integral term in $h(z)$.
Let the step function $\mathbf{w}$ be defined on $N$ subintervals of length $\mu_k{>}0$ with value $\omega_k \in \mathbb{R}^m$, $k{=}1,\ldots,N$.
Our goal is then to show
\begin{equation}\label{eqn:goalb}
\sup_{\abs{\upsilon}\ge\delta} \sum_{k=1}^N \mu_k K_r( \langle \tau, \omega_k \rangle  + i \langle \upsilon, \omega_k \rangle ) ~<~ \sum_{k=1}^N \mu_k K_r( \langle \tau, \omega_k \rangle )
\end{equation}

We are examining the $\sup$ over the region 
$\{ \upsilon \in \mathbb{R}^m : \abs{\upsilon}\ge\delta \}$, 
so it is convenient to use the \emph{spherical decomposition} $\upsilon = \abs{\upsilon}\hat{\upsilon}$ for unit vector $\hat{\upsilon} \in S^{m-1}$.
We take a fixed $\hat{\upsilon}$ and bound the $\sup$ over the corresponding ray to $\infty$,
and then go over the entire sphere $S^{m-1}$.
The $K_r$ terms on the left side become 
$K_r( \langle \tau, \omega_k \rangle  + i \langle \hat{\upsilon}, \omega_k \rangle \abs{\upsilon} )$.
Apply Lemma \ref{lemma:Kr} with $\beta_1{=}0$ and $\beta_2{=}\langle \hat{\upsilon}, \omega_k \rangle$ 
to conclude that for a fixed $\hat{\upsilon}$ and for \emph{every} $k$
\begin{equation}\label{eqn:b1}
\sup_{\abs{\upsilon}\ge\delta} K_r( \langle \tau, \omega_k \rangle  + i \langle \hat{\upsilon}, \omega_k \rangle \abs{\upsilon} ) ~\le~ K_r( \langle \tau, \omega_k \rangle )
\end{equation}
Since the right side does not depend on $\hat{\upsilon}$, we can let $\hat{\upsilon}$ range over the entire sphere
and conclude that \ref{eqn:b1} is true for \emph{all} $\abs{\upsilon}\ge\delta$.

Define the function $F(\hat{\upsilon}) := \max_k \abs{\langle \hat{\upsilon}, \omega_k \rangle}$  on  $S^{m-1}$.
$F(\hat{\upsilon})$ is continuous and non-negative, and it cannot vanish because that would imply that the functions $w_1,\ldots,w_m$ are linearly dependent, which is false.
Let $F_{\min} := \min_{\hat{\upsilon} \in S^{m-1}} F(\hat{\upsilon})$; 
$F_{\min}$ is positive because $S^{m-1}$ is compact.
So for every $\upsilon{\ne}0$, $\abs{\langle \hat{\upsilon}, \omega_k \rangle} \ge F_{\min}$, for \emph{some} $k$.
Apply \ref{lemma:Kr} again with $\beta_1{=}F_{\min}$ to conclude that 
for a fixed $\hat{\upsilon}$ there are \emph{some} $k$ (meaning 1 or more $k$ depending on $\hat{\upsilon}$), with
\begin{equation*} 
\sup_{\abs{\upsilon}\ge\delta} K_r( \langle \tau, \omega_k \rangle  + i \langle \hat{\upsilon}, \omega_k \rangle \abs{\upsilon} ) 
~\le~  \sup_{\abs{\upsilon}\ge\delta}K_r( \langle \tau, \omega_k \rangle  + i F_{\min} \abs{\upsilon} )   
~<~ K_r( \langle \tau, \omega_k \rangle )
\end{equation*}
The middle term does not depend on $\hat{\upsilon}$ directly,
but only indirectly through $k$, of which there are only finitely many, and so we can conclude
that for \emph{all} $\hat{\upsilon}$ there are \emph{some} $k$ (depending on $\hat{\upsilon}$) with
\begin{equation}\label{eqn:b2}
\sup_{\abs{\upsilon}\ge\delta} K_r( \langle \tau, \omega_k \rangle  + i \langle \hat{\upsilon}, \omega_k \rangle \abs{\upsilon} )  ~<~ K_r( \langle \tau, \omega_k \rangle )
\end{equation}
Multiply \ref{eqn:b1} and \ref{eqn:b2} by $\mu_k$ and sum (while dropping appropriate terms from \ref{eqn:b1}) to get
\begin{equation}
\sum_{k=1}^N  \mu_k   \sup_{\abs{\upsilon}\ge\delta}    K_r( \langle \tau, \omega_k \rangle  + i \langle \upsilon, \omega_k \rangle ) ~<~ \sum_{k=1}^N \mu_k K_r( \langle \tau, \omega_k \rangle )
\end{equation}
Move the $\sup$ outside of the sum and left side stays the same, or become smaller.
But the new expression is the left side of our goal \ref{eqn:goalb} and so this concludes part $(b)$.

Part $(c)$ is trivial.

For part $(d)$, write:
\begin{equation}\label{eqn:g_n}
\phi_n(z) := \exp \left[ g_n(z) \right] \hspace{5pt}  g_n(z) :=  \sum_{j=1}^n  \left( K( \langle w_{j,n}, z \rangle ) - n \int_{I_{j,n}} K( \langle z, \mathbf{w}(x) \rangle ) \, dx  \right)
\end{equation}
$\mathbf{w}(x)$ is now a step function with $J$ jumps.
Choose $n$ so large that each subinterval $I_{j,n}$ contains at most one jump.
If there are \emph{no} jumps on $I_{j,n}$ then $\mathbf{w}$ has the constant value $w_{i,j}$ there, and
\[
n \int_{I_{j,n}} K( \langle z, \mathbf{w}(x) \rangle ) \, dx  ~=~ n K( \langle z, w_{i,j} \rangle ) \int_{I_{j,n}} 1 \, dx ~=~ K( \langle z, w_{i,j} \rangle )
\]
Thus the $j'th$ term in $g_n(z)$ vanishes, and the sum has only $J$ non-zero terms.
Let $I_{j,n}$ be a subinterval with a jump at $x_0 = \lambda_1 (j-1)/n + \lambda_2 j/n$
where $\lambda_1 + \lambda_2 = 1$ and $\lambda_1,\lambda_2 \ge 0$.
And suppose  $\mathbf{w}(x)$ jumps from $y_1$ to $y_2$ at $x_0$.
Then the corresponding term in $g_n(z)$ is
\begin{equation}\label{eqn:jumpd}
K( \lambda_1 \langle y_2,z \rangle  +  \lambda_2 \langle y_1,z \rangle)  ~-~ \lambda_1 K( \langle y_2,z \rangle ) - \lambda_2 K( \langle y_1,z \rangle )
\end{equation}
and so $g_n(z)$ is a sum of $J$ terms like this one.
The vectors $y_1$ and $y_2$ do not depend on $n$,
but $\lambda_1$ and $\lambda_2$ \emph{do} depend on $n$.
At the saddlepoint $\tau_0$ define $B_r := \{ \tau_0 + i\upsilon : \abs{\upsilon} \le r \} \subset \mathbb{C}^m$.
As $z$ varies over $B_r$, $\langle y_1, z \rangle$ varies over a vertical line segment centered at 
$\langle y_1, \tau_0 \rangle$ on the real axis,
and similarly for $\langle y_2, \tau_0 \rangle$.
As $z$ varies over $B_r$ and $n \to \infty$, 
$\lambda_1 \langle y_2,z \rangle  +  \lambda_2 \langle y_1,z \rangle$
varies over the convex hull of the 2 segments, which is a trapezoid.
In $g_n(z)$ and for $z \in B_r$, all arguments to $K()$ are in the compact set of $J$ trapezoids in $\mathbb{C}$.
If necessary, shrink $r$ so that all the trapezoids are inside the horizontal strip
$\{ \tau +i\upsilon : \abs{\upsilon}{<}2\pi \}$, where $K(s)$ is analytic.
For this statement about the strip, see \emph{Remark 1} after Lemma \ref{lemma:Kr}.
It follows that there are constants $C_1$ and $C_2$ so that
\[
-C_1 \le \Re \left[ g_n(z) \right] \le C_2   \hspace{30pt} \text{for } z \in B_r \text{ and } n \text{ sufficiently large}
\]
which implies
\[
\abs{\phi_n(z)} \le \exp(C_2) \text{ and } \abs{\phi_n(z)}^{-1} \le \exp(C_1) \hspace{10pt} \text{for } z \in B_r \text{ and } n \text{ sufficiently large}
\]
Similary, since $K'(z)$ is analytic on the same strip, there is a $C_3$ so
\[
\norm{\phi_n}_{C^1} \le C_3 \hspace{10pt} \text{for } n \text{ sufficiently large, where } \norm{\cdot}_{C^1} \text{ is the } C^1\text{-norm on } B_r
\]
and so
\[
\norm{\phi_n}_{C^1} \le C_3 \cdot 1  \le  C_3 \exp(C_1) \abs{\phi_n( \tau_0 )}
\]
and this concludes part $(d)$.

Part $(e)$ is the longest, and it is easier to divide it into 2 claims
\begin{enumerate}
\item For $\delta>0$, there are $\alpha > 0$ (depending on the step function $\mathbf{w}$)
and $K_1$ (depending on $\mathbf{w}$, $y_0$, and $\delta$) so that
\begin{equation*}
\abs{ f(\tau_0 + i\upsilon) } ~\le~ K_1 \abs{\upsilon}^{-\alpha}   \hspace{10pt} \text{for } \abs{\upsilon} \ge \delta
\end{equation*}

\item For $\delta>0$, there are $\beta > 0$ (depending on the step function $\mathbf{w}$ but not on $n$)
and $K_1$ (depending on $\mathbf{w}$ and $\delta$, but not on $n$)  so that
\begin{equation*}
\abs{ \phi_n(\tau_0 + i\upsilon) } ~\le~  K_2 \abs{\upsilon}^{\beta}    \hspace{10pt} \text{for } \abs{\upsilon} \ge \delta
\end{equation*}
\end{enumerate}

For claim 1, write
\[
h(\tau + i\upsilon) ~=~ \sum_{k=1}^N \mu_k K( \langle \tau, \omega_k \rangle  + i \langle \upsilon, \omega_k \rangle )  ~-~ \langle y_0, \tau + i\upsilon \rangle
\]
as in part (b) above.
Also take the function $F(\hat{\upsilon})$ and its minimum $F_{\min}$ from part (b).
Pick a fixed $\hat{\upsilon} \in S^{m-1}$ and define the set of indices
$N_{\hat{\upsilon}} := \{ k :~ \abs{\langle \hat{\upsilon}, \omega_k \rangle} \ge F_{\min}/2 \}$.
By the construction of $F_{\min}$, $N_{\hat{\upsilon}}$ is non-empty.
If $k \in N_{\hat{\upsilon}}$ then by Corollary \ref{cor:Kr}, 
\[
K_r( \langle \tau, \omega_k \rangle  + i \langle \upsilon, \omega_k \rangle ) ~\le~ -\log \abs{\upsilon}  -\log\abs{\langle \omega_k,\hat{\upsilon} \rangle}  + C
~\le~ -\log \abs{\upsilon}  -\log(F_{\min}/2) + C
\]
where this constant $C$ depends on $\langle \tau,\omega_k \rangle$.
If $k \notin N_{\hat{\upsilon}}$ then by Lemma \ref{lemma:Kr}, 
\[
K_r( \langle \tau, \omega_k \rangle  + i \langle \upsilon, \omega_k \rangle ) ~\le~ K_r( \langle \tau, \omega_k \rangle)
\]
Multiply these two inequalities and sum to get
\begin{equation}\label{eqn:claim1A}
\sum_{k=1}^N \mu_k K_r( \langle \tau, \omega_k \rangle  + i \langle \upsilon, \omega_k \rangle ) ~\le~ -\left[ \sum_{k\in N_{\hat{\upsilon}}} \mu_k \right] \log \abs{\upsilon}  ~+~ C
\end{equation}
where this $C$ depends on $\langle \tau,\omega_k \rangle$ for the $k$'s inside and outside $N_{\hat{\upsilon}}$, and on $F_{\min}$.
Since there are only finitely many subsets $N_{\hat{\upsilon}}$, this $C$ can be taken so large that it applies to \emph{all} $\hat{\upsilon}$.
Define
\[
\alpha := \min_{ \hat{\upsilon} \in S^{m-1} } \left[ \sum_{k\in N_{\hat{\upsilon}}} \mu_k \right]   ~>~  0
\]
Similarly, although the number of $\hat{\upsilon} \in S^{m-1}$ is infinite, 
there are only finitely many subsets $N_{\hat{\upsilon}}$, and so $\alpha > 0$.
Apply Lemma \ref{lemma:alpha} to get
\begin{equation}\label{eqn:claim1B}
-\left[ \sum_{k\in N_{\hat{\upsilon}}} \mu_k \right] \log\abs{\upsilon}  ~\le~  -\alpha \log\abs{\upsilon} ~+~ C'  \hspace{15pt} \text{for all }  \abs{\upsilon} \ge \delta
\end{equation}
where this $C'$ depends on $\delta$, $\alpha$, and all possible sums over $N_{\hat{\upsilon}}$.
Now combine inequalities  \ref{eqn:claim1A} and \ref{eqn:claim1B} and replace $C'+C$ by $C$  get 
\begin{equation*}
\sum_{k=1}^N \mu_k K_r( \langle \tau, \omega_k \rangle  + i \langle \upsilon, \omega_k \rangle ) ~\le~ -\alpha\log\abs{\upsilon} + C  \hspace{20pt} \text{for } \abs{\upsilon} \ge \delta
\end{equation*}
Replace $\tau$ by the saddlepoint $\tau_0$ and subtract $\langle y_0,\tau_0 \rangle$ from both sides.
\begin{align*}
\sum_{k=1}^N \mu_k K_r( \langle \tau_0, \omega_k \rangle  + i \langle \upsilon, \omega_k \rangle ) - \langle y_0,\tau_0 \rangle ~&\le~ -\alpha\log\abs{\upsilon} + C - \langle y_0,\tau_0 \rangle  \hspace{20pt} \text{for } \abs{\upsilon} \ge \delta \\
\Re \left[ h(\tau_0 + i\upsilon) \right] ~&\le~ -\alpha\log\abs{\upsilon} + C - \langle y_0,\tau_0 \rangle  \hspace{20pt} \text{for } \abs{\upsilon} \ge \delta \\
\end{align*}
Exponentiate both sides to get claim 1.

For claim 2, perturb $\tau'_0$ again if necessary so that $\langle \tau'_0,\omega_k \rangle \neq 0$ for every $\omega_k \neq 0$.
We are using $\tau'_0$ here because we only care about large $\abs{\upsilon}$.
From part $(d)$ we know that $\abs{\phi_n( \tau'_0 + i\upsilon)} = \exp [ \Re( g_n(\tau'_0 + i\upsilon) ) ]$
where $g_n$ is the sum of $J$ (${=}N-1)$ terms of the form
\begin{equation}\label{eqn:claim2A}
K_r( \langle \bar{y},\tau'_0 \rangle + i\langle \bar{y},\upsilon \rangle) - \lambda_1 K_r(\langle y_2,\tau'_0 \rangle + i\langle y_2,\upsilon \rangle) - \lambda_2 K_r(\langle y_1,\tau'_0 \rangle + i\langle y_1,\upsilon \rangle)
\end{equation}
The vectors $y_1$ and $y_2$ are taken from the set of all $\omega_k, k=1,\ldots,N$.
The vector $\bar{y} = \lambda_1 y_2 + \lambda_2 y_1$ depends on $n$ (since $\lambda_1$ and $\lambda_2$ depend on $n$),
but $\bar{y}$ is on the compact line segment $[y_1,y_2]$ and so $\langle \bar{y},\tau'_0 \rangle$ is bounded for all $n$, and for all $J$ jumps.
So there is a constant $C$ so that $K_r(\langle \bar{y},\tau'_0 \rangle) \le C$.
By Lemma \ref{lemma:Kr}
\begin{equation}\label{eqn:claim2B}
K_r( \langle \bar{y},\tau'_0 \rangle + i\langle \bar{y},\upsilon \rangle) ~\le~ K_r(\langle \bar{y},\tau'_0 \rangle) ~\le~ C
\end{equation}
In the singular cases $\langle \bar{y},\tau'_0 \rangle {=} 0$ and $\langle \bar{y},\upsilon \rangle {=} 2\pi j (j\neq 0)$
there is not problem, since then the left side is $-\infty$.
This bounds the first term in \ref{eqn:claim2A}.
For the $y_2$ term, note that if $y_2=0$ then the term is $-\lambda_1 K_r(0) = 0$ and it can be ignored.
If $y_2\neq 0$ then $\langle y_2,\tau'_0 \rangle \neq 0$ by the choice of $\tau'_0$.
So by Corollary \ref{cor:Kr} there is a $C'$ so that
\begin{equation}\label{eqn:claim2C}
- \lambda_1 K_r(\langle y_2,\tau'_0 \rangle + i\langle y_2,\upsilon \rangle) ~\le~ \lambda_1 \log \abs{\upsilon} + \lambda_1 C' \hspace{20pt} \text{for } \abs{\upsilon} \ge \delta
\end{equation}
and the same $C'$ can be chosen for all $J$ jumps.
Similarly, for the $y_1$ term
\begin{equation}\label{eqn:claim2D}
- \lambda_2 K_r(\langle y_1,\tau'_0 \rangle + i\langle y_1,\upsilon \rangle) ~\le~ \lambda_2 \log \abs{\upsilon} + \lambda_2 C' \hspace{20pt} \text{for } \abs{\upsilon} \ge \delta
\end{equation}
Add up \ref{eqn:claim2B}, \ref{eqn:claim2C}, and \ref{eqn:claim2D} over all $J$ jumps to get
\[
\Re( g_n(\tau'_0 + i\upsilon) ) ~\le~ J ( C + \log\abs{\upsilon} + C' )  \hspace{20pt} \text{for } \abs{\upsilon} > \delta
\]
and exponentiate to get
\[
\abs{\phi_n( \tau'_0 + i\upsilon)} ~\le~ \exp(J(C+C')) \abs{\upsilon}^J  \hspace{20pt} \text{for } \abs{\upsilon} > \delta
\]
and this is claim 2.

Finally, from claims 1 and 2 we conclude that
\[
\abs{ \phi_n(\tau_0 + i\upsilon) }  \abs{ f(\tau_0 + i\upsilon) }^{n_0} \le K_1 K_2 \abs{\upsilon}^{\beta - n_0 \alpha}   \hspace{20pt} \text{for } \abs{\upsilon} \ge \delta
\]
Pick $n_0$ so $\beta -  n_0 \alpha < -m$ or equivalently $n_0 > (\beta + m)/\alpha$.
Then
\[
\int_{\abs{\upsilon} \ge \delta} \abs{ \phi_n(\tau_0 + i\upsilon) f(\tau_0 + i\upsilon)^{n_0} } \, d\upsilon ~\le~ K_1 K_2  \int_{\abs{\upsilon} \ge \delta} \abs{\upsilon}^{\beta -  n_0 \alpha} \,  d\upsilon
\]
Since $\beta -  n_0 \alpha < -m$, the value of the right side is finite; call it $M_\delta$.
We have already seen in part $(d)$ that there is a $C_1$ so $1 \le C_1 \abs{\phi_n(\tau_0)}$ and so
\[
\int_{\abs{\upsilon} \ge \delta} \abs{ \phi_n(\tau_0 + i\upsilon) f(\tau_0 + i\upsilon)^{n_0} } \, d\upsilon ~\le~ M_\delta C_1 \abs{\phi_n(\tau_0)}  \hspace{15pt} \text{for } \abs{\upsilon} \ge \delta \text{ and } n \text{ sufficiently large}
\]
and this is part $(e)$.
\end{proof}

\begin{proof}\label{proof:main} of the Main Theorem.
Pick a fixed $x_0 \in [0,1]$ which will remain fixed until near the end of this proof.
For each $n$ let $j_n$ be the index of $I_{j,n} \ni x_0$; the dependence of $j_n$ on $x_0$ is suppressed for simplicity.
Rewrite $\phi_n(z)$ slightly for this $x_0$.
In equation \ref{eqn:phi_n} in the exponent of $\phi_n(z)$ is a sum $j=1,\ldots,n$. 
Split off this one $j_n$ from the sum to get
\[
\phi_n(z) = \exp \left[ K( \langle w_{j_n,n},z \rangle ) ~+~ \sum_{j \neq j_n} K( \langle w_{j,n},z \rangle ) - n \int_0^1 K( \langle \mathbf{w}(x),z\rangle ) \, dx  \right ]
\]
or
\[
\phi_n(z) = P( -\langle z,w_{j_n,n}\rangle ) \exp \left[ \sum_{j \neq j_n} K( \langle w_{j,n},z \rangle ) - n \int_0^1 K( \langle \mathbf{w}(x),z\rangle ) \, dx  \right ]
\]
The point $x_0$ does not appear explicity here, only implicitly in the variable $j_n$.

From the definition of the centroid in equation \ref{eq:centroid}
\begin{equation}\label{eqn:centroid_n}
\operatorname{centroid}(H_{y_0} \cap Q^n )[j_n] :=  \frac{ \operatorname{moment}( H_{y_0} \cap Q^n ) [j_n] }{ \operatorname{vol}( H_{y_0} \cap Q^n ) }
\end{equation}
where the notation $[j_n]$ denotes the $j_n$-coordinate of a vector.
Lemma \ref{lemma:volume} gives an asymptotic expression for the volume in the denominator.
Now we want to derive one for the $j_n$-coordinate of the moment in the numerator.
To compute the $j_n$-coordinate of the moment amounts to replacing $\rho(t) = \mathbf{1}_{[0,1]}$ by
$\rho(t) = t \mathbf{1}_{[0,1]}$ at the $j_n$-slot in the product \ref{volumeintegral}.
This change propagates into Corollary \ref{cor:volume} where in the $j_n$-slot of this product,
$P$ is replaced by $\widehat{P}$ := $\mathcal{B} \{ t \mathbf{1}_{[0,1]} \} = -\mathcal{B} \{ \mathbf{1}_{[0,1]} \}' = -P'$.
This last step is the standard \emph{frequency-domain derivative rule} for $\mathcal{B}$.

This change in the $j_n$-slot propagates into the sequence of functions $\phi_n(z)$ to create a new sequence $\widehat{\phi}_n(z)$ given by
\[
\widehat{\phi}_n(z) = -P'( -\langle z,w_{j_n,n}\rangle ) \exp \left[ \sum_{j \neq j_n} K( \langle w_{j,n},z \rangle ) - n \int_0^1 K( \langle \mathbf{w}(x),z\rangle ) \, dx  \right ]
\]
\emph{If} we can show that $\widehat{\phi}_n(z)$ and $f(z)$ \emph{also} satisfies the conditions for the Laplace Approximation,
then it will follow from \ref{eqn:centroid_n} and Corollary \ref{cor:LaplaceMethod} that
\begin{equation}\label{eqn:Kprime}
\operatorname{centroid}(  H_{y_0} \cap Q^n )[j_n] \sim \frac{ \widehat{\phi}_n(\tau_0) }{ \phi_n(\tau_0) } = \frac{-P'( -\langle \tau_0,w_{j_n,n} \rangle ) }{P( -\langle \tau_0,w_{j_n,n} \rangle)} = K'(\langle \tau_0,w_{j_n,n} \rangle) \hspace{5pt} (n{\to}\infty)
\end{equation}
The last step is important, and follows directly from the definition $K(t) := \log(P(-t)$.

For parts $(a)$ and $(b)$ of the Approximation there is nothing to prove, since $f(z)$ is unchanged.

However $\phi_n(z)$ is modified. 
Write the modified $\widehat{\phi}_n(z) = \exp \left[ \widehat{g}_n(z) \right]$ where
\[
\widehat{g}_n(z) :=  \left( \widehat{K}( \langle w_{j_n,n}, z \rangle ) - n \int_{I_{j_n,n}} J(z,x) \, dx  \right) ~+~   \sum_{j \neq j_n}  \left( K( \langle w_{j,n}, z \rangle ) - n \int_{I_{j,n}} J(z,x) \, dx  \right)
\]
and where $J(z,x) := K( \langle z, \mathbf{w}(x) \rangle )$.
Compare this with \ref{eqn:g_n}.
The only difference is that in the very first term $K(s)$ is replaced by $\widehat{K}(s)$.
Recall from \emph{Remark 2} after Lemma \ref{lemma:Kr} that
$\widehat{K}(s) :=  \log ( \widehat{P}(-s) ) = \log ( -P'(-s) ) = \log( ((s-1)e^s + 1)/s^2 )$,
and $\widehat{K}_r(s):=\Re(\widehat{K}(s))$.

We can now see that part $(c)$ is trivial.

For part $(d)$, we see that $\widehat{g}_n(z)$ now has an extra term from $I_{j_n,n}$
\begin{equation*}
\widehat{K}( \lambda_1 \langle y_2,z \rangle  +  \lambda_2 \langle y_1,z \rangle)  ~-~ \lambda_1 K( \langle y_2,z \rangle ) - \lambda_2 K( \langle y_1,z \rangle )
\end{equation*}
which is nonzero whether there is a jump in interval $I_{j_n,n}$ or not.
Compare this term with equation \ref{eqn:jumpd}.
Since $\widehat{K}$ is analytic on the strip $\{ \tau +i\upsilon : \abs{\upsilon}{<}2\pi \}$
the same argument for part $(d)$ in Lemma \ref{lemma:volume} works here too.

For part $(e)$, claim 1 is about $f(z)$ which is unchanged, so there is nothing to do.
For part $(e)$ claim 2, the new term in $\widehat{g}_n(\tau'_0 + i\upsilon)$ is now
\begin{equation*} 
\widehat{K}_r( \langle \bar{y},\tau'_0 \rangle + i\langle \bar{y},\upsilon \rangle) - \lambda_1 K_r(\langle y_2,\tau'_0 \rangle + i\langle y_2,\upsilon \rangle) - \lambda_2 K_r(\langle y_1,\tau'_0 \rangle + i\langle y_1,\upsilon \rangle)
\end{equation*}
which is the same as equation \ref{eqn:claim2A} except for the $\widehat{K}$.
Once again, by \emph{Remark 2}, we know that 
$\widehat{K}_r( \langle \bar{y},\tau'_0 \rangle + i\langle \bar{y},\upsilon \rangle) \le \widehat{K}_r( \langle \bar{y},\tau'_0 \rangle )$
so the rest of claim 2 is the same as in the proof of of Lemma \ref{lemma:volume}.
This concludes part $(e)$, and the Laplace Approximation with $\widehat{\phi}_n(z)$ is justified.

Equation \ref{eqn:Kprime} is now verified and we can write
\[
\operatorname{centroid}(  H_{y_0} \cap Q^n )[j_n] ~\sim~  K'(\langle \tau_0,w_{j_n,n} \rangle) \hspace{25pt} (n \to \infty)
\]
This asymptotic equivalence holds for all $x_0$ and we think of both sides as functions of $x_0$ even though it does not appear explicity,
but only implicitly in the index $j_n$.
As $n \to \infty$ the intervals $I_{j_n,n}$ ``shrink nicely" to $x_0$,
and so $w_{j_n,n}$ converges to $\mathbf{w}(x_0)$ for almost all $x_0 \in [0,1]$.
This is the Lebesgue differentiation theorem, see \cite{Rudin} Theorem 8.8.
Actually, since $\mathbf{w}$ is a step function we do not need the full power of this theorem;
the sequence converges everywhere except at the jumps.
Since $K'(t)$ is continuous,
$K'( \langle \tau_0,w_{j_n,n} \rangle  )$ converges to $K'( \langle \tau_0,\mathbf{w}(x_0) \rangle )$ for almost all $x_0 \in [0,1]$.
By the asymptotic equivalence, $\operatorname{centroid}(  H_{y_0} \cap Q^n )[j_n]$ converges to $K'( \langle \tau_0,\mathbf{w}(x_0) \rangle )$ for almost all $x_0 \in [0,1]$.
Since $\abs{K'(t)}<1$ the Lebesgue Dominated Convergence theorem implies that $\operatorname{centroid}(  H_{y_0} \cap Q^n )[j_n]$
converges in $L^1[0,1]$ to $K'(\langle \tau_0,\mathbf{w} \rangle)$.
This is the expression for the centroid in the Main Theorem \ref{thm:main}, and so we are done.
\end{proof}

The saddlepoint equation \ref{eqn:saddlepoint} 
can be written in the form $\Lambda_\mathbf{w}( \sigma( \langle \tau_0, \mathbf{w} \rangle ) ) = y_0$,
which confirms the fact that the centroid of a convex set must lie in that set - a satisfying result.

In Figure \ref{fig:estimates-single} is an example when $m{=}1$.
In this case $\mathbf{w}{=}w_1{=}w$.
\begin{figure}[!ht]
    \centering
    \subfloat{\includegraphics[width=1.0\linewidth]{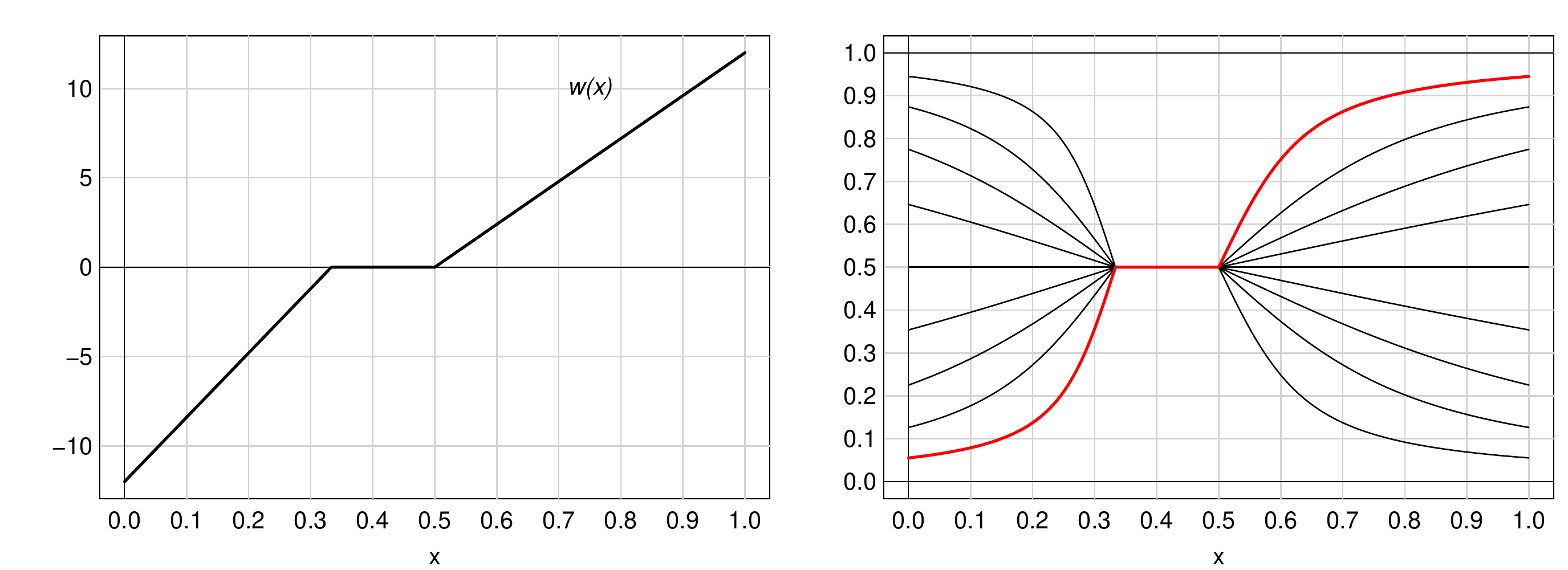}}    
    \caption{An example of centroids when $m{=}1$.
             In the left plot is $w(x)$, with corners at $1/3$ and $1/2$. 
             For this $w(x)$, $Z := [\int_0^1 \min(w(x),0)\,dx , \int_0^1 \max(w(x),0)\,dx ] = [-2,3]$. 
             The right plot shows the computed centroids at 9 equally spaced $y_0 = -3/2,-1,-1/2,\ldots,5/2$.
             For the last value $(y_0{=}5/2)$ the centroid is $f(x) \approx \sigma(1.51 w(x))$ and is plotted in red with a thicker line.
             For the center value $(y_0{=}1/2)$, the centroid is the constant function $f(x) \equiv 1/2$, which is also the centroid of $Q^\infty$.
             }
    \label{fig:estimates-single}
\end{figure}
Note that if $w(x)$ vanishes on a subinterval, 
then the centroid function takes the value $\sigma(0){=}1/2$ on that interval, as it should be.

\section{Reparameterization}\label{sec:reparameterization}

Operators frequently occur for functions on intervals other than [0,1].
In this section we consider an interval $[a,b]$,
and an essentially bounded $\mathbf{w} : [a,b] \to \mathbb{R}^m$.
Parameterize $[a,b]$ by $\lambda$ (think of wavelength), and $[0,1]$ by $\omega$.
$\mathbf{w}$ defines an operator $\Lambda_\mathbf{w} : L^1[a,b] \to \mathbb{R}^m$ given by $\Lambda_\mathbf{w}(f) = \int_a^b f(\lambda) \mathbf{w}(\lambda) \, d\lambda$.
For a given $y \in  \mathbb{R}^m$ we want to solve the operator equation $\Lambda_\mathbf{w}(f) = y$ for $f$ in the unit cube $Q_{[a,b]}^\infty$ of $L^1[a,b]$.

Consider an absolutely continuous homeomorphism $\varphi : [0,1] \to [a,b]$.
For simplicity abbreviate `absolutely continuous' by AC.
To parameterize [0,1] we use $\omega$ and so we write $\lambda = \varphi(\omega)$.
$\varphi$ induces isomorphisms $C_\varphi : L^1[a,b] \to L^1[0,1]$ and  $C_{\varphi^{-1}} : L^1[0,1] \to L^1[a,b]$
(both are \emph{composition operators}).
Note that these take the unit cube in one function space to the unit cube in the other.
The following composition is an operator from $L^1[0,1]$ to $\mathbb{R}^m$
\begin{equation}\label{eqn:reparam}
L^1[0,1] ~  {\overset{C_{\varphi^{-1}}}\longrightarrow}  ~  L^1[a,b] ~ {\overset{\Lambda_\mathbf{w}}\twoheadrightarrow}  ~ \mathbb{R}^m
\end{equation}
With the feasible region the unit cube,
we can solve it for $y$ by the centroid method, and then map the centroid (with respect to $\mathapprox{P}$)
from $L^1[0,1]$ to $L^1[a,b]$ by $C_{\varphi^{-1}}$,
which will be a solution to the original problem.
In symbols the ``reparameterized" solution is 
\begin{equation}\label{eqn:repapp}
C_{\varphi^{-1}} \left( \operatorname{centroid}_\mathapprox{P} \left( (\Lambda_\mathbf{w} \circ  C_{\varphi^{-1}})^{-1} (y) \cap Q^\infty \right) \right) =
\operatorname{centroid}_{C_{\varphi^{-1}}(\mathapprox{P})} \left( \Lambda_\mathbf{w}^{-1} (y) \cap Q_{[a,b]}^\infty \right)
\end{equation}
The right side follows from Proposition \ref{prop:infinite},
and gives an alternate interpretation.
The ``reparameterized" solution is the centroid with respect to the directed filtration $C_{\varphi^{-1}}(\mathapprox{P})$  of $L^1[a,b]$.
Recall that $V_n \in \mathapprox{P}$ is the subspace of step functions with jumps at $i/n, ~ i{=}1,\ldots,(n{-}1)$.
So $C_{\varphi^{-1}}(V_n)$ is the subspace of step functions with jumps at $\varphi(i/n), ~ i{=}1,\ldots,(n{-}1)$,
which are \emph{not} equally spaced in general.
The filtration $C_{\varphi^{-1}}(\mathapprox{P})$ is not uniform unless $\varphi$ is affine.

We know that the composition \ref{eqn:reparam} must be defined by some responsivity function on [0,1].
Here is the change of variable formula that we need:
\begin{prop}
Operator \ref{eqn:reparam} has the responsivity function $\omega \longmapsto \mathbf{w}(\varphi(\omega)) \varphi'(\omega)$.
\end{prop}
Note that because $\varphi$ is AC, $\varphi'(\omega)$ exists a.e. and is integrable.
\begin{proof}
Let $\mathbf{v}{:}[a,b] \to \mathbb{R}^m$ be any integrable function.
By the change of variables formula, \cite{Natanson} vol II page 236
\[
\int_a^b \mathbf{v}(\lambda) \, d\lambda ~=~ \int_0^1 \mathbf{v}( \varphi(\omega) ) \varphi'(\omega) \, d\omega
\]
For $f \in L^1[0,1]$ substitute $\mathbf{v}(\lambda) := f( \varphi^{-1} (\lambda) ) \mathbf{w}(\lambda)$ to get
\[
\int_a^b f( \varphi^{-1} (\lambda) ) \mathbf{w}(\lambda) \, d\lambda ~=~ \int_0^1 f(\omega) \mathbf{w}( \varphi(\omega) ) \varphi'(\omega) \, d\omega
\]
The left side is the value of $f$ under operator \ref{eqn:reparam},
and the right side shows the claimed responsivity function.
\end{proof}

Write $\mathbf{w}(\lambda) = ( w_1(\lambda), \ldots,  w_m(\lambda) )$.
Define $\widehat{w}_i(\omega) := w_i(\varphi(\omega)) \varphi'(\omega), ~ i=1,\ldots,m$  and call the  
$\widehat{w}_i$ the \emph{reparameterized responsivities}, with the function $\varphi(\omega)$ understood by context.
In the next section on applications, we want to find a $\varphi(\omega)$ so that
a linear combination of the reparameterized responsivities $\widehat{w}_1,\ldots,\widehat{w}_m$
is a positive constant.
The next proposition give a sufficient condition for this,
and a formula for $\varphi(\omega)$.
\begin{prop}\label{prop:reparam}
For $\mathbf{w} : [a,b] \to \mathbb{R}^m$, suppose there is a positive linear combination
\[
\alpha_1 w_1(\lambda) + \ldots +  \alpha_m w_m(\lambda) > 0    \hspace{25pt} \text{almost all } \lambda \in [a,b]
\]
Then there is an AC homemomorphism $\varphi(\omega)$ and a $C > 0$ so that
\[
\left[ \alpha_1 w_1(\varphi(\omega)) + \ldots +  \alpha_m w_m(\varphi(\omega)) \right] \varphi'(\omega) = 
\alpha_1 \widehat{w}_1(\omega) + \ldots + \alpha_m \widehat{w}_m(\omega) = C           \hspace{10pt} \text{almost all } \omega
\]
\end{prop}
\begin{proof} Define
\begin{equation*}
S(\lambda) := \alpha_1 w_1(\lambda) + \ldots +  \alpha_m w_m(\lambda) \hspace{10pt}  \hspace{10pt}   C := \int_a^b S(\gamma) \, d\gamma \\
\hspace{15pt}   \theta(\lambda) := C^{-1} \int_a^\lambda S(\gamma) \, d\gamma
\end{equation*}
Since $S(\lambda){>}0$ a.e., $\theta(\lambda):[a,b] \to [0,1]$ is strictly increasing and $\theta'(\lambda){>}0$ a.e.
Define $\varphi := \theta^{-1}:[0,1] \to [a,b]$; then $\varphi$ is AC by \cite{Natanson}, vol I exercise IX.13.
Substitute $\lambda := \varphi(\omega)$ and multiply by $C$
\[
C \omega = \int_a^{\varphi(\omega)} S(\gamma)   \, d\gamma
\]
and differentiate w.r.t. $\omega$
\begin{equation}\label{eqn:varphiprime}
C =  S(\varphi(\omega))   \varphi'(\omega)  =  \alpha_1 \widehat{w}_1(\omega) + \ldots + \alpha_m \widehat{w}_m(\omega) 
\end{equation}
and we are done.
\end{proof}
In case $\varphi(\omega)$ comes from the above procedure,
we call the $\widehat{w}_i$ the \emph{equalized responsivities}.

\medskip
\emph{Remark 1}. 
Suppose that $\mathbf{w}(\lambda)$ is a step function on $[a,b]$.
An examination of the recipe for $\varphi^{-1}(\lambda)$ in the previous proof shows that 
it has constant slope on the subintervals of $[a,b]$ defining $\mathbf{w}$.
Alternatively, it is piecewise-linear with corners at the jumps of $\mathbf{w}$.
It follows that all the equalized responsivities $w_i(\varphi(\omega)) \varphi'(\omega)$ 
are step functions on the corresponding subintervals of $[0,1]$.

\emph{Remark 2}. 
The proposition has a geometric interpretation:
if the image of the curve $\mathbf{w} : [a,b] \to \mathbb{R}^m$
lies in in an open halfspace $\mathring{H}$ with $0 \in \partial\mathring{H}$,
then there is a $\varphi(\omega)$ so that the curve
$\widehat{\mathbf{w}} : [0,1] \to \mathbb{R}^m$ is in a hyperplane inside $\mathring{H}$ and parallel to $\partial\mathring{H}$.

\medskip
Setting all the $\alpha$'s to 1 in Proposition \ref{prop:reparam} gives
\begin{corollary}
For $\mathbf{w} : [a,b] \to \mathbb{R}^m$, suppose that
\[
w_1(\lambda) + \ldots +  w_m(\lambda) > 0    \hspace{25pt} \text{all } \lambda \in [a,b]
\]
Then there is an AC homemomorphism $\varphi(\omega)$ and a $C > 0$ so that
\[
\widehat{w}_1(\omega) + \ldots +  \widehat{w}_m(\omega)  = C    \hspace{25pt} \text{all } \omega \in [0,1]
\]
\end{corollary}

For computing the centroid with equalized responsivities, 
there is a computational shortcut that works entirely with $\lambda$
and avoids explict calculation of $\varphi(\omega)$.
Define $S(\lambda) := \alpha_1 w_1(\lambda) + \ldots +  \alpha_m w_m(\lambda) > 0$
and $\widetilde{w}_i(\lambda) := w_i(\lambda) / S(\lambda)$.
We call $\widetilde{w}_i(\lambda)$ the \emph{normalized responsivities}.
From \ref{eqn:varphiprime} it follows that $\widehat{w}_i(\omega) := w_i(\varphi(\omega)) \varphi'(\omega) =  C \widetilde{w}_i(\varphi(\omega))$.
For $y \in \operatorname{int}( \Lambda_\mathbf{w}(Q^\infty) )$ the saddlepoint equations in $\omega$ are
\begin{equation}\label{eqn:saddleomega}
\int_0^1 \sigma\left( \tau_1 \widehat{w}_1(\omega) + \ldots + \tau_m \widehat{w}_m(\omega) \right) \widehat{w}_i(\omega) \, d\omega ~=~ y_i  \hspace{20pt} i=1,\ldots,m
\end{equation}
where $\tau_1, \ldots ,\tau_m$ are the unknowns, and $\sigma(t)$ is the ``squashing function".
After finding the vector root $\tau_0 := (\tau_1, \ldots, \tau_m)$ of this system,
the spectral estimate (in $\omega$) is $\sigma()$ of this linear combination.
Inside $\sigma()$ replace equalized by normalized responsivities to get
\[
\int_0^1 \sigma\left( C \tau_1 \widetilde{w}_1(\varphi(\omega)) + \ldots + C \tau_m \widetilde{w}_m(\varphi(\omega)) \right) \widehat{w}_i(\omega) \, d\omega ~=~ y_i  \hspace{20pt} i=1,\ldots,m
\]
Dropping the $C$ from this system just scales the root $\tau_0$ and does not make any difference to the final estimate.
So drop $C$ and replace $\widehat{w}_i(\omega)$ by its definition.
\[
\int_0^1 \sigma\left(  \tau_1 \widetilde{w}_1(\varphi(\omega)) + \ldots +  \tau_m \widetilde{w}_m(\varphi(\omega)) \right) w_i(\varphi(\omega)) \varphi'(\omega) \, d\omega ~=~ y_i  \hspace{20pt} i=1,\ldots,m
\]
Apply the change of variable formula
\begin{equation}\label{eqn:shortcutsystem}
\int_a^b \sigma\left(  \tau_1 \widetilde{w}_1(\lambda) + \ldots +  \tau_m \widetilde{w}_m(\lambda) \right) w_i(\lambda)   \, d\lambda ~=~ y_i  \hspace{20pt} i=1,\ldots,m
\end{equation}
This is the system we use for numerical work in the next section, because it avoids explict calculation of $\varphi(\omega)$.
The system has a unique solution because system \ref{eqn:saddleomega}  has a unique solution by Theorem \ref{thm:diffeo}.
The spectral estimate is a linear combination of the normalized responsivities, and then squashed by $\sigma$.

\section{Numerical Results in Colorimetry}\label{sec:numerical}

In this section we apply the main result to spectral reflectance estimation in colorimetry.
A function $f \in Q^\infty$ corresponds to the spectral reflectance (or transmittance) function of a material.
For a reflectance to be physically feasible, it must be between 0 and 1.
The operator $\Lambda$ corresponds to a light source and biological eye (or an electronic camera)
that responds linearly to light reflected from (or transmitted through) the material.
The vector $\Lambda(f)$ is the 3-channel response to $f$.
If $\Lambda(f_1) = \Lambda(f_2)$ then $f_1$ and $f_2$ are called \emph{metameric} for that light source and eye.
The set $\body := \Lambda(Q^\infty)$ is called the \emph{object-color solid} or \emph{Farbk{\"o}rper} for $\Lambda$, 
see \cite{Wyszecki1982} and Logvinenko \cite{Logvinenko2009}.
The standard involution of $\body$ corresponds to taking \emph{complementary colors}, see \cite{Logvinenko2009}.
The unique fixed point of the involution is 50\% neutral gray.
Points on $\partial \body$ correspond to \emph{optimal colors} 
or \emph{Optimalfarben}, see Schr{\"o}dinger \cite{ANDP:ANDP19203671504} p. 616.
The set $\body$ is a zonoid, and since we view the responsivities as step functions, it is a \emph{zonohedron},
i.e. a zonoid that is also a polyhedron.
See Centore \cite{Centore2013} for details, especially on transitions and calculations.
We allow the responsivities and the response to be negative;
this is not the case for biological eyes, but is useful for ``idealized cameras"
such as the \emph{taking characteristics} of the BT.709 RGB primaries, see \cite{poynton} p. 298.
But we \emph{do} assume that some linear combination of the responsivities is positive, 
as in Proposition \ref{prop:reparam}.

From now on we specialize to the classical case of the human eye and the 1931 standard observer.
For an illuminant $I(\lambda)$ and a material with reflectance $r(\lambda)$ the response is
the 3-vector of \emph{tristimulus values XYZ} given by
\begin{equation*}
XYZ := \int_a^b I(\lambda) r(\lambda) ~(~ \bar{x}(\lambda),\bar{y}(\lambda),\bar{z}(\lambda) ~)~ \, d\lambda
\end{equation*}
where $\bar{x}(\lambda)$, $\bar{y}(\lambda)$, and $\bar{z}(\lambda)$ are the CIE 1931 standard observer responsivities.
Given a response vector $XYZ$, the set $Q^\infty \cap \Lambda_\mathbf{w}^{-1}(XYZ)$ is the set of \emph{metamers} for $XYZ$.
The set is often called the \emph{metameric suite} for $XYZ$ with illuminant $I(\lambda)$, see \cite{cohen2001} and \cite{koenderink}.

For simplicity we use the equal energy Illuminant E so  $I(\lambda) \equiv 1$ in this section.
The function defining $\Lambda_\mathbf{w} : L^1[a,b] \to \mathbb{R}^3$ is
$\mathbf{w}(\lambda) :=  (\bar{x}(\lambda),\bar{y}(\lambda),\bar{z}(\lambda))$,
see Figure \ref{fig:reparam} \emph{(a)}.
In this section the wavelength interval $[a,b]$ is taken to be $[400,700]$ nm, 
although $\mathbf{w}$ is tabulated on a larger interval.
The calculation of the centroid of the metameric suite (which we think of as an ``average") is viewed
as a solution to the problem of \emph{spectral estimation} of the material reflectance $r(\lambda)$, 
given its response $XYZ$.
This problem is a part of ``inverse colorimetry" in Koenderink \cite{koenderink}, sections 13.1.3 and 13.4.4.
Many other methods have been studied.
Hawkyard \cite{Hawkyard1993} uses a special weighted average of the CIE color matching functions.
Murakami et. al. \cite{Murakami:02} use a Gaussian mixture of Wiener estimates.
DiCarlo and Wandell \cite{DiCarlo:03} use submanifold estimation methods, which are inherently non-linear.
Heikkinen et. al. \cite{Heikkinen:07, Heikkinen:08, Heikkinen:13} use regularization and PCA for dimensionality reduction
in a ``kernel machine" framework, and have used a non-linear transformation to enforce feasible reflectance values.
Bianco \cite{bianco2010} uses ICA for dimensionality reduction and minimizes a customized objective function with four terms.
Jakob and Hanika \cite{Jakob2019} use quadratic polynomials followed by an algebraic squashing function.
Burns \cite{Burns2015} uses a least-squares minimization with constraints and with squashing function $(\tanh(t) + 1)/2$ (the LHTSS method),
which is solved using Lagrange multipliers.

Our ``centroid method" is closest in spirit to \cite{Hawkyard1993} and \cite{Jakob2019};
in fact it combines elements of both.
Both the centroid and the Hawyard methods use a linear combination of 3 responsivity functions and solve a system of 3 equations in 3 unknowns,
and neither uses dimensionality reduction or a training set.
One could say that the centroid method is the Hawkyard method with the addition of a squashing function;
a more colorful description might be that it is ``Hawkyard with a squash".
The centroid method has these features:
\begin{enumerate}
\item no training set is required
\item no special numerical implementation, a generic $m$-dimensional root finder works
\item the spectral estimate has the exact desired response $y_0$ (up to numerical precision)
\item the spectral estimate works even for responses near the optimal color boundary (up to numerical stability)
\item the calculated reflectance values are always feasible, i.e. inside [0,1]; this property is intrinsic to the method
\end{enumerate}
On the other hand, irregularities in the spectrum of the light source will appear in the spectral estimate,
so it may not be well suited for fluorescent light sources.

For (2), our software implementation uses the root finder in \textbf{rootSolve} \cite{rootSolve},
which uses Newton-Raphson iteration.
As mentioned in the Remark after \ref{lemma:PD}, 
one could also use an $m$-dimensional function minimizer, or even a combination of the two.

\begin{figure}[!ht]
    \centering
    \subfloat{\includegraphics[width=1.0\linewidth]{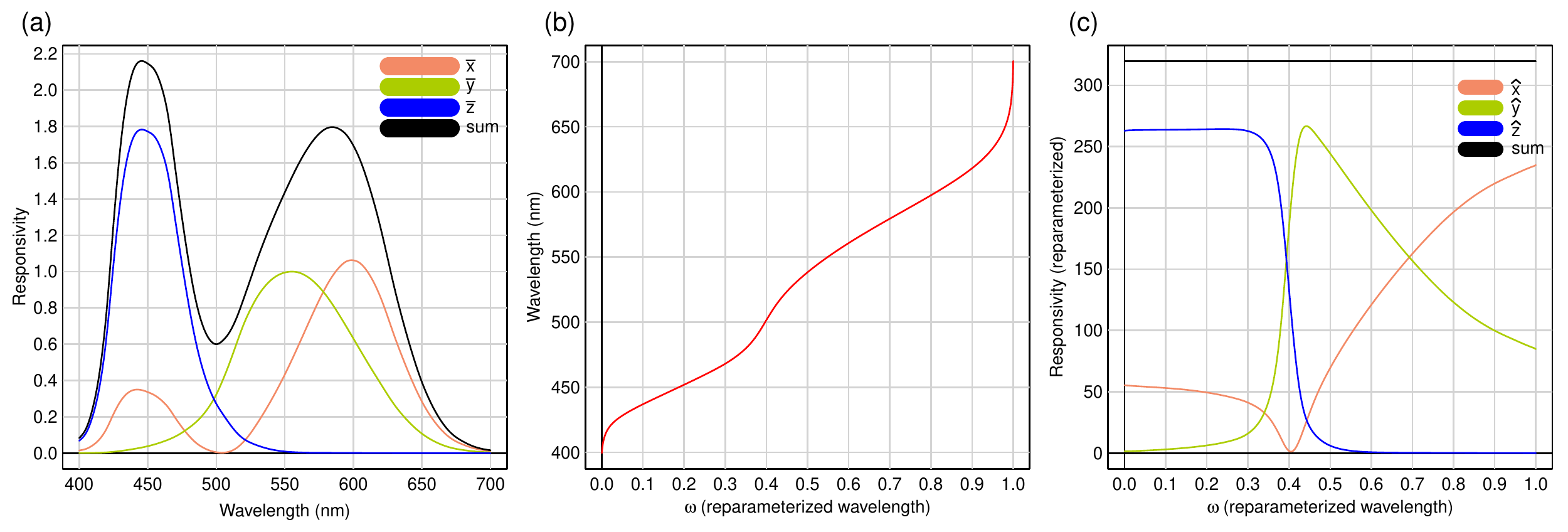}}    
    \caption{\emph{(a)} The CIE 1931 standard observer responsivities ($w_1{=}\bar{x},w_2{=}\bar{y},w_3{=}\bar{z}$) vs. wavelength, along with their sum.
             \emph{(b)} The reparameterization $\lambda{=}\varphi(\omega)$ from $\omega \in [0,1]  \text{ to } \lambda \in [400,700]$.
             The reparameterization is chosen so the sum of the three new responsivities is constant.
             \emph{(c)} The three new \emph{equalized} responsivities, along with their sum ($\sim$319.8).
             }
    \label{fig:reparam}
\end{figure}

If an object has a constant reflectance (a perfectly neutral spectral gray)
and the estimated reflectance is equal to the true reflectance (up to numerical tolerance),
we say that the estimator is \emph{neutral-exact}.
This is a desirable property for an estimator.
A fancy way to say this: 
the right-inverse $\gamma$ from equation \ref{eqn:inverse}
maps the line segment of neutral grays (the diagonal of the zonoid $\body$)
to the line segment of constant spectra (the diagonal of the cube $Q^\infty$).
The centroid method is neutral-exact iff  there is a linear combination of the responsivities
$w_1$, $w_2$, and $w_3$ that is a non-zero constant $C$.
For $\bar{x}$, $\bar{y}$, and $\bar{z}$ this far from true,
and so, following Logvinenko \cite{Logvinenko2009}, we reparametrize the interval [400,700]
with a function $\varphi : [0,1] \to [400,700]$
so that the sum of the new responsivities is constant, see Figures \ref{fig:reparam} \emph{(b)} and \emph{(c)}.
This new reparameterized $\Lambda$ in $(c)$ is now neutral-exact,
and we say the the responsivities are \emph{equalized}.

\begin{figure}[!ht]
    \centering
    \subfloat{\includegraphics[width=1.0\linewidth]{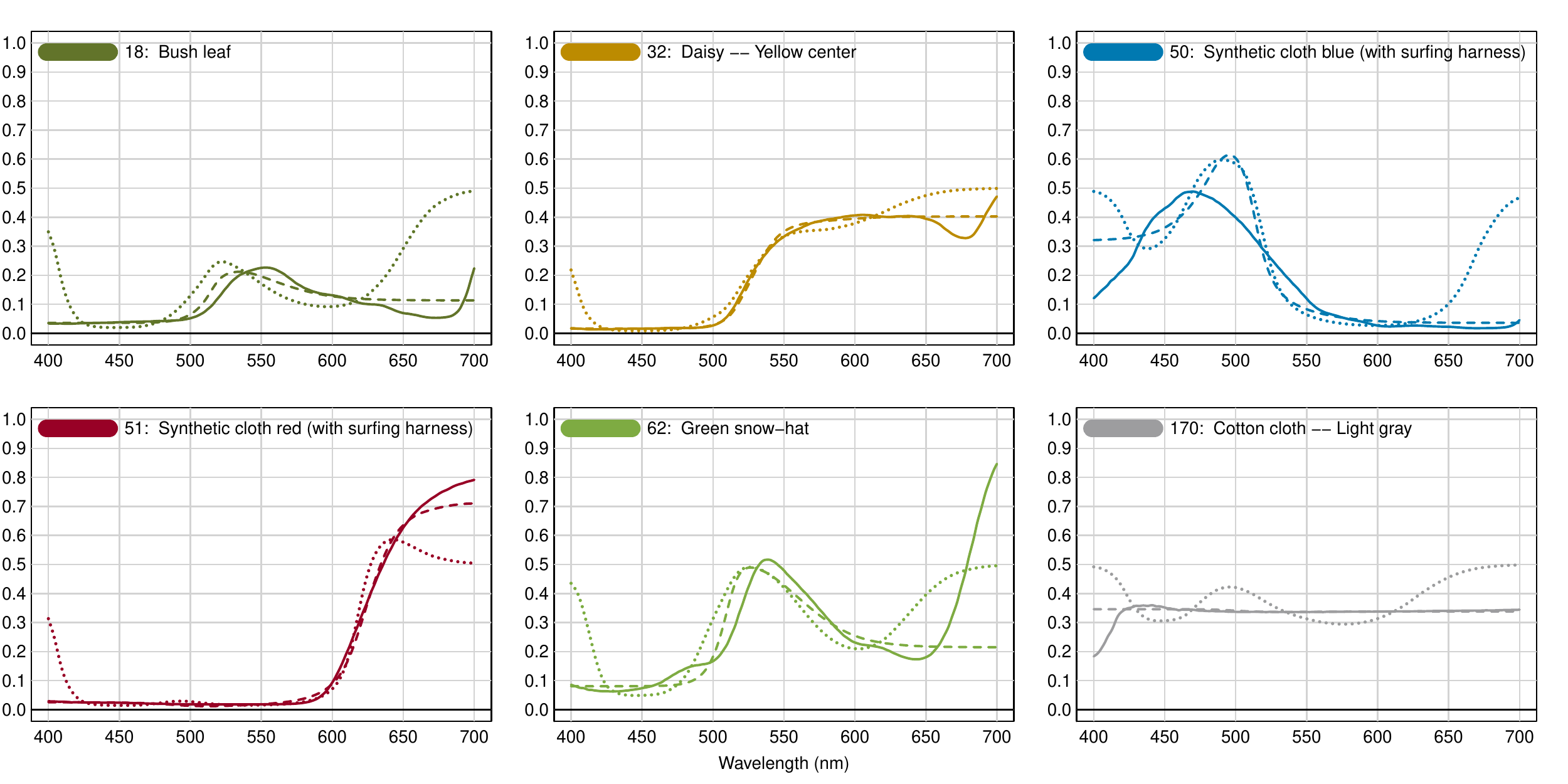}}    
    \caption{Six spectra from the NCSU dataset of 170 spectra.  The vertical axis is reflectance.
             In each plot, the original spectrum is drawn with solid linestyle.
             The estimated spectrum using the centroid method, and with the original responsivities in Figure \ref{fig:reparam} \emph{(a)}, is drawn with dotted linestyle.
             The estimated spectrum using the same method, but with the equalized responsivities in \ref{fig:reparam} \emph{(c)}, is drawn with dashed linestyle.
             Compare with Koenderink \cite{koenderink} Figure 13.8, p. 555.
             In each plot, all 3 spectra are metameric under Illuminant E.
             }
    \label{fig:koenderink6}
\end{figure}

For testing an estimator, it is standard practice to take the reflectance spectrum of a material,
compute its $XYZ$, estimate the spectrum from $XYZ$, and then compare the true and estimated spectra.
For example, Figure \ref{fig:koenderink6} shows the true and estimated spectra for 6 spectra
from the NCSU dataset of 170 reflectance spectra.
Each figure shows both the unequalized and the equalized estimates.
Note that true spectrum of sample ``170: Cotton cloth - light gray" is quite flat, except below 425nm.
The estimate from the unequalized $\Lambda$ has undulations,
which are artifacts from the peaks and valleys in  $\bar{x}$, $\bar{y}$, and $\bar{z}$.
The ``neutral-exact" estimate is much flatter and more accurate.
In equation \ref{eqn:repapp} we saw that the estimate with equalized responsivities could be interpreted
as a centroid with a non-uniform directed filtration.
In these 6 examples, the original and equalized estimates are different,
so this is an example of two different filtrations (one uniform and one non-uniform)
of $L^1[400,700]$ that yield different centroids of the same set.

We turn now to the Hawkyard method \cite{Hawkyard1993}.
Once again we use illuminant E for simplicity.
As presented in \cite{bianco2010}, the responsivities are normalized like this (see Figure \ref{fig:hawkyard2} \emph{(b)}):
\begin{equation}
\widetilde{x} := \bar{x} / (\bar{x} + \bar{y} + \bar{z}) \hspace{30pt} \widetilde{y} := \bar{y} / (\bar{x} + \bar{y} + \bar{z})\hspace{30pt} \widetilde{z} := \bar{z} / (\bar{x} + \bar{y} + \bar{z})
\end{equation}
and we then seek a reflectance $r$ that is a linear combination
$r = \alpha_x \widetilde{x} + \alpha_y \widetilde{y} + \alpha_z \widetilde{z}$ for three unknown $\alpha$'s.
Assume there are 301 wavelengths, from 400 to 700 with a 1nm step.
In matrix form we seek a 301-vector $\mathbf{r} = \widetilde{W} \mathbf{\alpha}$
where $\widetilde{W}$ is a 301$\times$3 matrix whose columns are $\widetilde{x}$, $\widetilde{y}$, and $\widetilde{z}$,
and where $\mathbf{\alpha}$ is the unknown 3-vector.
Let $\overline{W}$ be the 301$\times$3 matrix whose columns are $\bar{x}$, $\bar{y}$, and $\bar{z}$.
For given 3-vector $XYZ$, the equation  to be satisfied is
\[
\overline{W}^T r = XYZ  \hspace{10pt}\implies\hspace{10pt} \overline{W}^T ( \widetilde{W} \mathbf{\alpha} ) = XYZ   \hspace{10pt}\implies\hspace{10pt}   \left( \overline{W}^T  \widetilde{W} \right) \mathbf{\alpha}  = XYZ
\]
and this is an trivial $3{\times}3$ matrix equation.
This normalized Hawkyard method has the neutral-exact property as well.
Note the similarity between the normalized responsivities in Figure \ref{fig:hawkyard2} \emph{(b)},
and the equalized responsivities in Figure \ref{fig:reparam} \emph{(c)}.
In fact, we saw in the previous section that they are the same, except for the reparameterization from $\omega$ to $\lambda$ and a scale factor.
The chief problem with the Hawkyard method is that it can generate reflectances outside [0,1].
In our calculations, we simply clamp the reflectance to [0,1], but this changes the resulting XYZ.
More sophisticated XYZ-preserving corrections have been implemented, see \cite{Zuffi2008}.
Because of the ``squashing function" the centroid method does not have this problem.

\begin{figure}[!ht]
    \centering
    \subfloat{\includegraphics[width=1.0\linewidth]{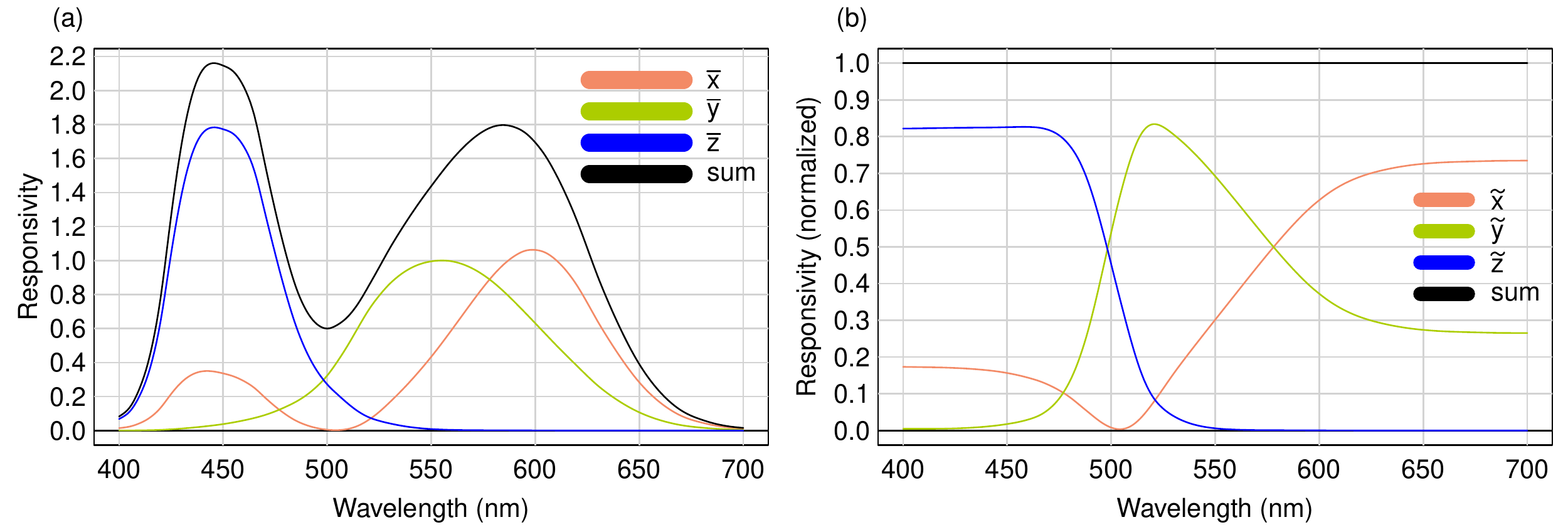}}
    \caption{\emph{(a)} The CIE 1931 standard observer responsivities ($w_1{=}\bar{x},w_2{=}\bar{y},w_3{=}\bar{z}$) vs. wavelength, along with their sum.
             \emph{(b)} The three normalized responsivities, along with their sum ($\equiv 1$), as used in the Hawkyard method.
             }
    \label{fig:hawkyard2}
\end{figure}

The two methods have been compared on the full NCSU dataset of 170 reflectance spectra.
For each spectrum the \emph{residual} spectrum is defined by
\[
residual := (estimated ~ spectrum) - (true ~ spectrum)
\]
and we split the residual into positive and negative parts by setting
\[
residual^+ := \max( residual, 0)  \hspace{30pt} \text{ and } \hspace{30pt} residual^- := \min( residual, 0)
\]
so that
\[
residual = residual^+  ~+~  residual^-  \hspace{20pt} \text{ and } \hspace{20pt}  \abs{residual} = residual^+  ~-~  residual^-
\]
The means of both $residual^+$ and $residual^-$ over all 170 objects,
and for both estimation methods,
are plotted in Figure \ref{fig:residuals} \emph{(a)}.
Note that for the interval from 650 to 700 nm, both estimates are mostly too low.
The `spread' between the means of $residual^+$ and $residual^-$  is the mean of $\abs{residual}$.
The difference between the accuracy of the two methods is not really significant,
but recall that some of the Hawkyard-estimated spectra do not have the exact desired XYZ
because of clamping.

\begin{figure}[!ht]
    \centering
    \subfloat{\includegraphics[width=1.0\linewidth]{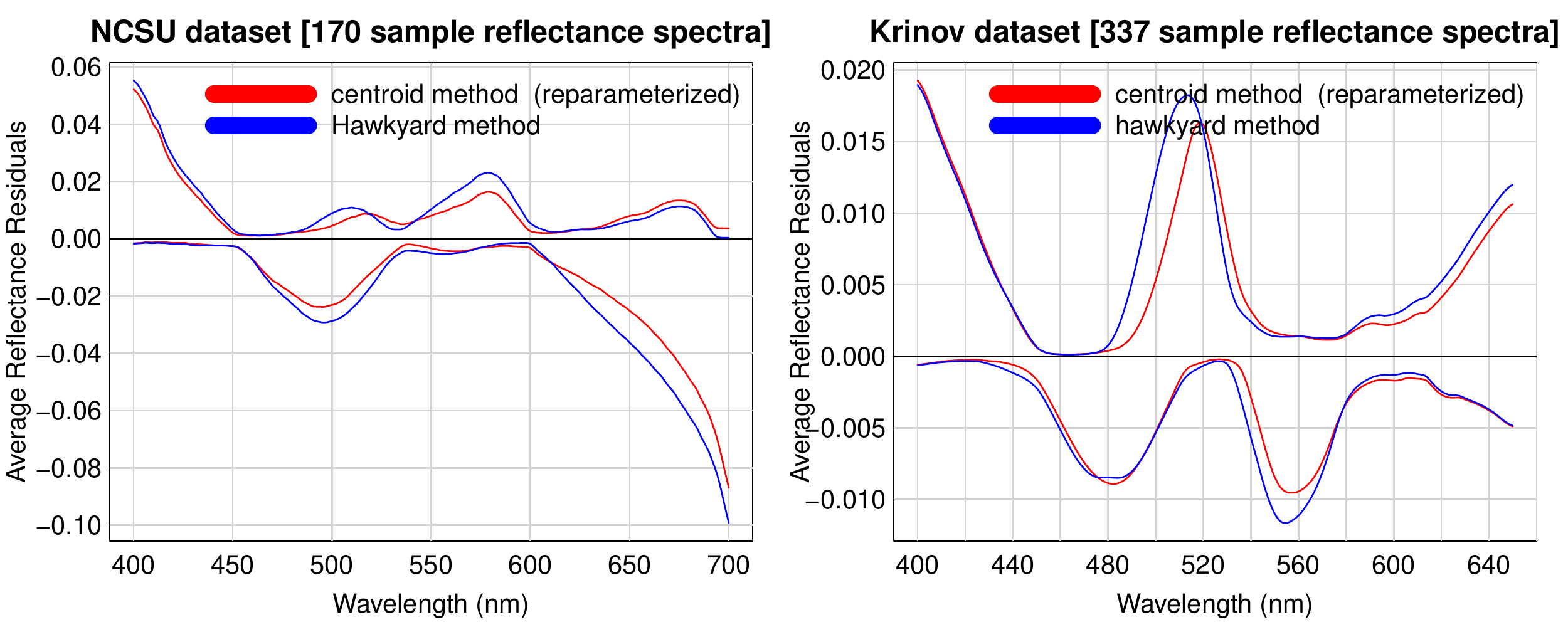}}    
    \caption{\emph{(left)} mean residuals for NCSU dataset of objects \cite{Vrhel}.
             The mean of $residual^+$ is above the x-axis, and the mean of $residual^-$ is below.
             \emph{(right)} mean residuals for the Krinov dataset of outdoor objects \cite{Krinov}.
             }
    \label{fig:residuals}
\end{figure}

The two methods are also compared on the Krinov dataset of 337 natural outdoor reflectance spectra.
These are plotted in Figure \ref{fig:residuals} \emph{(b)}.
Note that the scales in the two plots are not the same;
the residuals in the Krinov dataset are much smaller.
The reason is that the outdoor colors are less vivid than the NCSU dataset (which includes many man-made objects).
Once again the difference between the accuracy of the two methods is not really significant.
Compare these plots with Bianco \cite{bianco2010} Figure 6,
where the mean of $\abs{residual}$ is plotted for many methods.

One can create a different estimator by taking linear combinations of $\bar{x}$, $\bar{y}$, and $\bar{z}$,
to get a new basis for their span.
One popular change of basis is the Hunt-Pointer-Estevez transformation to responsivities $\bar{l}$, $\bar{m}$, and $\bar{s}$.
These are designed to approximate the responses of the long, medium, and short cones in the retina of the standard observer, 
see Hunt \cite{hunt2004}, page 598.
They are plotted in Figure \ref{fig:reparamHPE} \emph{(a)} along with the equalized
$\widehat{l}$, $\widehat{m}$, and $\widehat{s}$ in \emph{(c)}.
Compare this with Figure \ref{fig:reparam}.

\begin{figure}[!ht]
    \centering
    \subfloat{\includegraphics[width=1.0\linewidth]{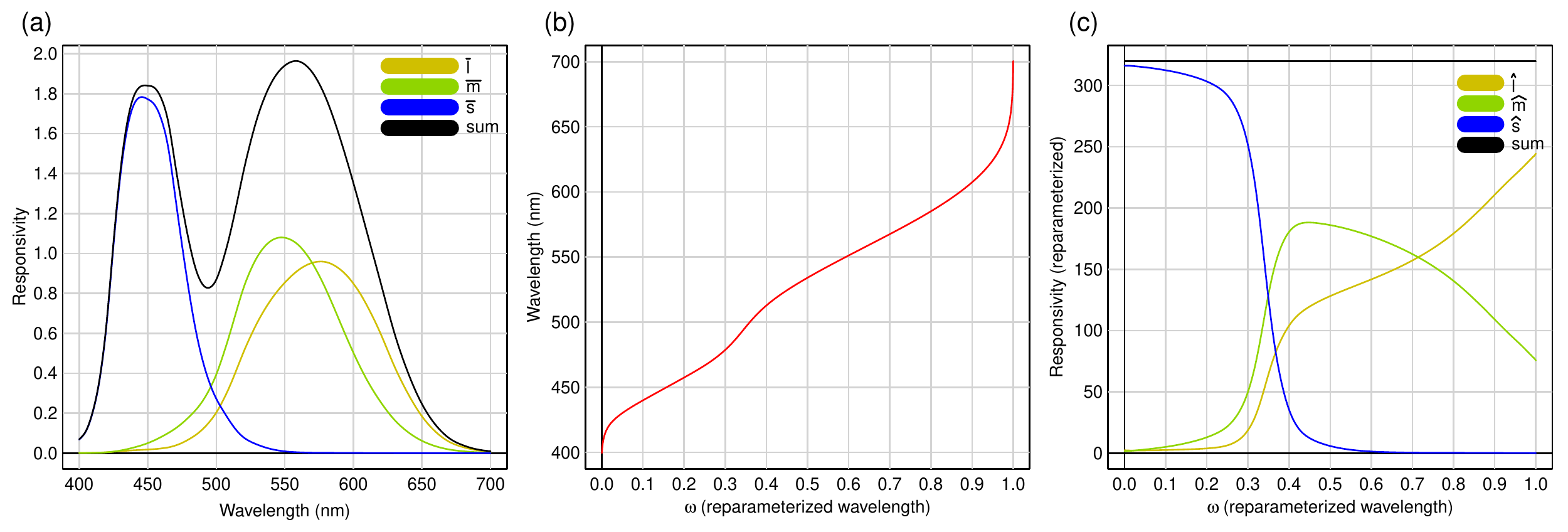}}    
    \caption{\emph{(a)} The  Hunt-Pointer-Estevez LMS responsivities ($w_1{=}\bar{l},w_2{=}\bar{m},w_3{=}\bar{s}$) vs. wavelength, along with their sum.
             \emph{(b)} The reparameterization $\lambda{=}\varphi(\omega)$ from $\omega \in [0,1]  \text{ to } \lambda \in [400,700]$.
             The reparameterization is chosen so the sum of the responsivities in \emph{(a)} is constant.
             \emph{(c)} The three \emph{equalized} responsivities $\widehat{l}$, $\widehat{m}$, and $\widehat{s}$, along with their sum.
             }
    \label{fig:reparamHPE}
\end{figure}

The responses XYZ and LMS are related by a $3{\times}3$ matrix.
The centroid method using XYZ and with LMS are compared in Figure \ref{fig:residualsHPE},
for both the NCSU and Krinov datasets.
Both estimators are neutral-exact.

\begin{figure}[!ht]
    \centering
    \subfloat{\includegraphics[width=1.0\linewidth]{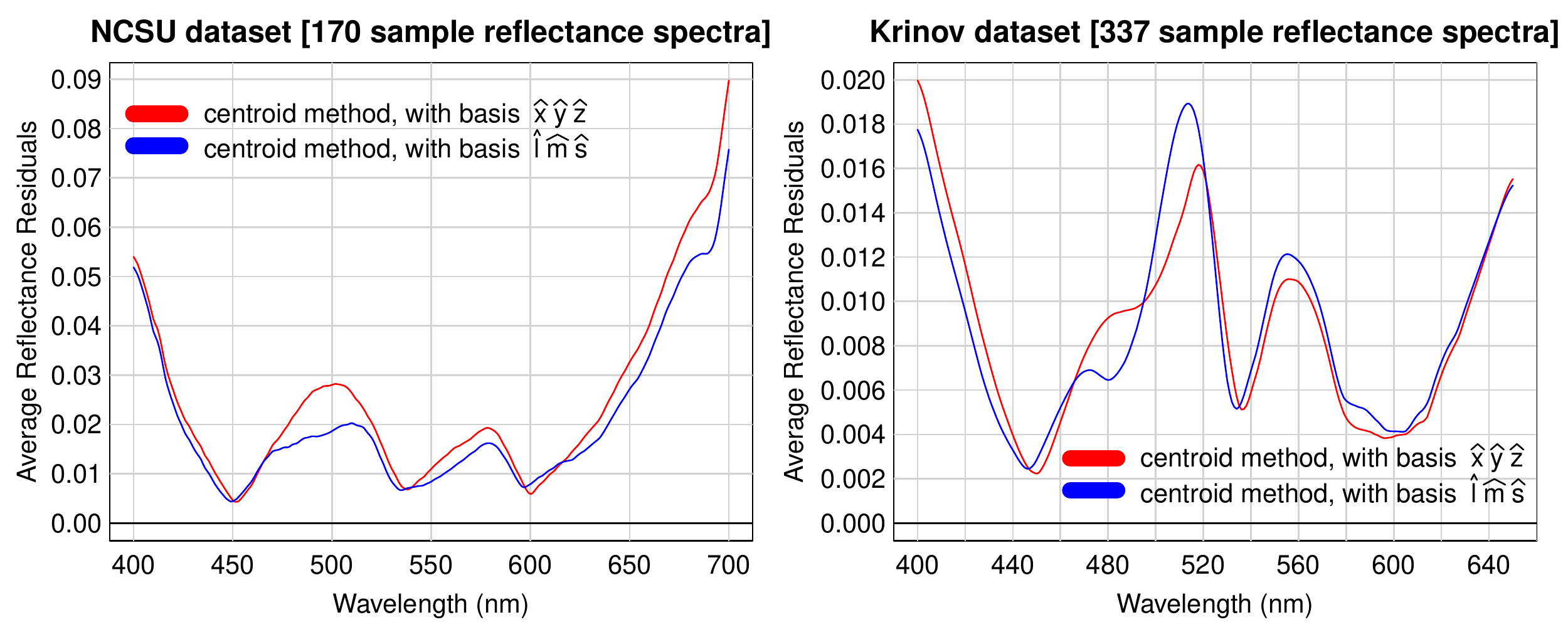}}    
    \caption{\emph{(left)} mean residuals for NCSU dataset of objects \cite{Vrhel}.
             These are the means of \abs{residual} for two estimators.
             \emph{(right)} means of \abs{residual} for the Krinov dataset of outdoor objects from \cite{Krinov}, and for the same 2 estimators.
             }
    \label{fig:residualsHPE}
\end{figure}

On the NCSU dataset the LMS estimator is slightly better than the XYZ, 
but the reverse is true for the Krinov dataset.
There is little evidence to choose one over the other.

\section{The Unbounded Case}\label{sec:unbounded}

In this section we briefly examine, without complete proofs, what happens when the cube $Q^\infty$
is replaced by the \emph{non-negative orthant} 
$L_+^1:= \{ f \in L^1[0,1] : f(x) \ge 0 ~a.e.\}$.
The inclusion and mapping situation is
\[
L_+^1 ~~ \subset L^1[0,1]  ~~  {\overset{\Lambda}\twoheadrightarrow}  ~~ \mathbb{R}^m
\]
For $y \in \mathbb{R}^m$ we are now interested in the centroid of $\Lambda^{-1}(y) \cap L_+^1$.
The first difficulty is that the set may be unbounded;
moreover $V_n \cap \Lambda^{-1}(y) \cap L_+^1$ may be unbounded for all $n$ and so $\centroid_\mathapprox{P}(\Lambda^{-1}(y) \cap L_+^1)$ is undefined.
So in this Section we place the following extra conditions on $\mathbf{w}(x) := (w_1(x), \ldots, w_m(x))$
\begin{equation}\label{eqn:conditions}
w_i(x) \ge 0 \text{ for all } i    \hspace{20pt}\text{and}\hspace{20pt}  \sum_{i=1}^m w_i(x) \ge \epsilon \hspace{20pt} \text{ for some }  \epsilon > 0
\end{equation}
As usual, these inequalities should be taken almost everywhere in $x$.
They imply that $\Lambda_{\mathbf{w}}^{-1}(y) \cap L_+^1$ is bounded.

Define $Y_\mathbf{w} := \Lambda_{\mathbf{w}}(L_+^1)$.
In the bounded case $\Lambda_{\mathbf{w}}(Q^\infty)$ was a convex body, 
but now $Y_\mathbf{w}$ is a convex cone contained in the 
\emph{non-negative orthant} $\mathbb{R}_+^m := \{ (t_1,\ldots,t_m) : all ~ t_i \ge 0 \}$.
Also define the \emph{non-positive orthant} $\mathbb{R}_-^m := \{ (t_1,\ldots,t_m) : all ~ t_i \le 0 \}$.

For the cube $Q^\infty$ we used $\rho := \mathbf{1}_{[0,1]}$,
but now we want to replace it by $\rho := \mathbf{1}_{[0,\infty)}$.
The Laplace Transform $P(s)$ is now
\[
P(s) := \mathcal{B}(\rho)(s) := \int_0^\infty e^{-ts}  \, dt ~=~ 1/s  \hspace{20pt} \text{for } \Re(s) < 0
\]
Note that $P(s)$ is only defined in the left half-plane.
We also get
\begin{align*}
K(t) &:= \log(P(-t)) = -\log(-t) \hspace{20pt} \sigma(t) := K'(t) = -1/t  \hspace{20pt}  \text{for } t<0 \\
K_r(s) &= -\log\abs{s   }\text{ for } \Re(s)<0
\end{align*}
Since $\rho$ does not have compact support,
the key results of Section \ref{sec:PK} -
Theorem \ref{thm:Ps}, Lemma \ref{lemma:Kr}, and Corollary \ref{cor:Kr} -
cannot be applied here.
However, because the form of these functions is so simple, 
the conclusions can be quickly checked directly,
with appropriate restrictions to the left half-plane.

We turn now to the diffeomorphism $G_\sigma$ in Theorem \ref{thm:diffeo}.
For $(t_1,\ldots,t_m) \in \mathring{\mathbb{R}}_-^m$ (all $t_i <0$), 
it follows from \ref{eqn:conditions} that $t_1 w_1(x) + \ldots + t_m w_m(x) < 0$.
We now define $G_\sigma : \mathring{\mathbb{R}}_-^m \to Y_\mathbf{w}$ by the same formula \ref{eqn:Gsigma}
\begin{equation*}
G_\sigma(t_1,\ldots,t_m) := \Lambda_{\mathbf{w}}( \sigma(t_1 w_1 +\ldots+ t_m w_m) )  \hspace{15pt} \text{where } \sigma(t) = -1/t
\end{equation*}
Since $\sigma'(t) > 0$, the derivative $\mathbf{D}G_\sigma$ is positive-definite as in \ref{lemma:PD},
and so $G_\sigma$ is injective as in the bounded case.
However, the argument for surjectivity breaks down because the support function $\psi_\body$ is unbounded,
and Lemma \ref{lemma:uniform} is false.
In fact, $G_\sigma$ is \emph{not} surjective in general;
see the example in Figure \ref{fig:unbounded1} part $(c)$.
This means that the saddlepoint equation may have no solution.

But if it \emph{does} have a solution, then the complex inversion has a saddlepoint
and the techniques for proving the Main Theorem go through as before.
The functions $K(s)$ and $K_r(s)$ have the required bounds.
By construction, all components of the saddlepoint 
are in the region of convergence of $P(s)$.
The unbounded version of the Main Theorem can be stated 
\begin{maintheorem} (\textbf{unbounded version})
Let $\mathbf{w} : [0,1] \to \mathbb{R}^m$ be a step function
with $w_1, \ldots, w_m$ linearly independent and satisfying conditions \ref{eqn:conditions}.
Let $\Lambda_\mathbf{w} : L^1[0,1] \twoheadrightarrow \mathbb{R}^m$ be the induced surjective linear operator.
If $y_0 \in \mathbb{R}_+^m$ and the saddlepoint equation
\begin{equation} 
\int_0^1 \sigma( \langle \tau, \mathbf{w}(x) \rangle ) \mathbf{w}(x) \,dx  =  y_0    \hspace{30pt} \tau \in \mathring{\mathbb{R}}_-^m  \hspace{30pt} \sigma(t) = -1/t
\end{equation}
has a solution $\tau_0 \in \mathring{\mathbb{R}}_-^m$,
then $\tau_0$ is unique and 
\begin{equation*} 
\centroid_{\mathapprox{P}} \left( L_+^1  \cap \Lambda_\mathbf{w}^{-1}(y_0) \right) ~~=~~ \sigma( \langle \tau_0, \mathbf{w} \rangle)
\end{equation*}
where $\mathapprox{P}$ is the  profile-gauge directed filtration of $L^1[0,1]$.
\end{maintheorem}

\medskip
For application in colorimetry we interpret a function  $f \in L^1_+[a,b]$ as a spectral power distribution of a \emph{light source}.
Its $L^1$-norm is the total power of the light source.
With the CIE 1931 standard observer responsivities  $w_1{=}\bar{x}$, $w_2{=}\bar{y}$, and $w_3{=}\bar{z}$ as in the previous section,
$\Lambda_\mathbf{w}(f)$ is the 3-vector of tristimulus values $XYZ$ of $f$.
Given an $XYZ_0 \in \mathbb{R}_+^3$, the above centroid is interpreted as a good estimate for a light source with that $XYZ_0$.

\begin{figure}[!ht]
    \centering
    \subfloat{\includegraphics[width=1.0\linewidth]{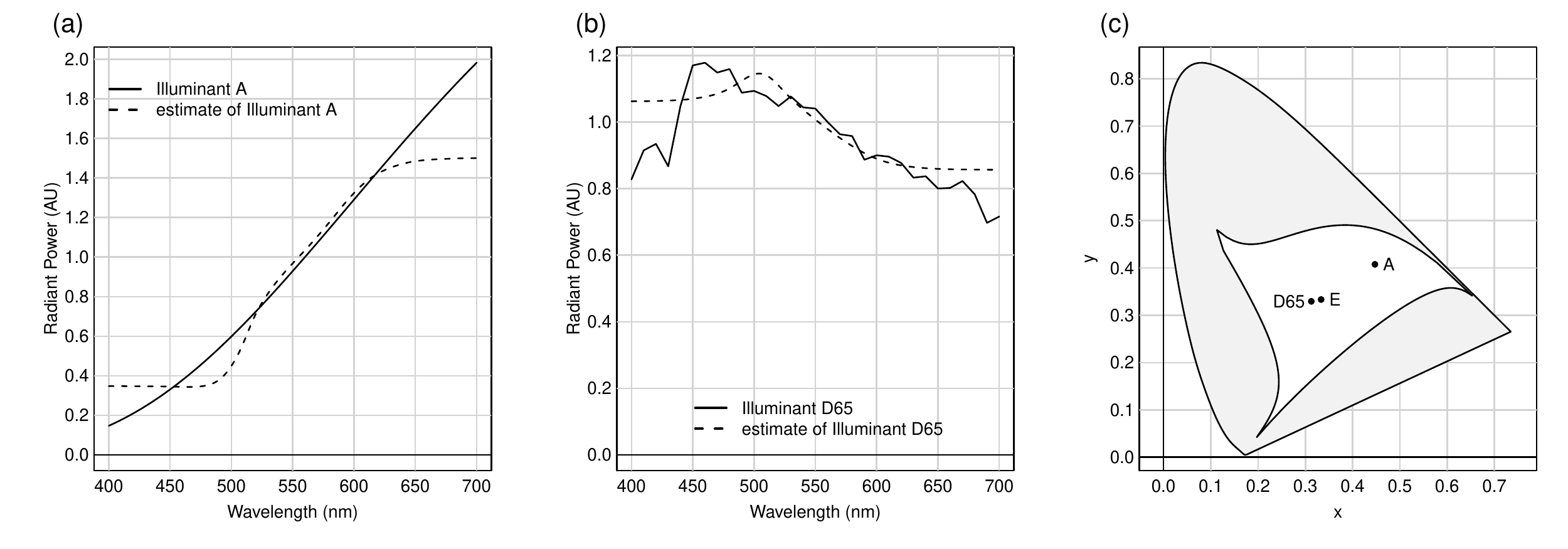}}    
    \caption{\emph{(a)} the spectrum of standard Illuminant A and an estimate for it, using the centroid method;
                        these two spectra are metameric.
             \emph{(b)} the same as \emph{(a)} but with Illuminant D65.
             \emph{(c)} see the text 
             }
    \label{fig:unbounded1}
\end{figure}

Figure \ref{fig:unbounded1} shows the estimates of two standard spectra from their $XYZ$ values.
The equalized responsivities $\widehat{x}$, $\widehat{y}$, and $\widehat{z}$ from Figure \ref{fig:reparam} are used here,
so the estimator is neutral-exact (i.e. the estimate of Illuminante E is exact).
These estimates tend to flatten at the ends.
In part $(c)$ the familiar inverted-U is the section of the convex cone $Y_{\mathbf{w}}  := \Lambda_{\mathbf{w}}(L_+^1)$ 
by the plane $X{+}Y{+}Z{=}1$;
it illustrates the chromaticities of all possible light sources.
The triskelion-shaped region is the section of the (non-convex) cone
$G_\sigma( \mathring{\mathbb{R}}_-^m )$ by the same plane.
Chromaticities outside this region cannot be spectrally estimated by the centroid method.

\section{Discussion}\label{sec:discussion}
To summarize, the centroid we have defined is relative to a specific directed filtration $\mathapprox{P}$,
and is a unique linear combination of the responsivities $w_1, \ldots, w_m$,
followed by a Langevin-function-based squashing function $\sigma(t)$.
We have proved this under the condition that the responsivities are step functions.
Some open questions are:

\begin{itemize}[label=\textbullet]
\item{
In the Main Theorem, can the requirement that 
$\mathbf{w} : [0,1] \to \mathbb{R}^m$ 
be a step function be removed?
From Corollary \ref{cor:diffeo2} we know that the saddlepoint equation has a solution $\tau_0$ for small perturbations of $\mathbf{w}$,
so the resulting function $\sigma( \langle \tau_0,\mathbf{w} \rangle)$ is the only possible candidate for the centroid.
But that does not mean that the net in Definition \ref{defn:centroid} converges to it.
A density argument might work to extend it to regulated functions, i.e. uniform limits of step functions, see (\cite{LangReal} p. 94).
Another approach might be to prove Corollary \ref{cor:LaplaceMethod}, about the \emph{ratio} of two Laplacian integrals,
with weaker conditions on $\phi_n(x)$ and $f(x)$.
In the almost trivial case where $m{=}1$ and the feasible region is the unbounded orthant from the previous section,
the function $P(s)=1/s$ is so simple that the answer is \emph{yes}.
}
\item{In the Main Theorem, can the profile-gauge directed filtration $\mathapprox{P}$
be replaced by a bigger one, perhaps even the collection of \emph{all} finite-dimensional subspaces of $L^1[0,1]$?
I suspect that for this collection $\mathapprox{F}$
the answer is \emph{no} - that the corresponding net does not converge in general.}
\end{itemize}

\appendix
\section{Proof of the Laplace Approximation} 
\label{sec:appA}

\medskip

\begin{lemma}\label{lemma:Wong1}
Let $\alpha := (\alpha_1,\ldots,\alpha_m)$ be a multi-index of non-negative integers,
and $\abs{\alpha} := \alpha_1 + \cdots + \alpha_m$.
Let $\mu := (\mu_1,\ldots,\mu_m)$ be a vector of positive numbers.
Let $B_{\delta}$ be the ball in $\mathbb{R}^m$ centered at $0$ with radius $\delta$.
Then
\begin{equation}\label{eqn:Wong1}
\int_{\mathbb{R}^m} \abs{x_1}^{\alpha_1} \cdots \abs{x_m}^{\alpha_m} \exp \left[ -\frac{n}{2} \sum_{j=1}^m \mu_j x_j^2 \right] \, dx ~=~ \mathcal{O} \left( n^{-(m+\abs{\alpha})/2} \right) ~~~(n \to \infty)
\end{equation}
\begin{equation}\label{eqn:Wong2}
\int_{\mathbb{R}^m \backslash B_{\delta}} \abs{x_1}^{\alpha_1} \cdots \abs{x_m}^{\alpha_m} \exp \left[ -\frac{n}{2} \sum_{j=1}^m \mu_j x_j^2 \right] \, dx ~=~ \mathcal{O}\left( e^{-\epsilon n} \right) ~~~~~(n \to \infty)
\end{equation}
and $\epsilon$ can be taken to be $\delta^2 \min(\mu)/2 > 0$, where $\min(\mu) := \min_j(\mu_j)$.
\end{lemma}

\begin{proof}
In \ref{eqn:Wong1}, for $m{=}1$ the integral is
\begin{equation*}
\int_{\mathbb{R}} \abs{x}^{\alpha} \exp \left[ -\frac{n}{2} \mu x^2 \right] \, dx ~=~ 2 \int_0^\infty x^{\alpha} \exp \left[ -\frac{n}{2} \mu x^2 \right] \, dx
\end{equation*}
The substitution $y = n \mu x^2 / 2$ quickly leads to
\begin{equation*}
\int_{\mathbb{R}} \abs{x}^{\alpha} \exp \left[ -\frac{n}{2} \mu x^2 \right] \, dx ~=~ \Gamma \left( \frac{\alpha+1}{2} \right) \left( \frac{\mu}{2} \right)^{-(\alpha+1)/2} n^{-(\alpha+1)/2}
\end{equation*}
This takes care of $m{=}1$.
For $m\ge2$, the integral splits into a product of $m$ integrals over $\mathbb{R}$,
the exponents of $n$ add up to $-(m+\abs{\alpha})/2$,
and this shows \ref{eqn:Wong1}.

In \ref{eqn:Wong2}, the case $m{=}1$ is equivalent to:
\begin{equation}\label{eqn:Wong3}
\int_{\delta}^\infty x^{\alpha} \exp \left[ -\frac{n}{2} \mu x^2 \right] \, dx ~=~ \mathcal{O}( e^{-\epsilon n} ) ~~~(n \to \infty)
\end{equation}
Since $0 \le (x-\delta)^2$, it suffices to show \ref{eqn:Wong3} with $-x^2$ replaced  $-2\delta x + \delta^2$.
When $\alpha{=}0$ this integral is elementary and \ref{eqn:Wong3} follows easily, with $\epsilon = \delta^2 \mu/2$.
For $\alpha \ge 1$ it follows by integration by parts and induction on $\alpha$.  
This takes care of $m{=}1$.
Define  $r := (x_1^2 + \cdots + x_m^2)^{1/2}$.
We have $\abs{x_j} \le r$ and $-\mu_j \le -\min(\mu)$ for all $j$, and so
\begin{equation*}
\abs{x_1}^{\alpha_1} \cdots \abs{x_m}^{\alpha_m} \exp \left[ -\frac{n}{2} \sum_{j=1}^m \mu_j x_j^2 \right] ~~\le~~ r^{\abs{\alpha}} \exp \left[ -\frac{n}{2} \min(\mu) r^2 \right]
\end{equation*}
Now use spherical coordinates and change of variables to integrate the right side over $\mathbb{R}^m \backslash B_{\delta}$
and conclude that the integral in \ref{eqn:Wong3} is bounded above by
\begin{equation*}
\Vol(S^{m-1}) \int_\delta^\infty r^{\abs{\alpha}+m-1} \exp \left[ -\frac{n}{2} \min(\mu) r^2 \right] \, dr
\end{equation*}
where $S^{m-1}$ is the unit sphere in $\mathbb{R}^m$.
Now \ref{eqn:Wong2} follows from \ref{eqn:Wong3}.
\end{proof}

\bigskip

Now we are ready for
\begin{proof}\label{proof:LaplaceMethod.m} of Theorem \ref{thm:LaplaceMethod.m},
based on Wong \cite{wong2001}, Theorem 3, p. 495.
For easier reference, the conclusion of the Laplace Approximation theorem is repeated here:
\begin{equation}\label{eqn:repeatLA}
\int_U  \phi_n(x)f(x)^n \, dx    ~ \sim ~  \abs{\det(f''(x_0))}^{-\frac{1}{2}}   { \left( \frac{2\pi}{n} \right) } ^{ \frac{m}{2} }    \phi_n(x_0) {f(x_0)}^{n+\frac{m}{2}}  \hspace{20pt} (n \to \infty)
\end{equation}
For the conditions $(a),\ldots,(e)$ on $\phi_n(x)$ and $f(x)$ see \ref{thm:LaplaceMethod.m}.

WLOG we can assume $f(x_0){=}1$.
Otherwise replace $f(x)$ by $f(x)/f(x_0)$.
Conditions $(a)$, $(b)$, and $(e)$ are \emph{still} true, with the new $f(x)$.
In the theorem's conclusion, if $f(x)$ is divided by any $a>0$,
the left side and right sides of the equivalence are both divided by $a^n$.

Split the integral
\begin{align}\label{eqn:split}
J(n) &:= \int_U \phi_n(x) f(x)^n \, dx  ~=~  \int_{B_0} \phi_n(x) e^{n h(x)} \, dx  ~+~ \int_{U \backslash B_0} \phi_n(x) f(x)^n \, dx   \nonumber  \\
     &:= J_1(n) ~+~ J_2(n)
\end{align}
Note that $J_1(n)$ is real, but $J_2(n)$ is complex in general.
Consider $J_1(n)$ first. 
If necessary, shrink $B_0$ so the Morse Lemma, \cite{Hirsch1976} p. 145, applies to $h(x)$ at the critical point $x_0$.
There is a neighborhood $\Omega$ of $0$ and a diffeomorphism $g : \Omega \to B_0$ with $g(0)=x_0$ and $\det g'(0) = 1$.
Define $\Phi_n : \Omega \to \mathbb{R}$ by
\begin{align}\label{eqn:Phi_n}
\Phi_n(y) &:= \phi_n(g(y)) \det g'(y)  \nonumber \\
\text{then} ~~~~ J_1(n) ~&=~ \int_\Omega \Phi_n(y) \exp \left[ -\frac{n}{2} \sum_{j=1}^m \mu_j y_j^2 \right] \, dy  ~~~~\text{by change of variables}
\end{align}
Since $f''(x_0)$ is negative-definite, all $\mu_j {>} 0$.
The Morse Lemma loses 2 degrees of differentiability.
So since $f(x)$ is $C^4$, $g(y)$ is $C^2$, and $g'(y)$ is $C^1$.
By $(d)$ $\phi_n(x)$ is $C^1$ and so $\Phi_n(y)$ is $C^1$.
Moreover, by condition $(d)$, $\norm{\Phi_n}_{C^1} \le K_1 \abs{\Phi_n(0)}$ for some $K_1$, where $\norm{\cdot}_{C^1}$ is the $C^1$-norm on $\Omega$.
By the Mean Value Theorem
\begin{equation}\label{eqn:MVT}
\Phi_n(y) =  \Phi_n(0) ~+~ \sum_{\abs{\alpha}{=}1} D^\alpha \Phi_n(\xi) y^\alpha  ~~~~~\text{ for some } \xi \in [0,y]
\end{equation}
Abbreviate $Q(y) := \sum_{j=1}^m \mu_j y_j^2$, and substitute \ref{eqn:MVT} into \ref{eqn:Phi_n} to get
\begin{equation*}
J_1(n) ~=~ \Phi_n(0) \int_\Omega \exp \left[ -\frac{n}{2} Q(y) \right] \, dy  ~~+~~ \sum_{\abs{\alpha}{=}1}  \int_\Omega  D^\alpha \Phi_n(\xi) y^\alpha \exp \left[ -\frac{n}{2} Q(y) \right] \, dy
\end{equation*}
Abbreviate these two terms by $J_3(n) + J_4(n)$ and consider $J_3(n)$ first.
\begin{equation*}
J_3(n) = \Phi_n(0) \int_{\mathbb{R}^m} \exp \left[ -\frac{n}{2} Q(y) \right] \, dy   ~-~   \Phi_n(0) \int_{\mathbb{R}^m \backslash \Omega} \exp \left[ -\frac{n}{2} Q(y) \right] \, dy
\end{equation*}
The first integral can be evaluated exactly, and the absolute value of second integral is bounded by \ref{eqn:Wong2}.
This gives
\begin{equation}\label{eqn:J_3}
J_3(n) = \phi_n(x_0) (\mu_1 \cdots \mu_m)^{-1/2} \left( \frac{2\pi}{n} \right) ^{m/2}  ~~+~~ \abs{\phi_n(x_0)} \mathcal{O} \left( e^{-\epsilon n} \right) \hspace{20pt} (n \to \infty) 
\end{equation}
Since $\mu_1 \cdots \mu_m = \abs{\det h''(x_0)}$, the first term is the asymptotic limit \ref{eqn:repeatLA} we want,
and now we show that the other terms tend to 0.

For $J_4(n)$, we have $\abs{ D^\alpha \Phi_n(\xi) } \le \norm{\Phi_n}_{C^1} \le K_1 \abs{\phi_n(x_0)}$, for some $K_1$ by $(d)$
\begin{align}\label{eqn:J_4}
\abs{J_4(n)} &\le K_1 \abs{\phi_n(x_0)} \sum_{\abs{\alpha}{=}1} \int_\Omega \abs{y}^\alpha \exp \left[ -\frac{n}{2} Q(y) \right] \, dy    \nonumber \\
&\le K_1 \abs{\phi_n(x_0)} \sum_{\abs{\alpha}{=}1} \int_{\mathbb{R}^m} \abs{y}^\alpha \exp \left[ -\frac{n}{2} Q(y) \right] \, dy    \nonumber  \\
&= \abs{\phi_n(x_0)}\mathcal{O} \left( n^{-(m+1)/2} \right)  \hspace{20pt} (n \to \infty) ~~~~~\text{ by } \ref{eqn:Wong1}
\end{align}

Now return to $J_2(n)$ in \ref{eqn:split}
\begin{align*}
\abs{J_2(n)} &\le \int_{U \backslash B_0} \abs{ \phi_n(x)  f(x)^{n} } \, dx ~=~ \int_{U \backslash B_0} \abs{ \phi_n(x) f(x)^{n-n_0} f(x)^{n_0} } \, dx  \nonumber  \\
     &\le M^{n-n_0}  \int_{U \backslash B_0} \abs{ \phi_n(x) f(x)^{n_0} } \, dx ~~~~\text{ since } \abs{f(x)} \le M < f(x_0) = 1 ~~~\text{ by } (b) \nonumber \\
     &\le  M^{n-n_0} K_2 n^k \abs{\phi_n(x_0)} \hspace{30pt} \text{for sufficiently large } n ~~\text{ by } (e)  \nonumber   \\
     &=  M^{-n_0} K_2 n^k M^n \abs{\phi_n(x_0)} \\
     &= \abs{\phi_n(x_0)} \mathcal{O}( n^k e^{-cn} )  \hspace{30pt}  (n \to \infty)   ~~~~\text{ where } c = -\log(M) > 0
\end{align*}
Combine \ref{eqn:split}, \ref{eqn:Phi_n}, \ref{eqn:J_3},  \ref{eqn:J_4}, and the previous estimate and divide by $\phi_n(x_0)$
\begin{equation*}
\frac{1}{\phi_n(x_0)} J(n) = d_0  n^{-\frac{m}{2}} ~+~  \mathcal{O} \left( e^{-\epsilon n} \right) +  \mathcal{O} \left( n^{-\frac{m+1}{2}} \right) + \mathcal{O} ( n^k e^{-cn} ) \hspace{20pt} (n \to \infty) 
\end{equation*}
where $d_0 = (\mu_1 \cdots \mu_m)^{-1/2} (2\pi) ^{m/2}$.
Multiply both sides by $n^{\frac{m}{2}}$ to get
\begin{equation*}
\frac{n^{m/2}}{\phi_n(x_0)} J(n) ~=~ d_0  ~+~   \mathcal{O} \left( n^{m/2} e^{-\epsilon n} \right) +   \mathcal{O} \left( n^{-\frac{1}{2}} \right) +   \mathcal{O} ( n^{k+m/2} e^{-cn} )  \hspace{20pt} (n \to \infty) 
\end{equation*}
The three error terms tend to 0 and so
\begin{align*}
\frac{n^{m/2}}{\phi_n(x_0)} J(n) ~~&\sim~~ d_0   \hspace{30pt} (n \to \infty)  \\
J(n) ~~&\sim~~ \frac{d_0}{n^{m/2}}\phi_n(x_0) ~~=~~  (\mu_1 \cdots \mu_m)^{-1/2} \left( \frac{2\pi}{n} \right) ^ \frac{m}{2} \phi_n(x_0)  \hspace{20pt} (n \to \infty) 
\end{align*}
and we are done.
\end{proof}

\medskip

\section*{Software and Figures}
All calculations and figures were made using the software 
\textbf{R: A Language and Environment for Statistical Computing} \cite{Rlanguage},
and with \textbf{R} packages 
\textbf{colorSpec} \cite{colorSpec} and 
\textbf{rootSolve} \cite{rootSolve}.

\section*{Acknowledgments}
I would like to thank Nik Weaver for the key idea that made this work possible, 
and for the term \emph{directed filtration}.

\bibliographystyle{plain}
\bibliography{references}
\end{document}